\documentclass[a4paper]{amsart}
\usepackage{amssymb}
\usepackage{amsmath}
\usepackage{amsthm}
\usepackage[dvipsnames]{xcolor}
\usepackage{braket,comment,soul}
\usepackage{hyperref}
\usepackage{paralist}
\usepackage{pdfpages,faktor}
\usepackage{graphicx,mathabx,bbm}
\usepackage{caption}
\usepackage{subcaption}
\usepackage{mathrsfs} 
\usepackage{tikz-cd} 
\usepackage{bm}
\usepackage{enumitem}
\usepackage{mathtools}
\usepackage{pgfplots}
\usepackage{import}
\usepackage{xifthen}
\usepackage{transparent}

\pgfplotsset{compat=1.18} 
\numberwithin{equation}{section}
\theoremstyle{plain}
\newtheorem{theorem}{Theorem}[section]

\newtheorem{prop}[theorem]{Proposition}
\newtheorem{lem}[theorem]{Lemma}
\newtheorem{cor}[theorem]{Corollary}

\theoremstyle{definition}
\newtheorem{example}[theorem]{Example}
\newtheorem{question}[theorem]{Question}
\newtheorem{task}[theorem]{Task}
\newtheorem{remark}[theorem]{Remark}
\newtheorem*{question*}{Question}

\theoremstyle{definition}
\newtheorem{definition}[theorem]{Definition}

\theoremstyle{definition}
\newtheorem{thmx}{Theorem}

\newcommand{\C}{\mathbb{C}}

\newcommand{\Z}{\mathbb{Z}}

\newcommand{\D}{\mathbb{D}}

\newcommand{\cF}{\mathcal{F}}

\newcommand{\cR}{\mathcal{R}}
\newcommand{\cG}{\mathcal{G}}

\newcommand{\cK}{\mathcal{K}}
\newcommand{\cI}{\mathcal{I}}

\newcommand{\cD}{\mathcal{D}}

\newcommand{\cV}{\mathcal{V}}
\newcommand{\cT}{\mathcal{T}}
\newcommand{\cU}{\mathcal{U}}
\newcommand{\dT}{\widetilde{\mathcal{T}}}
\newcommand{\cW}{\mathcal{W}}
\newcommand{\cQ}{\mathcal{Q}}
\renewcommand{\phi}{\varphi}

\DeclareMathOperator{\Int}{Int}

\usepackage{mathtools}

\DeclarePairedDelimiterX{\inp}[2]{\langle}{\rangle}{#1, #2}

\date{\today}

\begin{document}

\title[Correspondences on hyperelliptic surfaces, and Hurwitz spaces]{Correspondences on hyperelliptic surfaces, combination theorems, and Hurwitz spaces}

\author[S. Mukherjee]{Sabyasachi Mukherjee}
\address{School of Mathematics, Tata Institute of Fundamental Research, 1 Homi Bhabha Road, Mumbai 400005, India}
\email{sabya@math.tifr.res.in, mukherjee.sabya86@gmail.com}
\thanks{Both authors were supported by the Department of Atomic Energy, Government of India, under Project Identification No. RTI 4014, an endowment of the Infosys Foundation, and ANRF research grants MTR/2022/000248 and ANRF/ARGM/2025/000089/MTR} 

\author{S. Viswanathan}
\address{School of Mathematics, Tata Institute of Fundamental Research, 1 Homi Bhabha Road, Mumbai 400005, India}
\email{viswa@math.tifr.res.in, drsvkrishna4@gmail.com}

\begin{abstract}
We construct a general class of holomorphic correspondences of equal bi-degree on compact Riemann surfaces of arbitrary genus exhibiting rich dynamical behavior. These correspondences combine finitely many Fuchsian genus zero orbifold groups with Blaschke products. As an intermediate step, we first construct analytic combinations of these objects as partially defined maps on the Riemann sphere. We then give an algebraic characterization of these analytic combinations in terms of hyperelliptic involutions and meromorphic maps on compact Riemann surfaces. These involutions and meromorphic maps, in turn, give rise to the desired correspondences. The moduli space of such correspondences can be identified with a product of Teichm\"uller spaces and Blaschke spaces. The explicit description of the correspondences then allows us to construct a dynamically natural injection of this product space into appropriate Hurwitz spaces.
\end{abstract}

\maketitle

\setcounter{tocdepth}{1}
\tableofcontents

\section{Introduction}\label{intro_sec}

It is well-known that \emph{single-valued} holomorphic self-maps on higher genus compact Riemann surfaces do not have interesting dynamics; in fact, they are necessarily finite-order automorphisms and preserve the hyperbolic metric. This is also the prototypical situation on generic algebraic varieties in higher dimension.

On the other hand, \emph{multi-valued} holomorphic self-maps on higher genus surfaces (or more generally, on K{\"a}hler manifolds) are known to exhibit rich chaotic dynamics (cf. \cite{dinh2006distribution}). Important examples of multi-valued holomorphic maps arise from algebraic correspondences. These are Zariski-closed subsets $\mathfrak{C}\subset Z\times W$, where $Z$ and $W$ are algebraic varieties. The correspondence $\mathfrak{C}$ can be regarded as a multi-valued map $f: Z\rightarrow W$, for which there exist (finite and surjective) coordinate projections $\pi_1: \mathfrak{C}\rightarrow Z$, $\pi_2: \mathfrak{C}\rightarrow W$, allowing $f$ to be expressed as the composition $f = \pi_2\circ \pi_1^{-1}$ (cf. \cite{Ful98,dinh2006distribution, bartholdi2024correspondences}). A correspondence is said to have bi-degree $d_1\!:\! d_2$ if the degrees of the projections $\pi_2,\pi_1$ are $d_1,d_2$, respectively. Iteration of the multi-valued map $f$ gives us a dynamical system.
\begin{equation*}
\begin{tikzcd}
	& \mathfrak{C} \\
	\\
	Z && W
	\arrow["{\pi_1}"', from=1-2, to=3-1]
	\arrow["{\pi_2}", from=1-2, to=3-3]
	\arrow["f"', from=3-1, to=3-3]
\end{tikzcd}
\end{equation*} 

 Algebraic correspondences were first studied as dynamical objects in a seminal paper by Bullett and Penrose \cite{bullett1994mating}. Since then, a body of work spanning several decades has emerged (cf. \cite{lomonaco2025algebraiccorrespondencesschwarzreflections}). However, almost all examples of such dynamically well-understood correspondences are either defined on (subsets of) Riemann spheres, or have unequal bi-degree (in which case, the asymmetry between the forward and backward branches contribute to chaotic behavior).
\medskip
 
\noindent The principal results of the current paper are the following.
 
\begin{itemize}[leftmargin=8mm]
    \item We produce a general class of holomorphic correspondences of equal bi-degree on higher genus Riemann surfaces with fully describable complex-analytic dynamics. This is the first such class of correspondences to the best of our knowledge. The construction of these correspondences is motivated by the Sullivan dictionary, in the sense that they arise as combinations of certain Fuchsian groups and rational maps (see Theorem~\ref{mating_thm_intro}). 
    \item We show that quasiconformal deformation spaces of the said correspondences (arising as matings) are given by products of Teichm{\"u}ller spaces and Blaschke product spaces. On the other hand, appropriate Hurwitz spaces serve as ambient algebraic moduli spaces for these correspondences. A detailed study of the dynamical structure of the correspondences allows us to furnish embeddings of topological/analytic objects like Teichm{\"u}ller spaces and Blaschke product spaces into algebro-geometric objects like Hurwitz spaces (see Theorem~\ref{hurwitz_thm_intro}).
\end{itemize}

\subsection{Dynamical unification of Fuchsian groups and Blaschke products}
It was anticipated by Fatou that a comprehensive understanding of the dynamics of algebraic correspondences could provide a unifying framework for the dynamical theories of Kleinian groups and rational maps on the Riemann sphere \cite{Fat29}. However, efforts to realize this vision in complete generality have encountered significant technical obstacles (see \cite{bullett2001regular}). Meanwhile, several explicit classes of correspondences realizing combinations/matings of rational maps and Kleinian groups have been constructed and their parameter spaces have been investigated in the last few decades (see \cite{bullett1994mating, bullett2000mating,BF05,BH07,bullett2016mating, bullett2017dynamics, bullett2020mating, mj2023matings, bullett2024mating, luo2024general, de2024equidistribution, mj2024simultaneous, luo2025} for the mating phenomena in the holomorphic world, and \cite{LLMM21,LLMM23,LMmirror26,LMM24,luo2024general} for parallel results in the antiholomorphic world). The correspondences that feature in these are multi-valued self-maps of the Riemann sphere, possibly with nodal singularities.     

We construct correspondences $\mathfrak{C}\subset\Sigma\times\Sigma$, where $\Sigma$ is a hyperelliptic Riemann surface of genus $g$, that naturally arise as matings of
\begin{itemize}[leftmargin=6mm]
    \item a subhyperbolic rational map $R$,
    \item hyperbolic Blaschke products $\{B_j\}_{j=1}^r$, and
    \item Fuchsian groups $\{G_i\}_{i=1}^l$, with each quotient $\D/G_i$ belonging to a specific class $\mathcal{F}$ of genus zero orbifolds (see Section~\ref{mateable_maps_subsec} for a precise definition of~$\mathcal{F}$).
\end{itemize}
The construction provides a tractable approach to fully understand the dynamics of $\mathfrak{C}$. Further, it extends the classes of correspondences arising as matings in the literature. We now state one of our main theorems.

\begin{thmx}\label{mating_thm_intro}
Given hyperbolic Blaschke products $\{B_j\}_{j=1}^r$, and Fuchsian groups $\{G_i\}_{i=1}^l$, with $\D/G_i\in \mathcal{F}$, there exist a compact Riemann surface $\Sigma$ equipped with a hyperelliptic involution $\eta$, and a meromorphic map $\cR:\Sigma\rightarrow\widehat{\C}$, such that the algebraic correspondence $\mathfrak{C}\in \Sigma\times\Sigma$, defined by,  
\begin{equation}\label{intro_corr_eq}
    (z, w) \in \mathfrak{C} \iff \cR(w) = \cR(\eta(z)),\ w \neq \eta(z),
\end{equation}
combines the dynamics of $\{G_i\}_{i=1}^l$ and $\{B_j\}_{j=1}^r$ in the following sense. 

\noindent There exist $\eta-$symmetric subsets $\{\dT_{i}\}_{i=1}^l$ and $\{\widetilde{\cV}_j\}_{j=1}^r$ of $\Sigma$, such that the following hold.
\begin{enumerate}[leftmargin=8mm]
    \item On $\dT_i$, the action of the correspondence $\mathfrak{C}$ is orbit-equivalent to the action of the group $G_i$ on $\D$. Further, the orbifold $\dT_i/\mathfrak{C}$ is isomorphic to $\D/G_i$.
    \item $\widetilde{\cV}_j$ has two components, and $\mathfrak{C}$ has a forward branch carrying one of the components of $\widetilde{\cV}_j$ onto itself such that this branch is conformally conjugate to the Blaschke product $B_j$.
\end{enumerate}
\end{thmx}
\noindent For a precise statement, see Theorem~\ref{corr_recover_thm}, where the role of $G_i$ is played by the groups $\widehat{\Gamma}_i$.

\subsection{Parametric unification of Teichm{\"u}ller spaces and Blaschke product spaces}
The parallels between the dynamical theories of rational maps and Kleinian groups were systematically formalized into a dictionary by Sullivan \cite{sullivan1985quasiconformal} (see also \cite{mcmullen1995classification, mcmullen1996renormalization,LM97,mcmullen1998quasiconformal}). In this context, we intend to highlight the following two theorems.
\begin{itemize}[leftmargin=6mm]
    \item Let $G$ be a finitely generated, geometrically finite Kleinian group with a connected limit set, and let $\mathfrak{T}(G)$ denote its quasiconformal deformation space. Then we have the following biholomorphic equivalence:
$$
\mathfrak{T}(G) \cong \text{Teich}(S_1)\times\text{Teich}(S_2)\times\cdots\times\text{Teich}(S_l),
$$
for appropriate finite type Riemann surfaces $S_i$, $i\in \{1, 2, \cdots, l\}$ (see \cite[Theorem 5.1.3]{marden2016hyperbolic}, \cite{ahlfors1964finitely}).

\item Let $\mathcal{H}$ denote a hyperbolic component in the parameter space of degree $d$ (marked) rational maps with connected Julia set. Then
$$
\mathcal{H} \cong \mathcal{B}_{d_1}\times\mathcal{B}_{d_2}\times\cdot\times\mathcal{B}_{d_r},
$$
where $\mathcal{B}_{d_j}$ is the space of normalized degree $d_j$ Blaschke products with an attracting fixed point in $\D$, and $d_j\geq 2$, $j\in \{1, 2, \cdots, r\}$, are suitable integers (see \cite{milnor2012hyperbolic} for a precise statement).
\end{itemize}

The above two statements can be interpreted as a local-global principle in dynamics. That is, global deformations of Kleinian groups and rational maps can, in essence, be re-engineered by performing local deformations on certain fundamental components. The parameter space of the correspondences constructed in Theorem~\ref{mating_thm_intro} also admits such a local-global principle. More precisely, for a correspondence $\mathfrak{C}$ (as in Theorem~\ref{mating_thm_intro}), denote its deformation space by $\mathfrak{T(\mathfrak{C})}$; we show that,
$$
\mathfrak{T(\mathfrak{C})} \cong \left(\text{Teich}(\D/G_1)\times\cdots\times\text{Teich}(\D/G_l)\right)\times\left(\mathcal{B}_{d_1}\times\cdots\times\mathcal{B}_{d_r}\right),
$$
where $\mathfrak{C}$ combines the dynamics of the Fuchsian groups $\{G_i\}_{i=1}^l$, and the hyperbolic Blaschke products $\{B_j\}_{j=1}^r$ (see Section~\ref{Hurwitz_sec}). This bijection is preceded by the identification of the moduli space $\mathfrak{T(\mathfrak{C})}$ of marked correspondences with a suitable sub-collection in the Hurwitz space of degree~$d$ meromorphic maps on genus~$g$ hyperelliptic Riemann surfaces with~$q$ ordered marked points, denoted by $\mathcal{H}_{g, d, q}^{\text{hyp}}$ (see Section~\ref{equivalence_of_corr_subsec}). This brings us to our other main theorem, which is proved in Section~\ref{Hurwitz_sec}.
\begin{thmx}\label{hurwitz_thm_intro}
    Given a collection of Fuchsian groups $\{G_i\}_{i=1}^l$ with $\D/G_i\in\cF$, and hyperbolic Blaschke products $\{B_j\}_{j=1}^r$, there exists an injective map $\mathfrak{B}$ from 
    $$
    \text{Teich}(\D/G_1)\times\cdots\times\text{Teich}(\D/G_l)\times\mathcal{B}_{d_1}\times\cdots\times\mathcal{B}_{d_r}
    $$ 
 to an appropriate Hurwitz space $\mathcal{H}_{g, d, q}^{\text{hyp}}$; such that for any $(\Sigma,\cR)$ in the image, the correspondence on $\Sigma$ defined by Equation~\eqref{intro_corr_eq} combines the groups $\{G_i\}_{i=1}^l$ and the Blaschke products $\{B_j\}_{j=1}^r$ as conformal dynamical systems. 
\end{thmx}

In summary, the correspondences that we construct not only combine Fuchsian groups and Blaschke products, but also unify their deformation spaces.

\subsection{Intuitive proofs of the main theorems} 

\subsubsection*{Sketch of Theorem~\ref{mating_thm_intro}} We begin by encoding the groups $\{G_i\}_{i=1}^l$ into maps $\{A_{G_i}\}_{i=1}^l$ that are variants of certain piecewise analytic, Markov, circle maps introduced by Bowen and Series \cite{bowen1979markov} (see Section~\ref{mateable_maps_subsec}). Then we look for an appropriate critically fixed polynomial $P$, and surgically replace the dynamics of $P$ on its bounded immediate basins with those of the maps $\{A_{G_i}\}_{i=1}^l$ and $\{B_j\}_{j=1}^r$. This gives us a partially defined meromorphic map $S:\overline{\cD}\ (\subset\widehat{\C})\rightarrow\widehat{\C}$, which we call a \emph{conformal mating} of the collections $\{A_{G_i}\}_{i=1}^l$ and $\{B_j\}_{j=1}^r$. The surgery above is a combination of quasiconformal and David surgery tools (cf. \cite{lyubich2020david}). We also remark that these conformal matings can be regarded as degenerate versions of \emph{rational-like maps} (cf. \cite{buf03,luo2024general}).

To understand why conformal matings are partially defined maps on $\widehat{\mathbb{C}}$, it is perhaps instructive to recall that the properly discontinuous group action of $G_i$ on $\D$ comes equipped with a fundamental domain. Said differently, an essential attribute of such a group action is that of \emph{non-recurrence} or escaping a domain. On the other hand, the interesting bits of dynamics of a rational map are trapped in invariant sets. Therefore, conformal matings of these two objects exhibit dual behavior (see Section~\ref{conf_mating_sec}).

Now, if the points in $\cD$ eventually map to points in a region where $S$ is undefined, the machinery of iteration breaks down. To resolve this issue, we ``seal” the hole by welding a second copy of the domain along the boundary via the map $S$. This is possible because the restriction $S|_{\partial\cD}$ turns out to be an orientation-reversing piecewise analytic map. The resulting object is a compact Riemann surface $\Sigma$ of genus $g$, referred to in this paper as a \emph{blender surface} (see Section~\ref{blender_surf_subsec}). We then lift the dynamics of $S$ to the surface $\Sigma$ via a carefully constructed meromorphic map $\cR:\Sigma\rightarrow\widehat{\C}$ (see Theorem~\ref{alg_weak_b_inv_thm}).

In this lifted setting, points no longer disappear into a hole, they hop to the second copy of the domain $\cD$. By judiciously analyzing this hopping behavior across the two copies of the domain $\cD$, one can recover everything about the maps $\{B_j\}_{j=1}^r$ and the groups $\{G_i\}_{i=1}^l$ (see Sections~\ref{group-like_subsubsec} and~\ref{map-like_subsubsec}). And this is despite making a seemingly serious compromise in suppressing the richness of the group $G_i$ by a single-valued map $A_{G_i}$. 

\subsubsection*{Sketch of Theorem~\ref{hurwitz_thm_intro}} A precursor to the existence of an injective map from the product of Teichm\"uller spaces and Blaschke product spaces into an appropriate Hurwitz space is the existence of a well-defined map. To that end, we first observe that for every element
$$
(G'_i,\cdots,G'_l, B_1,\cdots,B_r) \in  \text{Teich}(\D/G_1)\times\cdots\times\text{Teich}(\D/G_l)\times\mathcal{B}_{d_1}\times\cdots\times\mathcal{B}_{d_r},
$$
we have an associated blender surface $\Sigma$, and a meromorphic map $\cR:\Sigma\rightarrow\widehat{\C}$. The pair $(\Sigma,\cR)$ turns out to be unique up to biholomorphic change of coordinates, as shown in Section~\ref{standard_matings_subsec}. This fact gives us a well-defined map $\mathfrak{B}$ from the product of Teichm\"uller spaces and Blaschke products spaces to an appropriate Hurwitz space $\mathcal{H}_{g, d, q}^{\text{hyp}}$ (see Section~\ref{equivalence_of_corr_subsec}). We denote by $\mathscr{C}$ the equivalence classes of correspondences $\mathfrak{C}$, defined (as in Equation~\ref{intro_corr_eq}) by the tuples $(\Sigma, \cR)$ lying in the image of $\mathfrak{B}$. 

The injectivity of the map $\mathfrak{B}$ is established by explicitly constructing an inverse $\mathfrak{B}'$ on $\mathscr{C}$. This amounts to recovering the data of the groups $\{G'_i\}_{i=1}^l$ and the maps $\{B_j\}_{j=1}^r$ entirely from the data of the marked blender surface $\Sigma$ and the meromorphic map $\cR$. This recovery is facilitated by the fact that the limit sets $\Lambda(\mathfrak{C})$ of the correspondences $\mathfrak{C} \in \mathscr{C}$ are all quasiconformally equivalent (see Section~\ref{limit_sets_equiv_subsec}). Therefore, the topological shape of the limit set is fixed for all parameters. 

The complement of $\Lambda(\mathfrak{C})$ is essentially the collection of domains where the action of the correspondence realizes a combination of the dynamics of the groups $G'_i$ and maps $B_j$. Thus, it is just a question of identifying the domains $\{\dT_{i}\}_{i=1}^l$ and the domains $\{\widetilde{\cV}_j\}_{j=1}^r$. The action of the correspondence on these domains would determine both the groups and the maps uniquely. We keep track of that combinatorial data by introducing a marking schematic on our blender surfaces $\Sigma$ (see Section~\ref{marked_blender surface_subsec}). These steps in succession give us means to construct the map $\mathfrak{B}'$, which in turn establishes the injectivity of $\mathfrak{B}$.

\subsection{Notable features and novelties}

\subsubsection*{Algebraicity of conformal matings} The conformal mating $S:\overline{\cD} \longrightarrow \widehat{\C}$ is an analytic combination of algebraic objects, namely, rational maps and Fuchsian groups. By construction, $S$ is only partially defined on the Riemann sphere, admitting no analytic continuation to the whole sphere. Hence, a priori, the map $S$ could be transcendental. Since no general GAGA-type principle applies to such maps, the algebraicity of $S$ requires new techniques.

To that end, we introduce an abstract class of maps, which we call \emph{weak boundary involutions}; conformal matings are particular instances of these maps. A weak boundary involution is a partially defined meromorphic map on a domain $\cD\subset\widehat{\C}$, that extends continuously to the closure $\overline{\cD}$, and is orientation-reversing on the non-singular part of $\partial\cD$. These maps generalize the notion of boundary involutions introduced in \cite[\S 14]{luo2024general}; hence the name, weak boundary involution (see Definition~\ref{weak_b_inv_def}). Some of the techniques used to study boundary involutions also extend to weak boundary involutions; but the latter presents a substantially general setting for our purposes (see Remark~\ref{B_involution_rem}). 

Recall the welding construction of a blender surface. An essential feature enabling that construction is the boundary action of conformal matings. We gather such functional features of conformal matings in the definition of weak boundary involutions. Consequently, given a general weak boundary involution $S:\overline{\cD}\rightarrow\widehat{\C}$, we can and will construct a compact Riemann surface $\Sigma$ by welding two copies of $\cD$ along $\partial\cD$ via the map $S$. We then define an involution $\eta$ that takes one copy of $\cD$ to the other, and a meromorphic map $\cR:\Sigma\rightarrow\widehat{\C}$, which acts as the identity map on one copy of $\cD$ and as the map $S$ on the other copy (see Section~\ref{char_weak_b_inv_subsec}). The map $S$ can then be described explicitly in terms of the algebraic objects $\cR$ and $\eta$ (see Equation~\ref{alg_char_eq}). 
This observation (implicit in \cite{luo2024general}) motivated us to work with weak boundary involutions, as it streamlined our treatment of conformal matings.

Algebraicity of conformal matings is fundamental to the development of our theory, for it enables the construction of correspondences on the Riemann surface $\Sigma$, ensures hyperellipticity, and dictates the topology of $\Sigma$. 

\subsubsection*{Hyperellipticity of blender surfaces} 
Consider a weak B-involution $S:\overline{\cD}\to\widehat{\C}$. Instead of welding two copies of $\cD$ (as in the construction of $\Sigma$), one could start with a single copy of $\overline{\cD}$, and glue the boundary $\partial\cD$ with itself via the action of $S|_{\partial\cD}$, producing a surface denoted by $\widecheck{\Sigma}$. This construction naturally induces a ramified double cover $\cQ: \Sigma \longrightarrow \widecheck{\Sigma}$ (see Section~\ref{blender_surf_subsec}). 

For a general weak boundary involution $S: \overline{\cD} \to \widehat{\C}$, the resulting surface $\Sigma$ may be disconnected. The same is true for the auxiliary surface $\widecheck{\Sigma}$. We have compiled a gallery of examples in Section~\ref{examples_sec} to showcase various ways in which $\Sigma$ and $\widecheck{\Sigma}$ appear. 
In the proof of Theorem~\ref{mating_thm_intro}, critically fixed polynomials $P$ are chosen for the construction of conformal matings $S:\overline{\cD}\to\widehat{\C}$, partly to ensure that $\cD$ is connected (see Theorem~\ref{mating_surf_exists_prop}). This in turn forces both $\Sigma$ and $\widecheck{\Sigma}$ to be connected. 

Now, to establish hyperellipticity\footnote{For convenience of exposition we shall include the Riemann sphere and tori in the class of hyperelliptic Riemann surfaces.} of the blender surface $\Sigma$, it would suffice to show that the auxiliary surface $\widecheck{\Sigma}$ is the Riemann sphere. To see this, observe that the holes in $\widehat{\C} \setminus \cD$ can be interpreted as fundamental domains for Fuchsian genus zero orbifold groups. The construction of the auxiliary surface $\widecheck{\Sigma}$ amounts to zipping along the boundaries of these holes, thus producing a topological sphere. If, instead, the groups $G_i$ uniformized orbifolds of higher genus, the gluing along the boundaries of the holes would inevitably introduce handles to the surface $\widecheck{\Sigma}$.
Thus, the fact that the groups in our combination framework uniformize genus zero orbifolds, lies at the heart of the hyperellipticity of the blender surfaces $\Sigma$.   

\subsubsection*{Emergence of higher genus blender surfaces} As noted above, all previously known correspondences that arise as matings of groups and maps are defined on the Riemann sphere, possibly with nodal singularities. In contrast, the setting we consider here breaks this genus zero tradition. While it is possible to compute the genus explicitly via combinatorial arguments in specific cases; we do not pursue this route, as such calculations offer little conceptual insight. Instead, we provide a geometric/dynamical explanation for why higher genus surfaces arise in our construction (see Section~\ref{construct_blender_subsec} for a detailed treatment).

Let $\eta$ be a hyperelliptic involution on the surface $\Sigma$, where $\Sigma$ is as in Theorem~\ref{mating_thm_intro}. If $\Sigma$ has genus $g$, then the map $\eta$ has exactly $2g + 2$ fixed points. It turns out that the number of fixed points of $\eta$ is at least the number of fixed points of the conformal mating $S$ on the non-singular part of the boundary $\partial\cD$. Further, the fixed points of $S$ on the non-singular part of $\partial\cD$ correspond to the order $2$ orbifold points of $\D/G_i$. Therefore, if the number of such orbifold points in $\{\D/G_i\}_{i=1}^l$ is greater than $2$; we are bound to see handles on the surface $\Sigma$. In summary, the emergence of a higher genus is a feature that is governed by the action of $S$ on $\partial\cD$ (see Proposition~\ref{genus_prop} and Corollary~\ref{orbifold_genus_cor}). 

\subsubsection*{Explicit embeddings into Hurwitz spaces} It is worth glossing over a couple of specific consequences of Theorems~\ref{mating_thm_intro} and~\ref{hurwitz_thm_intro}, as they illuminate potential directions for future study (for precise questions, see Section~\ref{questions_subsec}). 

With the help of our mating framework, we can analytically combine the dynamics of 
\begin{itemize}[leftmargin=6mm]
    \item the modular group, denoted by $G_1$,
    \item a group $G_2$ uniformizing a genus $0$ orbifold with two punctures and two orbifold points each of order $2$, and
    \item a degree $2$ hyperbolic Blaschke product $B$, 
\end{itemize}
to get a weak B-involution $S:\overline{\cD}\rightarrow\widehat{\C}$. The domain $\cD$ turns out to be an open annulus (see Figure~\ref{111_fig}). Therefore, the blender surface $\Sigma$ would be a torus. The meromorphic map $\cR:\Sigma\rightarrow \widehat{\C}$ would end up having degree $4$ with $6$ simple critical points and a double critical point (see Example~\ref{111_ex}). By our main theorems, we would get an injective map from $\text{Teich}(G_2)\times \mathcal{B}_2$ into the Hurwitz space $\mathcal{H}_{1,4,11}^{\text{hyp}}$, which is an appropriate moduli space of tori equipped with $11$ ordered marked points and degree $4$ elliptic functions. 

The same can be repeated for a Hecke group $G_3\cong \Z/2\Z\ast\Z/5\Z$, the group $G_2$ as above, and a group $G_4$ uniformizing a sphere with $2$ punctures and an order $2$ orbifold point (see Example~\ref{111''_ex}). In this case, $\Sigma$ is a genus $2$ Riemann surface, and the meromorphic map $\cR:\Sigma\rightarrow\widehat{\C}$ is of degree $5$ with a unique pole. Consequently, Theorem~\ref{hurwitz_thm_intro} would give us an injective map from $\text{Teich}(G_2)$ to the Hurwitz space~$\mathcal{H}_{2, 5, 16}^{\text{hyp}}$.

\subsection{Questions and Connections}\label{questions_subsec}

\subsubsection*{Generalizations of previous results} In \cite{mj2023matings}, the correspondences considered were obtained as matings of a single group with a single map. The current paper extends that framework to allow for the combination of multiple groups with multiple maps.

The main result of \cite{mj2024simultaneous} develops techniques for combining two suitably chosen Fuchsian groups, each uniformizing a surface in $\mathcal{F}$. This contrasts sharply with the Bers’ simultaneous uniformization theorem; as the two groups involved need not produce homeomorphic quotient orbifolds. In the present work, we generalize this on two fronts. We permit an arbitrary (finite) number of Fuchsian groups (each uniformizing a surface in $\mathcal{F}$), and we impose no restrictions on their choice.

\subsubsection*{Embedding connectedness loci in Hurwitz spaces} The statement of Theorem~\ref{hurwitz_thm_intro} asserts the existence of an injective map from the product of finitely many Teichm\"uller spaces and Blaschke product spaces into an appropriate Hurwitz space. Spaces of hyperbolic Blaschke products can be identified with the principal hyperbolic components in connectedness loci of complex polynomials. In the spirit of \cite{luo2024general}, we ask if one can construct injective maps from the product of finitely many Teichm\"uller spaces and connectedness loci of polynomials into Hurwitz spaces, and whether such maps are generically continuous?

In the special case of matings of the groups $G_1, G_2$, and quadratic hyperbolic Blaschke products $B$ described above, this reduces to the following concrete question: 
\begin{question}
Let $G_1$ be the modular group, $G_2$ be a Fuchsian group uniformizing a sphere with two punctures and two order $2$ orbifold points, and $g_c(z):=z^2+~c$. Does there exist a topological embedding $\Psi$ of the Mandelbrot set $\mathbb{M}$ into an appropriate moduli space of pairs $(\Sigma, \cR)$ such that
\begin{enumerate}[leftmargin=8mm]
    \item $\Sigma$ is a marked torus,
    \item $\cR:\Sigma\to\widehat{\C}$ is a degree $4$ elliptic function, and
    \item if $\Psi(c)=(\Sigma, \cR)$, for $c\in\mathbb{M}$, then the correspondence $\cR(w)=\cR(-z)$ on the torus $\Sigma$ combines $G_1$, $G_2$, and $g_c$?
\end{enumerate}
\end{question}

\subsubsection*{Boundedness and degeneration}

In order to investigate the parameter spaces of correspondences arising from the mating framework designed in this paper, it is important to extend the techniques and results of \cite{luo2025} to the current setting. One important question in this direction is whether the injective map $\mathfrak{B}$ in Theorem~\ref{hurwitz_thm_intro} admits relatively compact images under appropriate restrictions. More precisely, we ask the following.
    \begin{question}
        Fix all but one of the coordinates in the product of Teichm\"uller spaces and Blaschke product spaces, and consider its image under $\mathfrak{B}$. Is this restricted image relatively compact in an appropriate topology on the Hurwitz space?
    \end{question}
An affirmative answer to this question would allow one to study the correspondences that appear as limits of the correspondences combining groups and maps (as in Theorem~\ref{mating_thm_intro}).

We note that boundedness questions have a long and rich history in the study of Kleinian groups as well as rational maps; for instance, see \cite{Ber70,Thu86a,Thu86b, Kap01} for results in the Kleinian world, and \cite{Eps00,NP20,Luo22,DL25} and references therein for various boundedness results in rational parameter spaces.

The study of boundedness of various parameter loci of complex dynamical systems often necessitates a good understanding of dynamically meaningful compactifications of the ambient spaces. We refer the reader to \cite{DeM05,McM09,DF14,Kiw15,Luo21} for such compactifications  for rational maps and the study of limiting dynamics of their boundary objects. A framework to study degenerations of certain algebraic correspondences on the Riemann sphere to algebraic correspondences defined on trees of spheres was set up in \cite{luo2025}. In the present setting, the interplay between the dynamical properties of the correspondences and the algebro-geometric properties of the ambient Hurwitz spaces (that parametrize the correspondences) merits further exploration.

\begin{task}\label{degen_task}
Develop a theory of degenerations for correspondences on genus $g$ compact Riemann surfaces given by Formula~\eqref{intro_corr_eq}. Specifically, study when such correspondences degenerate to algebraic correspondences defined on trees of compact Riemann surfaces (possibly of lower genera) as the pair of $(\Sigma,\cR)$ degenerates.
\end{task}
We remark that when $g\geq 1$, the degeneration of the pair $(\Sigma,\cR)$ is equivalent to the surface $\Sigma$ going to the boundary of the moduli space of genus $g$ surfaces or the degree of the meromorphic map $\cR$ dropping in the limit. It is the former possibility that makes Task~\ref{degen_task} more involved than the degenerations studied in \cite{luo2025}.

\subsubsection*{The inverse problem} Finally, keeping in mind that the mating framework developed in this paper combines a collection of groups and maps to produce a surface $\Sigma$ and a meromorphic map $\cR:\Sigma\rightarrow\widehat{\C}$, we ask the following.  
    \begin{question}
    Can we characterize such tuples $(\Sigma, \cR)$? In other words, can we determine when such tuples appear as matings of groups and maps? 
    \end{question}

\subsection*{Acknowledgments} This research was supported in part by the International Centre for Theoretical Sciences (ICTS) during the course of the program - ICTS New trends in Teichm{\"u}ller theory (code: ICTS/nteich2025/02). Part of this work was carried out during the authors' visit to IIT Tirupati and University of Barcelona. The authors thank them for their support and hospitality.
The authors also thank Yusheng Luo and Mahan Mj for useful discussions.

\section{Preliminaries}\label{prelim_sec}

\subsection{Critically Fixed Rational Maps}\label{crit_fixed_rat_subsec}

Let $f: \widehat{\mathbb{C}}\rightarrow\widehat{\mathbb{C}}$ be a holomorphic map of degree $d$. In other words, $f$ is a rational function $f(z) = \frac{p(z)}{q(z)}$, where $p(z)$, $q(z)$ are co-prime polynomials. Let $\text{Fix}(f)$ be the set of fixed points of $f$, and $\text{Crit}(f)$ be the set of critical points of $f$. The sets $\text{Fix}(f)$ and $\text{Crit}(f)$ play an important role in both the local and global dynamics of $f$. Counted with multiplicity, there are $2d-2$ critical points and $d+1$ fixed points of $f$. We are interested in those functions $f$ whose critical points are fixed. 

Critically fixed rational maps form a proper subclass in the family of postcritically finite rational maps - these are maps with finite critical orbits. Critically fixed rational maps were studied extensively in \cite{pilgrim15,Hlu26} and a complete topological classification of such maps were given (based on the Thurston theory for branched coverings of $\mathbb{S}^2$).
The following theorem will be sufficient for our purposes.

\begin{theorem}[\cite{pilgrim15}]\label{kevin}
    Let $m_1, m_2, \cdots, m_n, d$ be positive integers such that 
    \begin{itemize}
        \item $m_i \leq d-1$, $i\in\{1,\cdots,n\}$, and
        \item $\displaystyle\sum_{i = 1}^n m_i = 2d-2$.
    \end{itemize}
    Then, there exists a critically fixed rational map $R$ of degree $d$, such that
    \begin{itemize}
        \item $\mathrm{Crit}(R)=\{c_1, c_2, \cdots, c_n\}$, with
        \item $m_i$ being the multiplicity of the critical point $c_i$, $i\in\{1,\cdots,n\}$,
    \end{itemize}
    if and only if $n \leq d$.
\end{theorem}

The theorem above gives us enough flexibility to construct critically fixed rational maps. We record a precise statement in the following lemma. This will be used in a later section.

\begin{lem}\label{combi_lem}
    Given positive integers $m_1, m_2, \cdots, m_n$, there exists a critically fixed polynomial $P$ and a subset $\{c_i\}_{i=1}^n$ of $\text{Crit}(P)$ such that the multiplicity of the critical point $c_i$ is $m_i$, $i\in\{1,\cdots,n\}$.
\end{lem}

\begin{proof}
    Set $m_{n+1} := \sum_{i = 1}^nm_i$ and $d := m_{n+1} + 1$. The following observations are immediate.
    \begin{itemize}
        \item $m_i \leq d -1$ for $i\in\{1,\cdots, n+1\}$.
        \item $\sum_{i = 1}^{n+1}m_i = 2d-2$.
        \item $n+1\leq \sum_{i = 1}^nm_i + 1 = d$.
    \end{itemize}
    By Theorem~\ref{kevin}, there exist a critically fixed rational map $P$ such that  $\text{Crit}(P)=\{c_1,\cdots,c_{n+1}\}$ and the multiplicity of the critical point $c_i$ is $m_i$, $i\in\{1,\cdots,n+1\}$. As the critical point $c_{n+1}$ of the degree $d$ rational map $P$ has multiplicity $m_{n+1}=d-1$ (i.e., $c_{n+1}$ is a fully ramified critical point), the map $P$ can be chosen to be a polynomial, up to M{\"o}bius conjugation.
\end{proof}

\subsection{Virtually mateable maps}\label{mateable_maps_subsec}
We will now briefly describe a collection of maps, called \emph{Bowen-Series maps} and \emph{factor Bowen-Series maps}, which were introduced in \cite{bowen1979markov, mj2023combining, mj2023matings}. These maps are (virtually) orbit equivalent to certain Fuchsian groups and topologically compatible with Blaschke products. The Fuchsian groups associated with (factor) Bowen-Series maps uniformize a class of genus $0$ orbifolds. Specifically, this is the class $\mathcal{F}$ of hyperbolic orbifolds of genus $0$, 
\begin{itemize}
    \item with any number of punctures, 
    \item at most one order $n\geq 3$ orbifold point, and
    \item at most two order $2$ orbifold points.
\end{itemize}
In the following subsections we will briefly recall the construction and a few relevant properties. 

\subsubsection{Continuous Bowen-Series maps}\label{cBS_subsec}
If $S\in\cF$ has no order $n\geq 3$ orbifold points, one can choose an appropriate fundamental domain for the Fuchsian model of $S$ such that the associated (classical) Bowen-Series map turns out to be continuous.

To see this, let $\Pi$ be a closed regular ideal (hyperbolic) polygon with $p$ sides, say $C_{1, s}$, $s\in\{1,\cdots, p\}$, (ordered cyclically), such that $1$ is an vertex of $\Pi$, and $C_{1,1}$ connects $1$ with $\exp (2\pi i/p)$. Throughout this section, the notation has been chosen to ensure consistency with the construction to follow in Section~\ref{fBS_subsec}.  

We shall now spell out two kinds of side-pairings of $\Pi$ by automorphisms of $\mathbb{D}$. The second one is specifically for $p$ even, while the first one works for any integer~$p$.
\begin{enumerate}[leftmargin=8mm]
    \item[\textbf{I.}] Let $g_s$ be the M{\"o}bius map that is reflection along $C_{1, s}$ followed by a reflection along the real line. It takes the geodesic $C_{1, s}$ to $C_{1, p+1-s}$, for $s\in\{1,\cdots, p\}$. Note that, $g_{p+1-s} = g_s^{-1}$. In particular, if $p$ is odd, then $g_{(p+1)/2}$ is an order two element. (See Figure~\ref{side_pairing_fig} (left).)
    
    \item[\textbf{II.}] When $p$ is even, there exists a geodesic $\widetilde{l}$ that meets both $C_{1,1}$ and $C_{1,(p+2)/2}$ orthogonally. Define the M{\"o}bius map $\widetilde{g}_s$ to be reflection along $C_{1,s}$ followed by reflection along $\widetilde{l}$, for $s\in\{1,\cdots, p\}$. Again, note that $\widetilde{g}_{p+2-s} = \widetilde{g}_s^{-1}$, where $\widetilde{g}_{p+1}\equiv \widetilde{g}_1$, by convention. This implies that $\widetilde{g}_1$ and $\widetilde{g}_{(p+2)/2}$ are the only order $2$ elements. (See Figure~\ref{side_pairing_fig} (right).)
\end{enumerate}

By Poincar{\'e} polygon theorem, the maps $g_s$, $s\in\{1,\cdots, p\}$, in Case I (respectively, $\widetilde{g}_s$, $s\in\{1,\cdots, p\}$, in Case II), form a symmetric set of generators for a Fuchsian group $\Gamma_{1, p}$. To avoid notational clutter, we use the same symbol for the group generated by the maps in Cases I and II, both here and in Section~\ref{fBS_subsec}. None of the general statements appearing in the sections to follow will require separate treatment for the two types of groups, except in a few concrete examples in Section~\ref{examples_sec} (and in discussions concerning these examples in Sections~\ref{correspondences_sec} and~\ref{Hurwitz_sec}), where different notations will be used. Note that the closed regular ideal $p$-gon $\Pi$ is a closed fundamental domain for the action of $\Gamma_{1, p}$ on $\mathbb{D}$.

\noindent In Case I, the quotient $\mathbb{D}/\Gamma_{1, p}$ is a sphere with 
\begin{itemize}
    \item $\lfloor p/2\rfloor+1$ punctures,
    \item at most one order $2$ orbifold point, and
    \item no order $n\geq 3$ orbifold point.
\end{itemize}
In Case II, the quotient $\mathbb{D}/\Gamma_{1, p}$ is a sphere with
\begin{itemize}
    \item $p/2$ punctures,
    \item two order $2$ orbifold points, and
    \item no order $n\geq 3$ orbifold point.
\end{itemize}

Note that $\cU_{\Gamma_{1, p}}:=\mathbb{D}\backslash\Pi$ is a disjoint union of $p$ hyperbolic half-planes, each bound by the geodesic $C_{1, s}$, $s \in \{1, 2, \cdots, p\}$, and a part of the boundary circle $\mathbb{S}^1$. Denote by $A_{\Gamma_{1, p}}^{\text{BS}}:\overline{\cU_{\Gamma_{1, p}}}\to\overline{\D}$, the Bowen-Series map for the aforementioned side-pairing transformations (see \cite{bowen1979markov}). Specifically, the map is given by the action of $g_s$ (respectively, $\widetilde{g}_s$) in the half-plane bound by the geodesic $C_{1, s}$ and $\mathbb{S}^1$.
\begin{figure}
\captionsetup{width=0.98\linewidth}
    \centering
    \includegraphics[width=0.96\linewidth]{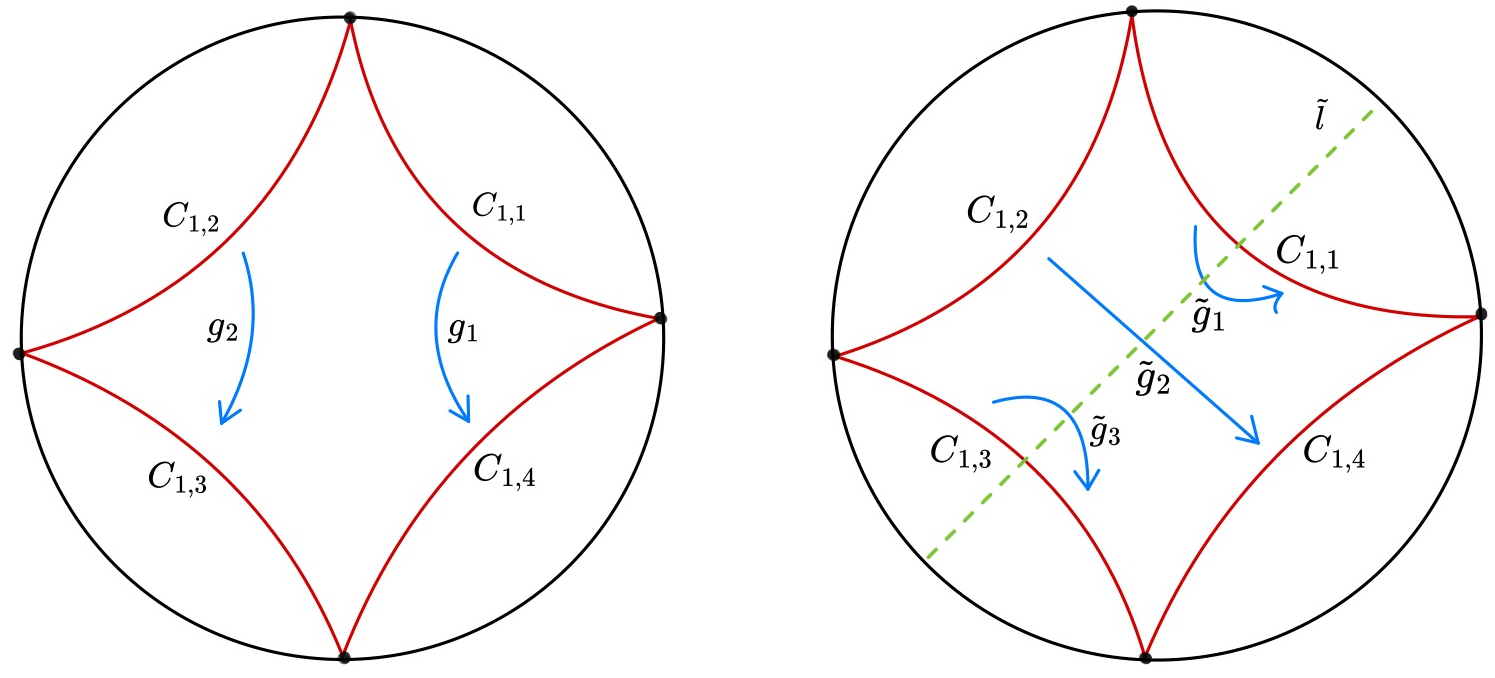}
    \caption{Left: Depicted is a side-pairing of the first kind. The quotient surface for the group generated by these side-pairings has three cusps and no orbifold points. Right: A side-pairing of the second kind is shown here. The corresponding quotient surface has two order $2$ orbifold points and two cusps.} 
    \label{side_pairing_fig}
\end{figure}

According to \cite[\S 3]{mj2023combining}, $A_{\Gamma_{1, p}}^{\text{BS}}:\mathbb{S}^1\rightarrow \mathbb{S}^1$ is an orientation-preserving covering map of degree $p-1$ and is \emph{orbit equivalent} to the group $\Gamma_{1,p}$. Further, $A_{\Gamma_{1, p}}^{\text{BS}}\vert_{\mathbb{S}^1}$ is an expansive map, implying that it is topologically conjugate to the map $z^{p-1}\big|_{\mathbb{S}^1}$. We denote by $h:\mathbb{S}^1 \rightarrow \mathbb{S}^1$ the homeomorphism that conjugates $z^{p-1}$ to $A_{\Gamma_{1, p}}^{\text{BS}}$, normalized so that 
\begin{enumerate}[leftmargin=8mm]
    \item in Case I, the fixed point $1$ of $z^{p-1}$ goes to the fixed point $1$ of~$A_{\Gamma_{1, p}}^{\text{BS}}$, and
    \item in Case II, the fixed point $1$ of $z^{p-1}$ goes to the endpoint of $\widetilde{l}$ in the half-plane determined by $C_{1,1}$ (contained in $\overline{\cU_{\Gamma_{1, p}}}$).
\end{enumerate}

The elements of the Teichm{\"u}ller space $\text{Teich}(\Gamma_{1, p})$ are given by Fuchsian representations $\rho: \Gamma_{1, p}\rightarrow \Gamma < \text{PSL}_2(\mathbb{R})$, $\rho(g) = \psi_{\rho}\circ g \circ \psi_{\rho}^{-1}$, for $g\in \Gamma_{1, p}$, where $\psi_{\rho}$ is a quasiconformal homeomorphism of $\widehat{\mathbb{C}}$ that preserves $\mathbb{D}$. We can and will require that $\psi_\rho(\Pi)$ is a hyperbolic ideal polygon and that $\psi_{\rho}(1) = 1$. 
\begin{definition}
    For each representation $(\rho: \Gamma_{1, p}\rightarrow \Gamma) \in \text{Teich}(\Gamma_{1, p})$, the \textbf{Bowen-Series map} $A_{\Gamma}^{\text{BS}}:\overline{\cU_{\Gamma}}\rightarrow\overline{\mathbb{D}}$ associated to the representation $\rho$ is given by $\psi_{\rho}\circ A_{\Gamma_{1, p}}^{\text{BS}}\circ \psi_{\rho}^{-1}$, where, $\cU_{\Gamma}:=\psi_\rho\big(\cU_{\Gamma_{1, p}}\big)$.
\end{definition}
\noindent By construction, the deformation space of the Bowen-Series map $A_{\Gamma_{1, p}}^{\text{BS}}$ is parametrized by the Teichm{\"u}ller space of the quotient $\mathbb{D}/\Gamma_{1, p} \in \mathcal{F}$.

Finally, we set $h_\Gamma:=\psi_\rho\circ h:\mathbb{S}^1\to\mathbb{S}^1$, which is a normalized conjugacy between $z^{p-1}$ and $A_{\Gamma}^{\text{BS}}$. The surgery construction in Section~\ref{conf_mating_sec} relies crucially on this conjugacy.

\subsubsection{Factor Bowen-Series Maps}\label{fBS_subsec} 
If the surface $S \in \mathcal{F}$ contains an orbifold point of order $n \geq 3$, selecting a fundamental domain for which the associated (classical) Bowen-Series map is continuous becomes intrinsically problematic. To overcome this difficulty, we pass to a subgroup of index $n$ and construct a set of generators for the subgroup such that the corresponding Bowen-Series map exhibits only controlled discontinuities. This Bowen-Series map, under an $n$-fold branched covering, descends to the desired continuous map.

We now proceed to expound the construction of factor Bowen-Series maps, and refer the reader to Figure~\ref{fac_BS_modular_fig} and Figure~\ref{remark_2.5_fig} for pictorial illustrations. Pick $n, p \in \mathbb{N}$, with $n\geq 3$. Set $\omega = e^{\frac{2\pi i}{n}}$, and $M_{\omega}(z) := \omega z$. 

Let $\Pi$ be a closed regular ideal (hyperbolic) polygon with $np$ sides, say $C_{r, s}$, $r\in\{1,\cdots, n\}$, $s\in\{1,\cdots, p\}$, where each side is explicitly given by the following bi-infinite hyperbolic geodesic,
$$
C_{r, s}:= \overline{e^{\frac{2\pi i(r-1)}{n} + \frac{2\pi i(s-1)}{np}}, e^{\frac{2\pi i(r-1)}{n}+\frac{2\pi is}{np}}}.
$$
As before, we shall present two kinds of side pairings of $\Pi$ by automorphisms of $\mathbb{D}$. The latter is specifically for even $p$, while the former works for any pair $n, p$.
\begin{enumerate}[leftmargin=8mm]
    \item[\textbf{I.}] Let $l$ denote the diameter in $\mathbb{D}$ joining $\pm e^{i\pi/n}$, and let $g_s$ be the M{\"o}bius map that is reflection along $C_{1, s}$ followed by reflection along $l$, for $s\in\{1,\cdots, p\}$. It takes the geodesic $C_{1, s}$ to $C_{1, p+1-s}$. This sets up side pairings for geodesics in the sector bound by points $1$ and $e^{\frac{2\pi i}{n}}$. For the other $n-1$ sectors, the side pairings are given by the collection of maps: $M_{\omega}^{r-1}\circ g_s\circ M_{\omega}^{-(r-1)}$, for $r\in \{2, \cdots, n\}$ and $s\in \{1, 2, \cdots, p\}$. (See Figure~\ref{fac_BS_modular_fig}.)
    
    \item[\textbf{II.}] When $p$ is even, there exists a geodesic $\widetilde{l}$ that intersects both $C_{1,1}$ and $C_{1, (p+2)/2}$ orthogonally. Define the M{\"o}bius map $\widetilde{g}_s$ to be reflection along $C_{1, s}$ followed by reflection along $\widetilde{l}$, for $s\in\{1,\cdots, p\}$. We now transfer the side pairings in the first sector to the other $n-1$ sectors via the collection of maps: $M_{\omega}^{r-1}\circ \widetilde{g}_s\circ M_{\omega}^{-(r-1)}$, for $r\in \{2, \cdots, n\}$, and $s\in \{1, 2, \cdots, p\}$. (See Figure~\ref{remark_2.5_fig}.)
\end{enumerate}
\begin{figure}[h!]
\captionsetup{width=0.98\linewidth}
    \centering
    \includegraphics[width=0.8\linewidth]{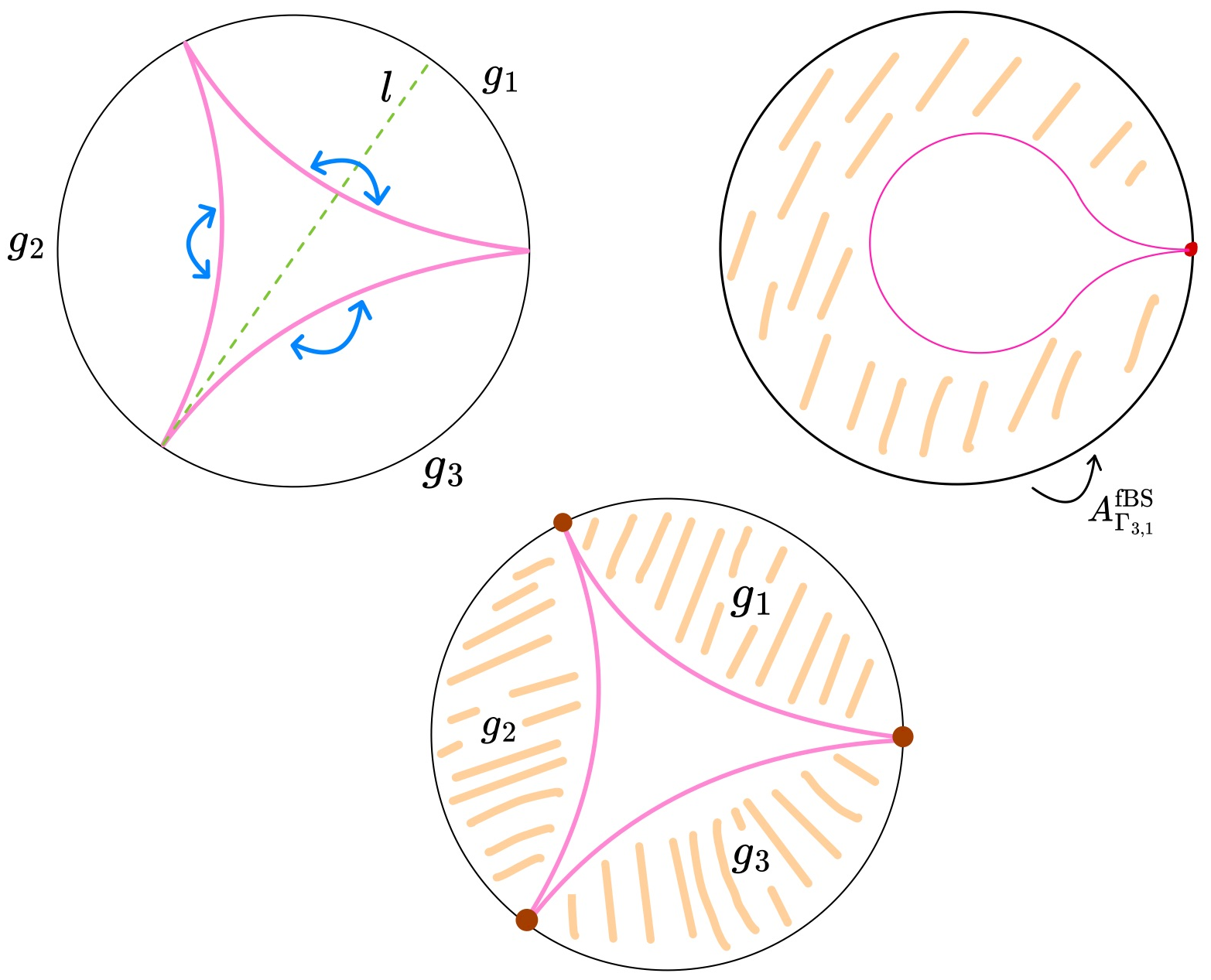}
    \caption{Illustrated above is the construction for $n = 3$, $p = 1$. The side pairing $g_1$ is defined to be reflection along the geodesic $C_{1, 1}$ followed by the reflection along the perpendicular bisector of $C_{1,1}$ (drawn in green). The map $A_{\Gamma_{3, 1}}$ is defined on the three orange pockets as shown. And finally, by taking the quotient under a $\mathbb{Z}/3\mathbb{Z}$ action, we end up with a factor of $A_{\Gamma_{3, 1}}^{\text{BS}}$, shown on the top right.} 
    \label{fac_BS_modular_fig}
\end{figure}

\noindent By the Poincar{\'e} polygon theorem, the maps 
$$
M_{\omega}^{r-1}\circ g_s\circ M_{\omega}^{-(r-1)}, \hspace{2mm} r\in \{1, 2, \cdots, n\}, \hspace{1mm} s\in \{1, 2, \cdots, p\}, \text{ in Case I,}
$$
and the maps
$$
M_{\omega}^{r-1}\circ \widetilde{g}_s\circ M_{\omega}^{-(r-1)}, \hspace{2mm} r\in \{1, 2, \cdots, n\}, \hspace{1mm} s\in \{1, 2, \cdots, p\}, \text{ in Case II,}
$$
generate a Fuchsian group $\Gamma_{n, p}$, where the closed regular ideal $np$-gon $\Pi$ is a closed fundamental domain for the $\Gamma_{n, p}$-action on $\mathbb{D}$. The orbifold $\mathbb{D}/\Gamma_{n, p}$ is an $n$-fold cyclic cover of $\mathbb{D}/\widehat{\Gamma}_{n,p}$, where $\widehat{\Gamma}_{n,p}:=\langle\Gamma_{n, p}, M_{\omega}\rangle$. 
The closed polygon bound by the geodesics $\{C_{1,s}\}_{s=1}^p$, and the geodesics connecting $0$ to $1$ and $e^{\frac{2\pi i}{n}}$, is a closed fundamental domain for the action of $\widehat{\Gamma}_{n,p}$ on $\mathbb{D}$.  

\noindent It now follows that in Case I, the quotient $\mathbb{D}/\widehat{\Gamma}_{n,p}$ is a sphere with 
\begin{itemize}
    \item $\lfloor p/2\rfloor+1$ punctures,
    \item at most one order $2$ orbifold point, and
    \item one order $n\geq 3$ orbifold point;
\end{itemize}
and in Case II, the quotient $\mathbb{D}/\widehat{\Gamma}_{n,p}$ is a sphere with
\begin{itemize}
    \item $p/2$ punctures,
    \item two order $2$ orbifold points, and
    \item one order $n\geq 3$ orbifold point.
\end{itemize}

The domain $\overline{\mathbb{D}}\backslash \Pi$ is a collection of hyperbolic half-planes (or ``pockets"), each bound by the geodesic $C_{r, s}$ and the boundary circle $\mathbb{S}^1$. We now look at the usual Bowen-Series map $A_{\Gamma_{n, p}}^{\text{BS}}$ for the aforementioned side-pairing transformation \cite{bowen1979markov}. The map
$$
A_{\Gamma_{n, p}}^{\text{BS}}: \overline{\mathbb{D}}\backslash \Int{\Pi}\rightarrow \overline{\mathbb{D}}
$$
is given by the action of $M_{\omega}^{r-1}\circ g_s\circ M_{\omega}^{-(r-1)}$ (or $M_{\omega}^{r-1}\circ \widetilde{g}_s\circ M_{\omega}^{-(r-1)}$) in the pocket bound by $C_{r, s}$ and $\mathbb{S}^1$. Note that in Case I this map is discontinuous at points in the $M_{\omega}$-orbit of $1$ (i.e., the $n$-th roots of unity), and in Case II it is discontinuous at the $M_{\omega}$-orbits of the points $1, e^{2\pi i/np}$. But the discontinuities are controlled, in the sense that in both the cases, the left-sided and right-sided limits of $A_{\Gamma_{n, p}}^{\text{BS}}$ at the points of discontinuity lie in the same $M_{\omega}$-orbit. Moreover, $A_{\Gamma_{n, p}}^{\text{BS}}$ commutes with the map $M_{\omega}$. Thus, the map descends via the standard $n$-fold branched covering $z\mapsto z^n$ to a factor of the Bowen-Series map, denoted by $A_{\Gamma_{n, p}}^{\text{fBS}}$. Moreover, the map $A_{\Gamma_{n, p}}^{\text{fBS}}$ has a unique critical value in $\mathbb{D}$ at $0$ with $p$ points in its fiber each of multiplicity~$n-1$.

Define $\mathcal{U}_{\Gamma_{n,p}}$ to be the projection of $\mathbb{D}\backslash \Pi$ under $z\mapsto z^n$, then $A_{\Gamma_{n, p}}^{\text{fBS}}$ is defined on $\overline{\mathcal{U}}_{\Gamma_{n,p}}$. Note that $\mathbb{S}^1 \subset \partial \mathcal{U}_{\Gamma_{n,p}}$. By \cite[Proposition~2.5]{mj2023matings}, the map $A_{\Gamma_{n, p}}^{\text{fBS}}: \mathbb{S}^1 \rightarrow \mathbb{S}^1$ is an orientation-preserving expansive covering map of degree $np-1$ . In particular, it is topologically conjugate to the map $z^{np-1}\big|_{\mathbb{S}^1}$. Let us denote by $h:\mathbb{S}^1 \rightarrow \mathbb{S}^1$ the homeomorphism  that conjugates $z^{np-1}$ to $A_{\Gamma_{n, p}}^{\text{fBS}}$, normalized so that,
\begin{itemize}[leftmargin=6mm]
    \item in Case I, the fixed point $1$ of $z^{np-1}$ goes to the fixed point $1$ of $A_{\Gamma_{n, p}}^{\text{fBS}}$, and
    \item in Case II, the fixed point $1$ of $z^{np-1}$ goes to the projection under $z\mapsto z^n$ of the endpoint of $\widetilde{l}$ in the half-plane determined by $C_{1,1}$ (contained in $\overline{\D}\setminus\Pi$).
\end{itemize}

We denote by $\text{Teich}(\Gamma_{n, p})$ the Teichm{\"u}ller space of $\Gamma_{n,p}$, where the quasiconformal conjugations $\psi_\rho$ are normalized as in Section~\ref{cBS_subsec}.
Its subset $\text{Teich}^{\omega}(\Gamma_{n, p})$ is defined as the collection of $(\rho: \Gamma_{n, p}\rightarrow \Gamma) \in \text{Teich}(\Gamma_{n, p})$ such that $\psi_\rho$ commutes with $M_{\omega}$. Define
$$
\widehat{\Gamma} := \langle\Gamma, M_{\omega}\rangle.
$$
By associating to each $(\rho:\Gamma_{n,p}\to\Gamma) \in \text{Teich}^{\omega}(\Gamma_{n, p})$ the representation $(\widehat{\rho}:\widehat{\Gamma}_{n,p}\to\widehat{\Gamma})\in \text{Teich}(\widehat{\Gamma}_{n,p})$, where
$$
\widehat{\rho}\big|_{\Gamma_{n,p}} \equiv \rho\big|_{\Gamma_{n,p}} \hspace{2mm}\text{and}\hspace{2mm} \widehat{\rho}(M_{\omega}) = M_{\omega},
$$
we see that $\text{Teich}(\widehat{\Gamma}_{n,p})$ can be identified with $\text{Teich}^{\omega}(\Gamma_{n, p})$.
\begin{figure}[ht!]
    \captionsetup{width=0.98\linewidth}
    \begin{tikzpicture}
            \node[anchor=south west,inner sep=0] at (2.8,6) {\includegraphics[width=0.56\textwidth]{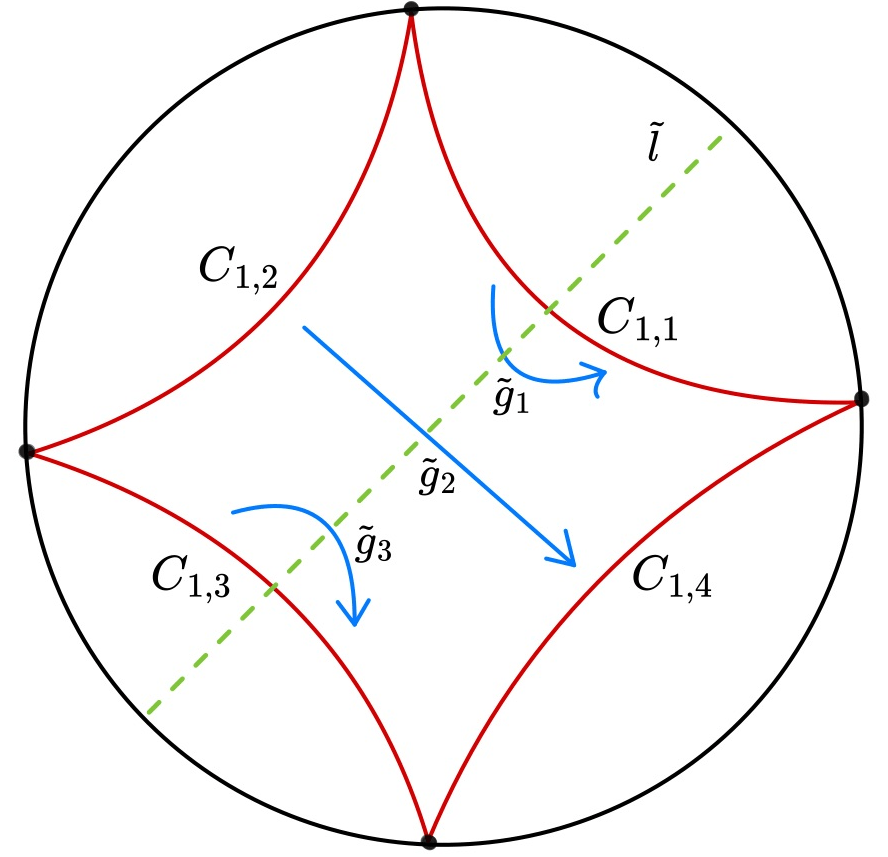}};
            \node[anchor=south west,inner sep=0] at (0,0)     {\includegraphics[width=0.48\linewidth]{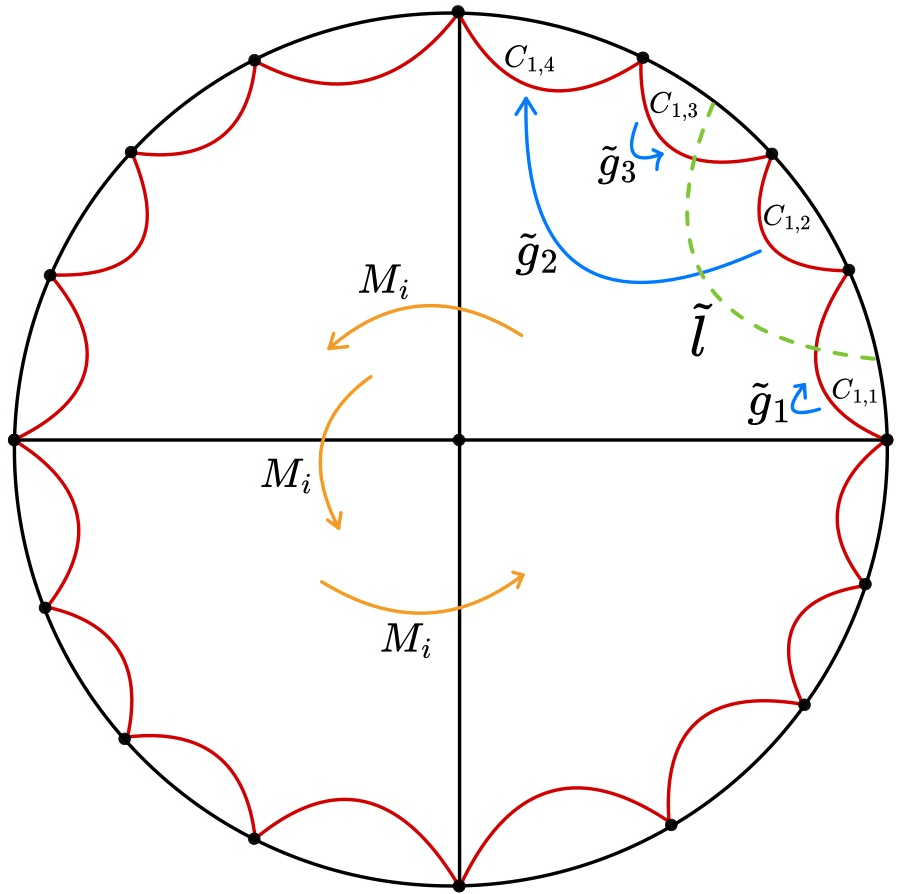}};
            \node[anchor=south west,inner sep=0] at (6.4,-0.2) {\includegraphics[width=0.48\linewidth]{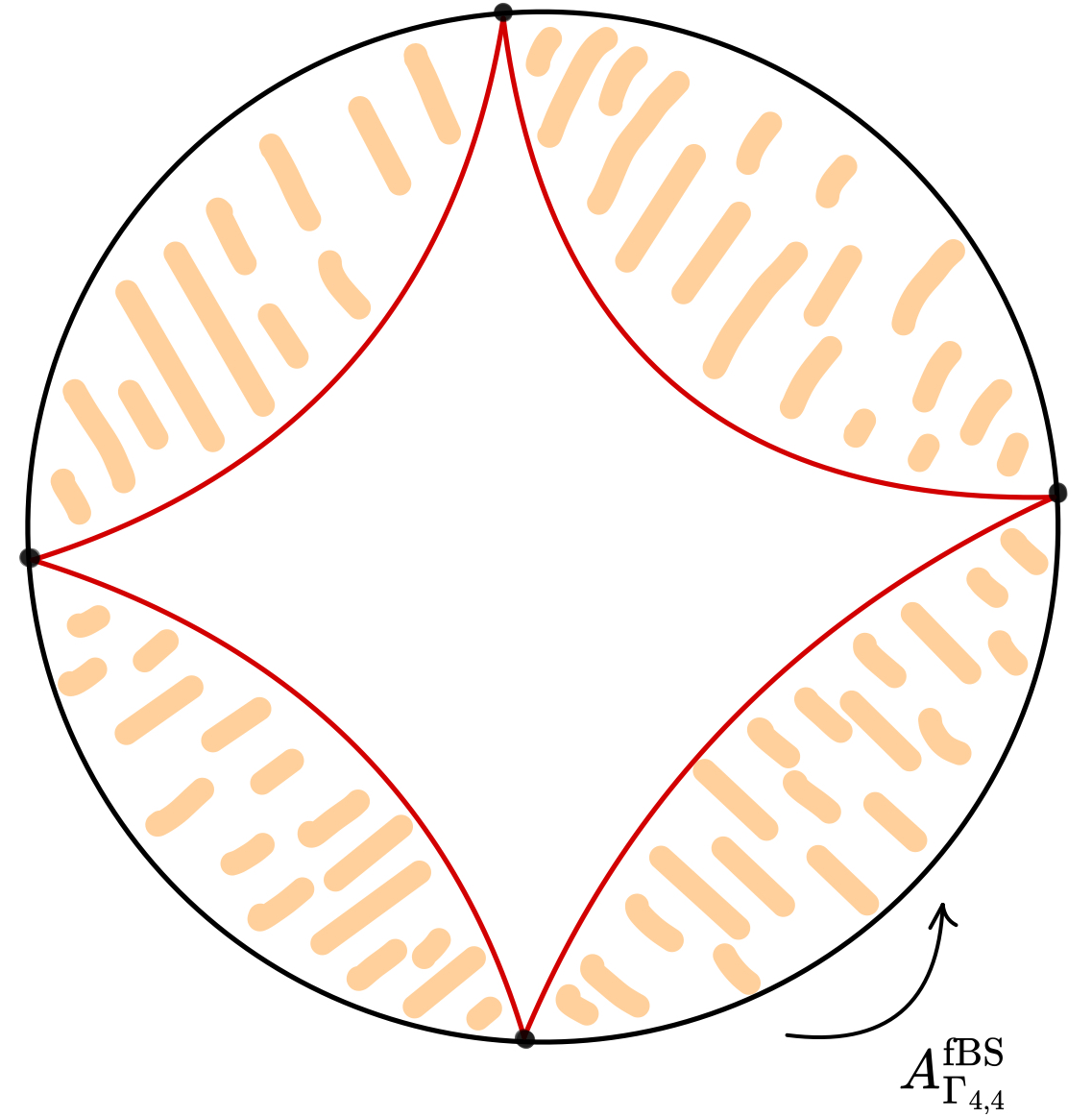}};
    \end{tikzpicture}
    \caption{Top: Shown are the side-pairings for a fundamental domain of a group that uniformizes a surface with two cusps and two order $2$ orbifold points, and no other orbifold points. Bottom left: The side pairings in the first quadrant here are combinatorially identical to the ones shown in the top figure. Those pairings are transferred to the other three sectors via the iterates of the map $z\mapsto iz$. The resulting quotient surface has two cusps, two order $2$ orbifold points, and one order $4$ orbifold point. Bottom right: The (classical) Bowen-Series map associated to the side-parings illustrated in the bottom left, descends under $z\mapsto z^4$ to a continuous map. The region shaded in orange is the domain where the said factor map is~defined.}
    \label{remark_2.5_fig}
\end{figure}

\begin{definition}
    For a given representation $(\rho:\Gamma_{n, p} \rightarrow \Gamma) \in \text{Teich}^{\omega}(\Gamma_{n, p})$, the \textbf{factor Bowen-Series map} $A_{\Gamma}^{\text{fBS}}: \overline{\mathcal{U}}_{\Gamma} \rightarrow \overline{\mathbb{D}}$, associated with the representation $\rho$, is given by $\psi_{\rho}\circ A_{\Gamma_{n, p}}^{\text{fBS}}\circ \psi_{\rho}^{-1}$, where $\mathcal{U}_{\Gamma}:= \psi_{\rho}\big(\mathcal{U}_{\Gamma_{n, p}}\big)$. 
\end{definition}

Note that for each $\rho \in \text{Teich}^{\omega}(\Gamma_{n, p})$, the associated Bowen-Series map $A_{\Gamma}^{\text{BS}}$ commutes with $M_{\omega}$, hence the quasiconformal conjugacy $\psi_{\rho}$ between $\Gamma_{n, p}$ and $\Gamma$ descends to a quasi-conformal conjugacy $\widehat{\psi}_{\rho}$ between $A_{\Gamma_{n, p}}^{\text{fBS}}$ and $A_{\Gamma}^{\text{fBS}}$. As the space $\text{Teich}^{\omega}(\Gamma_{n, p})$ can be identified with $\text{Teich}(\widehat{\Gamma}_{n,p})$, the collection of the factor Bowen-Series maps $A_\Gamma^{\text{fBS}}$ is parametrized by the Teichm{\"u}ller space of the quotient $\mathbb{D}/\widehat{\Gamma}_{n,p}~\in~\mathcal{F}$. 

Define $h_\Gamma:=\widehat{\psi}_\rho\circ h:\mathbb{S}^1\to\mathbb{S}^1$, the normalized map that establishes a topological conjugacy between $z^{np-1}$ and $A_{\Gamma}^{\text{fBS}}$. This conjugacy will play an important role in the surgery construction of Section~\ref{conf_mating_sec}.

\begin{remark}
     Notice that the side-pairing constructions in Sections~\ref{cBS_subsec} and~\ref{fBS_subsec} are identical when $n=1$. In this case, we use the convention that $\widehat{\Gamma}_{n,p} = \Gamma_{n,p}$, $\widehat{\Gamma}=\Gamma$, $\widehat{\Pi} = \Pi$, $\widehat{\rho}=\rho$, and $\text{Teich}^{\omega}(\Gamma_{n, p})=\text{Teich}(\Gamma_{n, p})$.
\end{remark}

\begin{remark}
     The key difference between the constructions in Sections~\ref{cBS_subsec} and~\ref{fBS_subsec} is the emergence of an order $n\geq 3$ orbifold point in the quotient surface of the latter. This particular feature is accommodated for by 
    {\upshape
    \begin{enumerate}[leftmargin=8mm,label=\alph*)]
        \item taking $n$ sectors of angle $\frac{2\pi}{n}$ each, and fitting geodesics appearing in Section~\ref{cBS_subsec} in every sector, and
        \item preserving the identifications among the geodesics in one sector, and transferring the identifications to the geodesics in the other sectors via maps $M_{\omega}^{r-1}$, for $r\in\{1,\cdots, n\}$. (See Figure~\ref{remark_2.5_fig}.)
    \end{enumerate}
    }
    Therefore, as the groups in Section~\ref{cBS_subsec} are associated with orbifolds that have any number of punctures and at most $2$ order two orbifold points; the groups in Section~\ref{fBS_subsec} are associated with orbifolds that admit an arbitrary number of punctures, at most $2$ order two orbifold points, and exactly $1$ order $n\geq 3$ orbifold point. The maps in Section~\ref{fBS_subsec} were first introduced in \cite{mj2023matings} to present an alternate way of combining modular groups with quadratic maps, thus bringing to attention a different perspective on the Bullett-Penrose family of correspondences \cite{bullett1994mating}.
\end{remark}

\begin{remark}\label{BSmap_domain_rem}
The domain of definition of the (factor) Bowen-Series map $A$ can be split into three parts, 
    $$
    \overline{\mathcal{U}}_{\Gamma} = \mathbb{S}^1\cup \cU_{\Gamma} \cup \big(\partial \mathcal{U}_{\Gamma}\backslash\mathbb{S}^1\big).
    $$
    {\upshape
\begin{enumerate}[leftmargin=8mm,label=(\roman*)]
        \item On the ``outer boundary", $\mathbb{S}^1$, the map $A$ is topologically equivalent to $z^d$, enabling boundary compatibility with Blaschke products. In Section~\ref{conf_mating_sec}, this will be used to construct conformal matings of Fuchsian groups (uniformizing orbifolds in $\mathcal{F}$) with Blaschke products.
        \item On the ``inner boundary", $\partial \mathcal{U}_{\Gamma}\backslash\mathbb{S}^1$, the map $A$ is an orientation-reversing analytic involution. This aspect of the map is imperative in the welding construction of Section~\ref{b_involutions_sec}, leading us to an explicit algebraic description of the conformal mating. It also underpins the construction of correspondences on higher genus hyperelliptic Riemann surfaces in Section~\ref{correspondences_sec}.
        \item  And finally, the map $A\big|_{\mathcal{U}_{\Gamma}}$ is parametrized by the Teichm\"uller space of an appropriate group. This allows us to quasiconformally deform the map $A$ and study its parameter space consequences in Section~\ref{Hurwitz_sec}. 
    \end{enumerate}
    }
    In summary, the map exhibits three distinct behaviors on three regions of its domain, which together form the central theme of the four sections that follow.
\end{remark}

\subsection{David Homeomorphisms}
The class of maps that will feature in this section is crucial our surgery construction. In fact, it is one of the key steps in our combination framework. Recall that a measurable function $f: U
\rightarrow \mathbb{C}$ is said to belong to the Sobolev space $W_{\text{loc}}^{1, p}(U)$ if $f \in L_{\text{loc}}^p(U)$ and $f$ has distributional derivatives lying in $L_{\text{loc}}^p(U)$.

\begin{definition}
    An orientation-preserving homeomorphism $H: U\rightarrow V$ between domains in the Riemann sphere is called a \textbf{David homeomorphism} if $H\in W_{\text{loc}}^{1, 1}(U)$, and there exist constants $C, \alpha, K_0 >0$ such that, 
    $$
    \sigma (\{z\in U: K_H(z) \geq K\}) \leq Ce^{-\alpha K}, \hspace{2mm} K\geq K_0,
    $$
    where $\sigma$ is the spherical measure, and $K_H$ is the dilatation of $H$ given by, 
    $$
    K_H(z) = \frac{1+|\mu_H(z)|}{1 - |\mu_H(z)|},
    $$
    where $\mu_H(z) = \frac{\partial H/\partial\overline{z}}{\partial H/\partial z}$ is the Beltrami coefficient of $H$. In addition, a measurable function $\mu: U \rightarrow \mathbb{D}$ that satisfies the above condition will be called a \textbf{David coefficient}. 
\end{definition}

An aspect of the theory of David homeomorphisms relevant for our purpose is the following integrability theorem  \cite[Theorem~20.6.2]{astala2008elliptic}.

\begin{theorem}[David Integrability Theorem]
    Let $\mu: \widehat{\mathbb{C}} \rightarrow \mathbb{D}$ be a David coefficient. Then there exists a David homeomorphism $H:\widehat{\mathbb{C}} \rightarrow \widehat{\mathbb{C}}$ in $W_{\text{loc}}^{1,1}(\widehat{\mathbb{C}})$ that solves the Beltrami equation
    $$
    \frac{\partial H}{\partial\overline{z}} = \mu\frac{\partial H}{\partial z}.
    $$
    Moreover, $H$ is unique up to post-composition with M{\"o}bius transformations.
\end{theorem}

David maps can be regarded as natural generalizations of quasiconformal maps, and the integrability theorem (above) serves as an analogue of the Measurable Riemann Mapping Theorem \cite[Theorem~5.3.4]{astala2008elliptic}. There is also a version of Weyl's lemma for David maps \cite[Theorem~20.4.19]{astala2008elliptic}. We record it in the following result.

\begin{theorem}
    Let $U\subset \widehat{\mathbb{C}}$ be open and let $f, g: U \rightarrow \widehat{\mathbb{C}}$ be David homeomorphisms with
    $$
    \mu_f = \mu_g
    $$
    almost everywhere. Then $f\circ g^{-1}$ is a conformal map on $g(U)$.
\end{theorem}

In the following lemma, David extension of certain circle homeomorphisms is established. This will play a key role in the construction of conformal matings in Section~\ref{conf_mating_sec}. Recall the map $h_{\Gamma}$ introduced in Sections~\ref{cBS_subsec} and~\ref{fBS_subsec}. 

\begin{lem}\label{david_ext_lem}
The circle homeomorphism $h_\Gamma:\mathbb{S}^1\rightarrow \mathbb{S}^1$ continuously extends to a David homeomorphism $h_\Gamma$ of $\mathbb{D}$. 

More generally, if $B$ is a degree $d$ Blaschke product with an attracting fixed point in $\D$, then there exists a circle homeomorphism $h_{B,\Gamma}$ that conjugates $B$ to $A_\Gamma^{\text{fBS}}$ (respectively, $A_{\Gamma}^{\text{BS}}$), and extends continuously to a David homeomorphism of $\D$.
\end{lem}
\begin{proof}
The first statement is the content of \cite[Lemma~3.4]{mj2023matings}, and \cite[Proposition~2.7]{mj2023combining}. For the second statement, let $\theta_{B}: \mathbb{S}^1\rightarrow\mathbb{S}^1$ be a quasisymmetric conjugacy from $B$ to $z^d$. The map $\theta_{B}$ can be extended continuously as a quasiconformal homeomorphism of $\mathbb{D}$. And by the first statement, the map $h_{\Gamma}$ can be extended continuously to $\mathbb{D}$ as a David homeomorphism. Using the same notation for the extensions of $\theta_B$ and  $h_{\Gamma}$ to the unit disk $\mathbb{D}$, define $h_{B, \Gamma}:= h_{\Gamma}\circ \theta_{B}$. By \cite[Proposition~2.5 (iv)]{lyubich2020david}, $h_{B, \Gamma}$ is a David homeomorphism of $\mathbb{D}$.
\end{proof}

\section{Conformal matings of Fuchsian groups and Blaschke products}\label{conf_mating_sec}

In this section, we shall prescribe a way to house conformal mating of several Fuchsian groups (uniformizing certain genus $0$ orbifolds) with Blaschke products via subhyperbolic rational maps. 

We shall continue to use the notation introduced in Section~\ref{mateable_maps_subsec}. In addition, we direct the reader to \cite{petersen2012notions} for a thorough treatment of topological/conformal matings.

Let $R$ be a subhyperbolic rational map with connected Julia set, and let $\{U_i\}_{i=1}^k$ be a collection of invariant Fatou components of $R$. We intend to replace the dynamics of $R$ on some of these Fatou components with the actions of appropriate (factor) Bowen-Series maps or hyperbolic Blaschke products. 

To this end, consider $(\rho_i: \Gamma_{n_i, p_i} \rightarrow \Gamma_i) \in \text{Teich}^{\omega_i}(\Gamma_{n_i, p_i})$, and let $A_{\Gamma_i}$ be the associated continuous Bowen-Series map or factor Bowen-Series map for $i\in \{1, 2, \cdots, l\}$, $l\leq k$, such that, 
$$
\text{deg}\Big(R\Big|_{U_i}\Big) = \text{deg}\Big(A_{\Gamma_i}\Big|_{\mathbb{S}^1}\Big) = n_ip_i-1.
$$

Let us also consider a collection of hyperbolic Blaschke products $B_j$, $j\in\{1,\cdots,r\}$, $r+l\leq k$ (i.e., Blaschke products with an attracting fixed point in $\D$), such that 
$$
\deg (R\vert_{U_{l+j}})=\deg (B_j),\ j\in\{1,\cdots,r\}.
$$

Let $r_i: \mathbb{D}\rightarrow U_i$ be the uniformizer of $U_i$, $i\in\{1,\cdots,l+r\}$. We know that the Julia set of $R$ is locally connected \cite[Theorem~19.7]{milnor2011dynamics}. Therefore, $r_i$ extends continuously to the boundary circle $\mathbb{S}^1$. The map
$$
\mathscr{B}_i\Big|_{\mathbb{D}}:= r_i^{-1}\circ R\circ r_i,\ i\in\{1,\cdots,l+r\},
$$
is a hyperbolic Blaschke product, and $r_i$ is a semi-conjugacy between $\mathscr{B}_i\Big|_{\mathbb{S}^1}$ and~$R\Big|_{\partial U_i}$. 

For $i\in\{1,\cdots,l\}$, let $h_i:\mathbb{S}^1\to\mathbb{S}^1$ be the normalized topological conjugacy from $\mathscr{B}_i$ to $A_{\Gamma_i}$ that admits a David extension to $\D$ (see Lemma~\ref{david_ext_lem}). Set $\zeta_i:=h_i\circ r_i^{-1}:U_i\to \D$, and note that $\zeta_i^{-1}$ extends as a continuous surjection $\zeta_i^{-1}:\mathbb{S}^1\to\partial U_i$ semi-conjugating $A_{\Gamma_i}$ to $R$, $i\in\{1,\cdots,l\}$.

For $j\in\{1,\cdots,r\}$, let $h_{l+j}:\mathbb{S}^1\to\mathbb{S}^1$ be a quasisymmetric conjugacy from $\mathscr{B}_{l+j}$ to $B_{j}$. We continuously extend these quasisymmetric conjugacies to quasiconformal homeomorphisms of $\D$.
Set $\zeta_{l+j}:=h_{l+j}\circ r_{l+j}^{-1}:U_{l+j}\to \D$, and note that $\zeta_{l+j}^{-1}$ extends as a continuous surjection $\zeta_{l+j}^{-1}:\mathbb{S}^1\to\partial U_{l+j}$ semi-conjugating $B_{j}$ to $R$, $j\in\{1,\cdots,r\}$.

Let us now formally define the notion of a conformal mating of the dynamical systems $R$, $\{A_{\Gamma_i}\}_{i=1}^l$, and $\{B_j\}_{j=1}^r$. Consider the disjoint union 
$$
\Bigg(\widehat{\mathbb{C}} - \bigsqcup_{i = 1}^{l+r} U_i\Bigg)\bigsqcup\left(\bigsqcup_{i = 1}^{l+r}\overline{\mathbb{D}}_i\right),
$$
and the map 
$$
R\bigsqcup_{i= 1}^l A_{\Gamma_i}\bigsqcup_{j=1}^r B_j: \Bigg(\widehat{\mathbb{C}} - \bigsqcup_{i = 1}^{l+r} U_i\Bigg)\bigsqcup\left(\bigsqcup_{i = 1}^l \overline{\cU}_{\Gamma_i}\right)\bigsqcup\left(\bigsqcup_{j=1}^r\overline{\D}_{l+j}\right) \longrightarrow \widehat{\mathbb{C}}\bigsqcup\left(\bigsqcup_{i = 1}^{l+r}\overline{\mathbb{D}}_i\right)
$$
where,
\begin{equation*}
    R\bigsqcup_{i= 1}^l A_{\Gamma_i}\bigsqcup_{j=1}^r B_j:= \begin{cases*}
        R & on $\ \widehat{\mathbb{C}} - \displaystyle\bigsqcup_{i = 1}^{l+r} U_i$\\
        A_{\Gamma_i} & on $\ \overline{\cU}_{\Gamma_i} ,\ i\in\{1,\cdots,l\}$\\
        B_j & on $\ \overline{\D}_{l+j},\ j\in\{1,\cdots,r\}$.
    \end{cases*}
\end{equation*}
Let $\sim_m$ be an equivalence relation on $\displaystyle\Bigg(\widehat{\mathbb{C}} - \bigsqcup_{i = 1}^{l+r} U_i\Bigg)\bigsqcup\left(\bigsqcup_{i = 1}^{l+r}\overline{\mathbb{D}}_i\right)$ where the points $w \in \partial U_i$ and $z \in \mathbb{S}^1_i$ are related if
$$
\zeta_i^{-1}(z) = w,\ i\in\{1,\cdots,l+r\}.
$$
After identifying $U_i$ with $\D_i$ in the co-domain via the homeomorphism $\zeta_i^{-1}$, $i\in\{1,\cdots,l+r\}$, it is easy to see that the map $R\bigsqcup_{i= 1}^l A_{\Gamma_i}\bigsqcup_{j=1}^r B_j$ gives rise to a partially defined continuous map, say $\widetilde{S}$, on the quotient, 
$$
\Bigg(\widehat{\mathbb{C}} - \bigsqcup_{i = 1}^{l+r} U_i\Bigg)\bigsqcup\left(\bigsqcup_{i = 1}^{l+r}\overline{\mathbb{D}}_i\right)/\sim_m \hspace{2mm}\cong \hspace{2mm}S^2.
$$
The map $\widetilde{S}$ will be called a the \textbf{topological mating} of $R$, $\{A_{\Gamma_i}\}_{i=1}^l$, and $\{B_j\}_{j=1}^r$. We say that the maps $R$, $\{A_{\Gamma_i}\}_{i=1}^l$, and $\{B_j\}_{j=1}^r$ are \textbf{conformally mateable} if the topological $2$-sphere, featured above, admits a complex structure that turns the topological mating $\widetilde{S}$ into a meromorphic map.

\begin{theorem}\label{conf_mating_thm}
    There exists a conformal mating $S$ of $R$, $\{A_{\Gamma_i}\}_{i=1}^l$, and $\{B_j\}_{j=1}^r$.
\end{theorem}
\begin{proof}
    Simply connected Fatou components of subhyperbolic rational maps are John domains \cite[Chapter~7, Theorem~3.1]{carleson1996complex}. In particular, $\{U_i\}_{i=1}^{l+r}$ are all John domains, which implies $\partial U_i$ is removable for $W^{1, 1}$ functions \cite{jones2000removability}.

  Now define the following topological dynamical system:
    \begin{equation*}
        \widetilde{S} = \begin{cases*}
      R & on\ $\ \widehat{\mathbb{C}}\backslash \displaystyle\bigsqcup_{i=1}^{l+r}U_i$ \\
      \zeta_i^{-1}\circ A_{\Gamma_i}\circ \zeta_i  & on $\ \zeta_i^{-1}(\overline{\cU}_{\Gamma_i})\subseteq \overline{U_i},\ i\in\{1,\cdots,l\}$\\
      \zeta_{l+j}^{-1}\circ B_{j}\circ \zeta_{l+j} & on $\ \zeta_{l+j}^{-1}(\overline{\D}_{l+j})=\overline{U_{l+j}},\ j\in\{1,\cdots,r\}$.
    \end{cases*}
    \end{equation*}
    To upgrade $\widetilde{S}$ to a conformal dynamical system, we construct an invariant Beltrami coefficient. Let $\mu_i$ be the pullback of the standard complex structure on $\mathbb{D}_i$ under the map $\zeta_i$, $i\in\{1,\cdots,l+r\}$. Use the dynamics of $R$ to spread $\mu_i$ on all pre-images of $U_i$, and everywhere else put the standard complex structure. It is now evident that the Beltrami coefficient $\mu$ so defined is $\widetilde{S}$ invariant. By applying arguments in \cite[Lemma~7.1]{lyubich2020david} (as is) one can prove that $\mu$ is a David coefficient. Hence, by David Integrability Theorem \cite[Theorem~20.6.2]{astala2008elliptic}, there exists a David homeomorphism 
    $$
    H:\widehat{\mathbb{C}}\rightarrow \widehat{\mathbb{C}}
    $$
    that straightens the coefficient $\mu$. Set, 
    $$
    S:= H\circ \widetilde{S}\circ H^{-1}
    $$
    As $\partial U_i$ is $W^{1, 1}$ removable, by \cite[Theorem~2.7]{lyubich2020david}, the set $H(\partial U_i)$ is conformally removable for $i\in \{1, 2, \cdots, l+r\}$. It is therefore sufficient to check that $S$ is conformal on $\displaystyle\text{Dom}(S)\backslash \bigcup_{i=1}^{l+r} H(\partial U_i)$.

    The maps $\zeta_i$ and $H$ are David homeomorphisms on $U_i$ straightening $\mu$, therefore by \cite[Theorem~20.4.19]{astala2008elliptic}, the map $\zeta_i\circ H^{-1}$ is conformal on $H(U_i)\cap \text{Dom}(S)$, $i\in \{1, 2, \cdots, l+r\}$. This implies that $S$ is holomorphic on $H(U_i)\cap\mathrm{Dom}(S)$, $i\in \{1, 2, \cdots, l+r\}$. Further, the maps $H\circ R$ and $H$ restricted to $\displaystyle\widehat{\mathbb{C}}\backslash \bigsqcup_{i=1}^{l+r} U_i$ again straighten $\mu$. Therefore, the map $S$ is holomorphic on $\displaystyle H\big(\widehat{\mathbb{C}}\backslash \bigsqcup_{i=1}^{l+r} U_i\big)$ as well. This concludes the proof of the fact that $S$ is holomorphic in the interior of its domain of definition. Evidently,
    \begin{enumerate}[leftmargin=6mm]
        \item $\mathfrak{X}_R:= H$ is a conformal conjugacy from $R$ to $S$ outside the grand orbit of~$\displaystyle\bigsqcup_{i=1}^{l+r} U_i$,
    
        \item $\mathfrak{X}_i := H\circ\zeta_i^{-1}$ is a conformal conjugacy from $A_{\Gamma_i}$ to $S$ on $\D_i$, $i\in\{1,\cdots,l\}$, and

         \item $\mathfrak{X}_{l+j}:= H\circ\zeta_{l+j}^{-1}$ is a conformal conjugacy from $B_{j}$ to $S$ on $\D_{l+j}$, for indices $j\in\{1,\cdots,r\}$.
    \end{enumerate}
    Thus, $S$ is the desired conformal mating of the dynamical systems under consideration.
\end{proof}

\begin{remark}\label{invariant_subset_rem}\upshape
Define $\cT_i:=H(U_i)$, $i\in\{1,\cdots,l\}$, and $\mathcal{V}_j:=H(U_{l+j})$, $j\in\{1,\cdots,r\}$. Let $\mathcal{W}_i$ stand for the grand orbit of $\cT_i$ under $S$, for $i\in\{1,\cdots,l\}$, and the grand orbit of $\mathcal{V}_{i-l}$ under $S$, for $i\in\{l+1,\cdots,l+r\}$. Finally, let $\cK := \widehat{\mathbb{C}}\backslash \bigcup_{i=1}^{l} \mathcal{W}_i$.
    
    The tessellation induced by the group dynamics of $\Gamma_i$ on $\mathbb{D}$ can be transferred via appropriate conjugacies to $\cT_i \subset \mathcal{W}_i$, $i\in\{1,\cdots,l\}$. The totally invariant sets $\displaystyle\bigcup_{i = 1}^{l} \mathcal{W}_i$ and $\cK$, which are the \emph{escaping} and \emph{non-escaping} sets of $S$ (respectively), partition the Riemann sphere. This partition will be of importance in Section~\ref{correspondences_sec}.
\end{remark}

\section{Algebraicity of conformal matings, and hyperelliptic surfaces}\label{b_involutions_sec}

Here we introduce a class of maps we call \emph{weak boundary involutions} (abbreviated as weak B-involutions). These maps subsume a class of complex-analytic maps called \emph{B-involutions} introduced in \cite{luo2024general}. Weak boundary involutions appear naturally as conformal matings of factor Bowen-Series maps and Blaschke products produced in Theorem~\ref{conf_mating_thm}. We will give a complete algebraic description of such maps with the intention of promoting the conformal mating discussed in this paper to algebraic correspondences.

\subsection{Weak B-involution}\label{weak_b_inv_def_subsec}

\begin{definition}\label{weak_b_inv_def}
    Let $\{\Omega_1, \cdots, \Omega_k\}$ be disjoint finitely connected proper subdomains of $\widehat{\mathbb{C}}$ with $\text{int}(\overline{\Omega_i}) = \Omega_i$ for $i\in \{1, 2, \cdots, k\}$. Set $\mathcal{D}:= \bigsqcup_{i=1}^k \Omega_i$. Further, let $X\subset \partial \cD$ be a finite set such that $\partial^0\cD := \partial\cD\backslash X$ is a finite union of disjoint non-singular real analytic curves. For notational convenience, let us also set $\partial^0 \Omega_i:= \partial^0\cD \cap \overline{\Omega}_i$. We call $\cD$ an \textbf{inversive multi-domain} if there exists a continuous map $S: \overline{\cD} \rightarrow \widehat{\C}$ satisfying the following:
    \begin{enumerate}
        \item\label{cond_mero} $S$ is meromorphic on $\cD$,
        \item\label{cond_inv} $S:(\partial\cD,X)\to (\partial\cD,X)$ is an involution, and
        \item\label{cond_or} $S: \partial^0 \cD \rightarrow \partial^0 \cD$ is orientation-reversing.
    \end{enumerate}
    The map $S$ is called a \textbf{weak B-involution} of the inversive multi-domain $\cD$. If $k= 1$, the domain $\cD$ is called an \textbf{inversive domain}. Set $T:= \widehat{\C}\backslash \cD$ and $T^{0}:= T\backslash X$.
\end{definition}

\begin{remark}\label{B_involution_rem}
\upshape
Definition~\ref{weak_b_inv_def} generalizes the class of maps called \emph{B-involutions}, which were introduced and studied in \cite[\S 14]{luo2024general}.
The key difference between B-involutions a la \cite{luo2024general} and our notion of weak B-involutions lies in the subtle fact that Definition~\ref{weak_b_inv_def} only requires $S$ to be orientation-reversing on the desingularized boundary $\partial^0\cD$ of its domain of definition; while a B-involution in the sense of \cite{luo2024general}
is also required to reverse the cyclic order of the local real-analytic arcs of $\partial^0\cD$ meeting at a singular point in $X$. Further, we drop the requirement that $S$ carries the boundary of some component $\Omega_i$ to that of some component $\Omega_{i'}$, allowing for richer boundary behavior. Finally, we relax the condition that each $\Omega_i$ is simply connected. 

All of the above relaxations are dictated by our general combination procedure for multiple Fuchsian groups with multiple Blaschke products (see Section~\ref{examples_sec}).
\end{remark}

\begin{remark}
    A conformal mating $S$ in the sense of Theorem~\ref{conf_mating_thm} is meromorphic on the interior $\cD$ of its domain of definition. Note that by Remark~\hyperref[BSmap_domain_rem]{\ref*{BSmap_domain_rem}~{\upshape(ii)}}, a (factor) Bowen-Series map is an orientation-reversing involution on (the closure of) its inner boundary, fixing the set of ideal vertices. And since $S$ is conjugate to the dynamics of (factor) Bowen-Series maps on appropriate regions, $S$ satisfies Conditions~\eqref{cond_inv} and~\eqref{cond_or} in Definition~\ref{weak_b_inv_def}. Therefore, such maps readily show up as instances of weak B-involutions on inversive multi-domains. 
\end{remark}

\subsection{Characterization of weak B-involutions}\label{char_weak_b_inv_subsec}
We shall now establish the algebraicity of general weak $B$-involutions.  

\begin{theorem}\label{alg_weak_b_inv_thm}
    Let $\cD$ be an inversive multi-domain with a weak B-involution $S:\overline{\cD} \rightarrow \widehat{\C}$. Then there exist
    \begin{itemize}
        \item $\Sigma$, a finite union of compact Riemann surfaces,
        \item $\mathfrak{D}\subset\Sigma$, a disjoint union of finitely connected domains with a piecewise non-singular real-analytic boundary,
        \item a conformal involution $\eta: \Sigma\rightarrow\Sigma$ with $\eta(\mathfrak{D})=\Sigma\backslash\overline{\mathfrak{D}}$, and
        \item a holomorphic map $\mathcal{R}:\Sigma \rightarrow \widehat{\C}$ that carries $\mathfrak{D}$ conformally onto $\cD$,        
    \end{itemize}
    such that 
    \begin{equation}\label{alg_char_eq}
        S\big|_{\cD} \equiv \mathcal{R}\circ \eta\circ \big(\mathcal{R}\big|_{\mathfrak{D}}\big)^{-1}.
    \end{equation}
\end{theorem}

\begin{remark}\upshape

1) For the restricted class of weak B-involutions considered in \cite{luo2024general}, an analogous algebraic description was given in \cite[Theorem~14.5]{luo2024general}, where the surface $\Sigma$ turned out to be a disjoint union of Riemann spheres. On the other hand, in our relaxed setting, the surface $\Sigma$ can have arbitrary genus even if $\cD$ is connected and simply connected (see Section~\ref{examples_sec}).

2) Let us also point out that an analogous construction exists in the antiholomorphic setting; namely, the Schwarz reflection map of a quadrature domain (an anti-meromorphic map of a domain that extends as the identity map on the boundary) admits an algebraic description in terms of a reflection map and a meromorphic function on a compact Riemann surface \cite{Gus83}. However, the antiholomorphic setting is somewhat restricted: for a Schwarz reflection map on a simply connected quadrature domain, the associated meromorphic function is always defined on $\widehat{\C}$.
\end{remark}

\begin{proof}
We shall prove this statement in steps (cf. \cite[Lemma~14.3]{luo2024general}).
\medskip
    
\noindent\textbf{Step 1: Welding a Riemann Surface.}    
    Observe that $S:\partial \cD\rightarrow \partial \cD$ is a homeomorphism that is orientation-reversing on $\partial^0\cD$. Further, $S$ is meromorphic on $\cD$ and extends continuously to $\partial^0\cD$ preserving $\partial^0\cD$ . Hence, by Schwarz reflection principle $S$ extends to a holomorphic map in a neighborhood of $\partial^0 \cD$. As $S^{\circ 2}\big|_{\partial^0\cD} = \mathrm{id}$, by the Identity Theorem, $S$ is a conformal involution in a neighborhood of $\partial^0 \cD$. We now weld $\cD \cup \partial^0 \cD$ with a copy of itself via the conformal involution $S: \partial^0 \cD\rightarrow \partial^0 \cD$ (cf. \cite[Chapter~2, Page~117]{ahlfors2015riemann}). The resulting object is a finite union of bordered Riemann surfaces of finite genus (see Figure~\ref{top_weld_surf_1_fig}). We denote this surface by $\Sigma^0$.
\begin{figure}
\captionsetup{width=0.98\linewidth}
    \centering
    \includegraphics[width=0.88\linewidth]{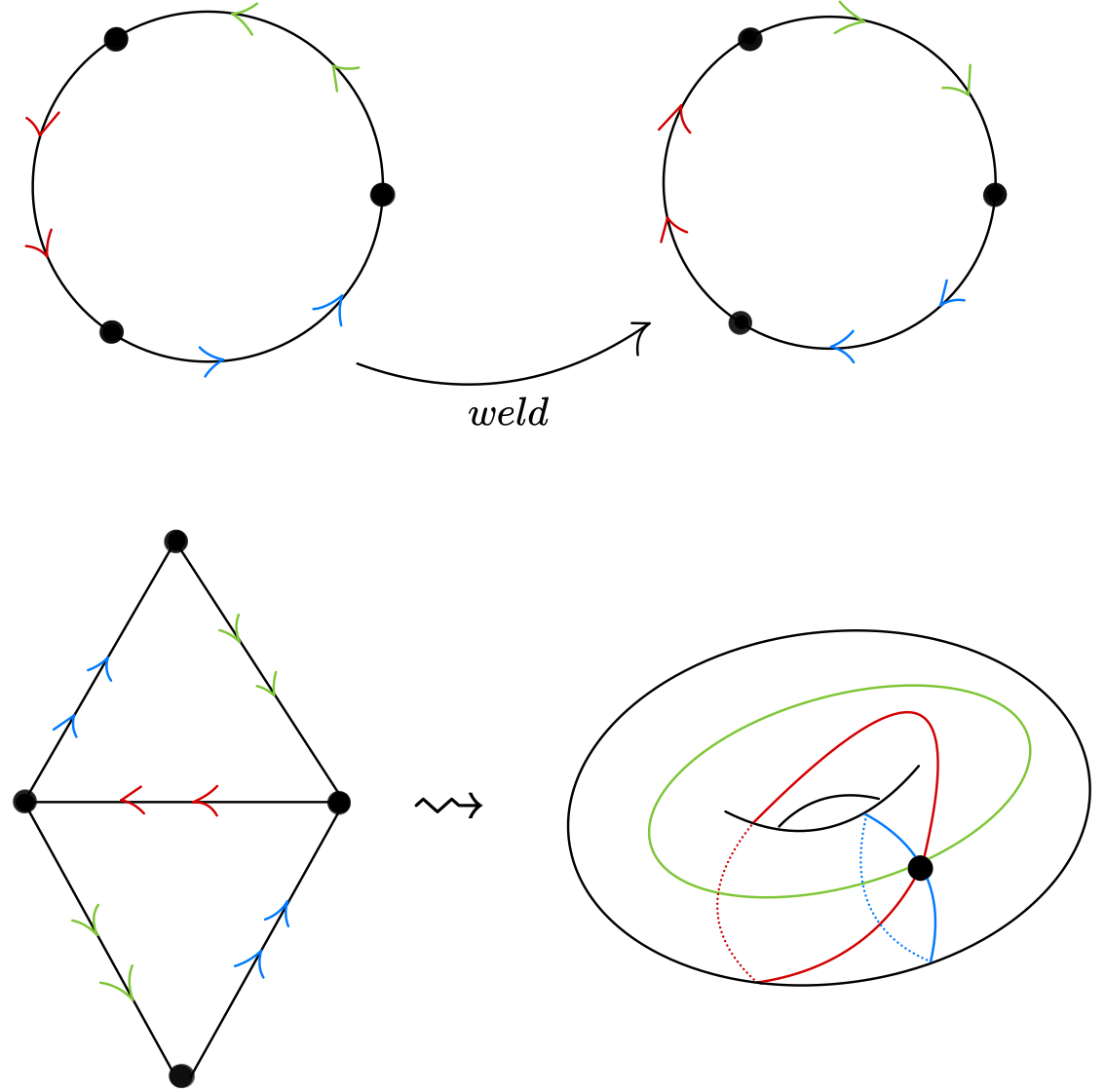}
    \caption{This illustration corresponds to the inversive domain $\cD$ constructed in Example~\ref{111'''_ex} (cf. Figure~\ref{111'''_fig}). The boundary of the domain $\cD$ is a wedge of three circles. Under the uniformization of the domain $\cD$ by a round disk, the single meeting point of the three circles opens up to three points labeled in black (top). Two copies of this disk are then welded via the lift of the associated B-involution $S:\partial\cD\to\partial\cD$. Finally, the topology of the surface one gets from such a welding procedure is illustrated (bottom).}
    \label{top_weld_surf_1_fig}
\end{figure}
\medskip

\noindent\textbf{Step 2: Conformal Involution $\eta$ on $\Sigma^0$.}   
The welding of $\cD\cup \partial^0\cD$ with a copy of itself to get $\Sigma^0$ comes with maps $\psi_1$, $\psi_2: \cD\cup \partial^0\cD \rightarrow \Sigma^0$, such that, 
\begin{itemize}
    \item $\psi_i$ is conformal in $\cD$ and a homeomorphism on $\cD\cup \partial^0 \cD$, $i\in\{1,2\}$, 
    \item $\psi_1(\cD)\cap \psi_2(\cD) = \emptyset$,
    \item $\psi_1(\partial^0 \cD) = \psi_2(\partial^0 \cD)$ is a union of non-singular real analytic curves,
    \item $\psi_2^{-1}\circ\psi_1 \equiv S$ on $\partial^0\cD$, and
    \item $\Sigma^0 = \psi_1(\cD\cup \partial^0\cD) \cup \psi_2 (\cD)$.
\end{itemize}
\vspace{1mm}
Define $\eta: \Sigma^0\rightarrow \Sigma^0$ as
\begin{equation*}
    \eta:= \begin{cases*}
      \psi_2\circ \psi_1^{-1} & on $\psi_1(\cD\cup \partial^0\cD)$ \\
      \psi_1\circ \psi_2^{-1}        & on $\psi_2(\cD)$.
    \end{cases*}
\end{equation*}
Since $\psi_1 = \psi_2\circ S$ on $\partial^0\cD$ and $S^{\circ 2} \equiv \mathrm{id}$ on $\partial^0\cD$, we have that $\eta$ is a conformal involution on $\Sigma^0$.
\medskip

\noindent\textbf{Step 3: $\Sigma^0$ is a finite type surface.}
Define the meromorphic map, 
$$\mathcal{R}: \Sigma^0\backslash \psi_2((S\vert_{\cD})^{-1}(X)) \rightarrow \widehat{\mathbb{C}}\backslash X$$
as follows, 
\begin{equation*}
    \mathcal{R}:= \begin{cases*}
        \psi_1^{-1} & on $\psi_1(\cD\cup \partial^0\cD)$ \\
        S\circ \psi_1^{-1}\circ \eta = S\circ \psi_2^{-1} & on $\psi_2(\cD\backslash (S\vert_\cD)^{-1}(X))$
    \end{cases*}
\end{equation*}
We first claim that $\mathcal{R}$ is a proper mapping. Indeed, $\psi_1$, $\psi_2$ being homeomorphisms are proper maps. As $S$ is continuous on $\mathcal{D}$, for any sequence of points $\{x_n\}_n$ converging to $\psi_2((S\vert_{\cD})^{-1}(X))$, the image sequence $\mathcal{R}(x_n)$ converges to $X$. Now suppose that the sequence $\{x_n\}_n$ escapes to a boundary component of $\Sigma^0$, then $\psi_2^{-1}(x_n)$ converges to a point in $X$. Since $S\vert_{\overline{\cD}}$ is continuous carrying $X$ to itself, it follows that $\mathcal{R}(x_n) = S(\psi_2^{-1}(x_n))$ converges to a point in $S(X) = X$.
This shows that $\mathcal{R}$ is a proper holomorphic map. 

Suppose that $\text{deg}(\mathcal{R}) = \infty$. Then, for points $z\in \widehat{\mathbb{C}}\backslash X$, the pre-image $\mathcal{R}^{-1}(z)$ is an infinite set in $\Sigma^0\backslash \psi_2((S\vert_{\cD})^{-1}(X))$. By the Identity Theorem, as $\mathcal{R}$ is non-constant, $\mathcal{R}^{-1}(z)$ accumulates either on $\psi_2((S\vert_{\cD})^{-1}(X))$ or on $\partial \Sigma^0$ which violates the fact that $\mathcal{R}$ is proper. Hence, $\mathcal{R}$ is a finite degree branched cover from $\Sigma^0\backslash \psi_2((S\vert_{\cD})^{-1}(X))$ to $\widehat{\mathbb{C}}\backslash X$. 

We make further punctures in $\widehat{\mathbb{C}}\backslash X$ (if necessary) to ensure that $\mathcal{R}$ is a covering map over the resulting punctured sphere. We then triangulate this punctured sphere, and lift the triangulation via the covering map $\mathcal{R}$ to a triangulation upstairs. This would yield a finite triangulation on the surface $\Sigma^0\backslash \psi_2((S\vert_{\cD})^{-1}(X))$, with finitely many points removed. This would force $\Sigma^0\backslash \psi_2((S\vert_{\cD})^{-1}(X))$ to be a surface of finite type. This implies that $(S\vert_{\cD})^{-1}(X)$ is a finite set, and hence $\Sigma^0$ is a finite union of punctured surfaces of finite genus. We will denote the compactification of $\Sigma^0$ (which is obtained by filling in the finitely many punctures) by $\Sigma$.
\medskip

\noindent\textbf{Step 4: Explicit Algebraic Description of $S$.}
By the Riemann Removability Theorem, the maps $\mathcal{R}$ and $\eta$ extend holomorphically to the compact Riemann surface $\Sigma$. We remark that if $\Sigma$ is the Riemann sphere, one can after a change of coordinates assume that $\eta(z)= 1/z$. Set $\mathfrak{D} = \psi_1(\cD)$. Then from the definition of $\mathcal{R}$ one obtains, 
$$
\mathcal{R}\Big|_{\Sigma\backslash \overline{\mathfrak{D}}} \equiv S\Big|_{\cD}\circ \mathcal{R}\Big|_{\mathfrak{D}}\circ \eta\Big|_{\Sigma\backslash \overline{\mathfrak{D}}}
$$
which yields, 
$$
S\Big|_{\cD} = \mathcal{R}\circ \eta\circ \Big(\mathcal{R}\Big|_{\mathfrak{D}}\Big)^{-1}.
$$
The facts that $\eta(\mathfrak{D})=\Sigma\backslash\overline{\mathfrak{D}}$ and
 $\mathcal{R}:\mathfrak{D}\to\cD$ is a conformal isomorphism follow directly from the construction. 
\end{proof}

\begin{cor}\label{b_inv_fin_deg_cor}
With the setup as in Theorem~\ref{alg_weak_b_inv_thm}, we have that $S: S^{-1}(\cD) \rightarrow \cD$ is a branched covering of degree $d_{\mathcal{R}}-1$, and  $S: S^{-1}(\text{int}(T)) \rightarrow \text{int}(T)$ is a branched covering of degree $d_{\mathcal{R}}$, where $d_{\mathcal{R}}$ is the degree of the meromorphic map $\mathcal{R}:\Sigma\to\widehat{\C}$.
\end{cor}
\begin{proof}
The statements follow from the relation $S\Big|_{\cD} = \mathcal{R}\circ \eta\circ \Big(\mathcal{R}\Big|_{\mathfrak{D}}\Big)^{-1}$. 
\end{proof}

\begin{cor}\label{crit_points_cor}
    Let the weak B-involution $S:\overline{\cD}\to\widehat{\C}$ be a conformal mating of a sub-hyperbolic map $R$, the Blaschke products $\{B_j\}_{j=1}^r$ and the (factor) Bowen-Series maps $\{A_{\Gamma_i}\}_{i=1}^l$ associated with the representations $\rho:\Gamma_{n_i, p_i}\rightarrow \Gamma_i$ (constructed in Theorem~\ref{conf_mating_thm}). Let $\cT_i$ be as in Remark~\ref{invariant_subset_rem}. 
        For indices $i \in \{1, 2, \cdots, l\}$ with $n_i\geq 3$, the map $\cR$ has a critical value in $\cT_i$. Moreover, that critical value admits at least $p_i$ distinct pre-images under $\mathcal{R}$, each with multiplicity $n_i-1$.
\end{cor}
\begin{proof}
    Note that for indices $i \in \{1, 2, \cdots, l\}$ with $n_i\geq 3$, the factor Bowen-Series map $A_{\Gamma_{n_i, p_i}}$  has a unique critical value in $\mathbb{D}$ at $0$ with $p_i$ distinct pre-images each of multiplicity $n_i$; and hence so does $A_{\Gamma_i}$. Recall from Theorem~\ref{conf_mating_thm} that the map $S:\cT_i\cap\overline{\cD}\to\cT_i$  is conjugate to $A_{\Gamma_i}:\overline{\cU}_{\Gamma_i}\cap\D\to\D$, for all $i\in \{1, 2, \cdots, l\}$. The result now follows directly from Relation~\eqref{alg_char_eq}  (cf. \cite[Proposition~4.14]{mj2023matings}). 
\end{proof}

\begin{remark}
It is easy to see that $(\Sigma,\mathfrak{D},\eta,\mathcal{R})$ satisfying the conditions of Theorem~\ref{alg_weak_b_inv_thm} gives rise to a weak B-involution $S:\overline{\cD}\to\widehat{\C}$, where $\cD=\mathcal{R}(\mathfrak{D})$ and $S\big|_{\cD} \equiv \mathcal{R}\circ \eta\circ \big(\mathcal{R}\big|_{\mathfrak{D}}\big)^{-1}$. In fact, one can check that the quadruples $(\Sigma,\mathfrak{D},\eta,\mathcal{R})$ and the pairs $(S,\cD)$ are in bijective correspondence, up to appropriate equivalences.
\end{remark}

\begin{definition}
    For a weak B-involution $S:\overline{\cD}\to\widehat{\C}$ arising from the conformal mating construction of Theorem~\ref{conf_mating_thm}, the associated welded Riemann surface $\Sigma$ (defined in Theorem~\ref{alg_weak_b_inv_thm}) will be called the \emph{blender surface}. We call the map $\cR$, also defined in Theorem~\ref{alg_weak_b_inv_thm}, the \emph{uniformizing meromorphic map}.
\end{definition}

\subsection{Welding graph}\label{weld_graph_subsec}

We will now introduce a combinatorial object that records the connected components of the welded surface $\Sigma$ and the action of the involution $\eta$ on these components.

\begin{definition}\label{weld_graph_defn}
    Given a weak B-involution $S:\overline{\cD}\to\widehat{\C}$, with $\cD:= \bigsqcup_{i=1}^k \Omega_i$, let $\mathcal{G}$ be a graph with vertices $\{v_i^\pm\}_{i=1}^k$, where $v_{i_1}^-$, $v_{i_2}^+$ share an edge iff $S(\partial^0\Omega_{i_1})\cap\partial^0\Omega_{i_2} \neq \emptyset$. The graph $\mathcal{G}$ will be called the \textbf{welding graph} for $S:\overline{\cD}\to\widehat{\C}$.
\end{definition}
\noindent(See Section~\ref{examples_sec} for illustrations.)

\begin{lem}\label{graph_symm_lem}
There exists an edge between $v_{i_1}^-$ and $v_{i_2}^+$ if and only if there exists an edge between $v_{i_1}^+$ and $v_{i_2}^-$. 
\end{lem}
\begin{proof}
Since $S:\partial\cD\to\partial\cD$ is an involution, we have that 
\begin{align*}
&S(\partial^0\Omega_{i_1})\cap\partial^0\Omega_{i_2}\neq\emptyset\\
\iff & S(\partial^0\Omega_{i_2})\cap\partial^0\Omega_{i_1}\neq\emptyset .
\end{align*}
The result now follows from the definition of $\cG$.
\end{proof}

The vertices $v_i^\pm$ of $\cG$ correspond to the copies $\Omega_i^\pm$ (sitting inside $\Sigma$) of the component $\Omega_i$ of $\cD$. 
By definition of $\Sigma$ and $\cG$, the domains $\Omega_{i_1}^-$ and $\Omega_{i_2}^+$ share parts of their boundary in $\Sigma$ precisely when there is an edge in $\cG$ connecting $v_{i_1}^-$ and $v_{i_2}^+$.  
The next statement follows from the above discussion. 

\begin{lem}\label{graph_conn_lem}
The connected components of $\Sigma$ are in bijective correspondence with those of $\cG$.   
\end{lem}

Observe that there are two types of components of $\Sigma$: components that are preserved by the involution $\eta$, and components that lie in a $2-$cycle under $\eta$. This dichotomy is encoded in the graph $\cG$. We now proceed to explicate this connection.

For $i\in\{1,\cdots,k\}$, the involution $\eta$ maps $\Omega_i^-\subset\Sigma$ onto $\Omega_i^+\subset\Sigma$. This induces an action 
\begin{align*}
&\widehat{\eta}:V(\cG)\to V(\cG),\\
& \widehat{\eta}(v_i^-)=v_i^+,
\end{align*}
where $V(\cG)$ denotes the vertex set of the graph $\cG$.
Let us consider a component $\mathcal{C}_\Sigma$ of $\Sigma$. We denote the corresponding component of $\cG$ by $\mathcal{C}_\cG$.
Suppose that 
$$
V(\mathcal{C}_\cG)=\{v_{i_1}^-,\cdots,v_{i_\alpha}^-,v_{j_1}^+,\cdots,v_{j_\beta}^+\}.
$$
This means, using the notation from Theorem~\ref{alg_weak_b_inv_thm}, that 
$$
\mathcal{C}_\Sigma\setminus\partial\mathfrak{D}=\Omega_{i_1}^-\sqcup\cdots\sqcup\Omega_{i_\alpha}^-\sqcup\Omega_{j_1}^+\sqcup\cdots\sqcup\Omega_{j_\beta}^+.
$$
By Lemma~\ref{graph_symm_lem}, there exists a symmetric component $\mathcal{C}'_\cG$ of $\cG$ with
$$
V(\mathcal{C}'_\cG)=\{v_{j_1}^-,\cdots,v_{j_\beta}^-,v_{i_1}^+,\cdots,v_{i_\alpha}^+\}=\widehat{\eta}(V(\mathcal{C}_\cG)).
$$
(See Figure~\ref{weld_graph_fig} for an illustration.)
We denote the corresponding component of $\Sigma$ by $\mathcal{C}'_\Sigma$, and note that $\eta(\mathcal{C}_\Sigma)=\mathcal{C}_\Sigma'$.

\begin{lem}\label{graph_eta_lem}
With notation as above, we have 
\begin{align*}
& \eta(\mathcal{C}_\Sigma)=\mathcal{C}_\Sigma\\
\iff & \widehat{\eta}(\mathcal{C}_\cG)=\mathcal{C}_\cG\\
\iff & \{i_1,\cdots,i_\alpha\}=\{j_1\cdots,j_\beta\}\\
\iff &\{i_1,\cdots,i_\alpha\}\cap\{j_1\cdots,j_\beta\}\neq\emptyset.
\end{align*}
\end{lem}
\begin{figure}[ht!]
\captionsetup{width=0.98\linewidth}
\includegraphics[width=0.96\linewidth]{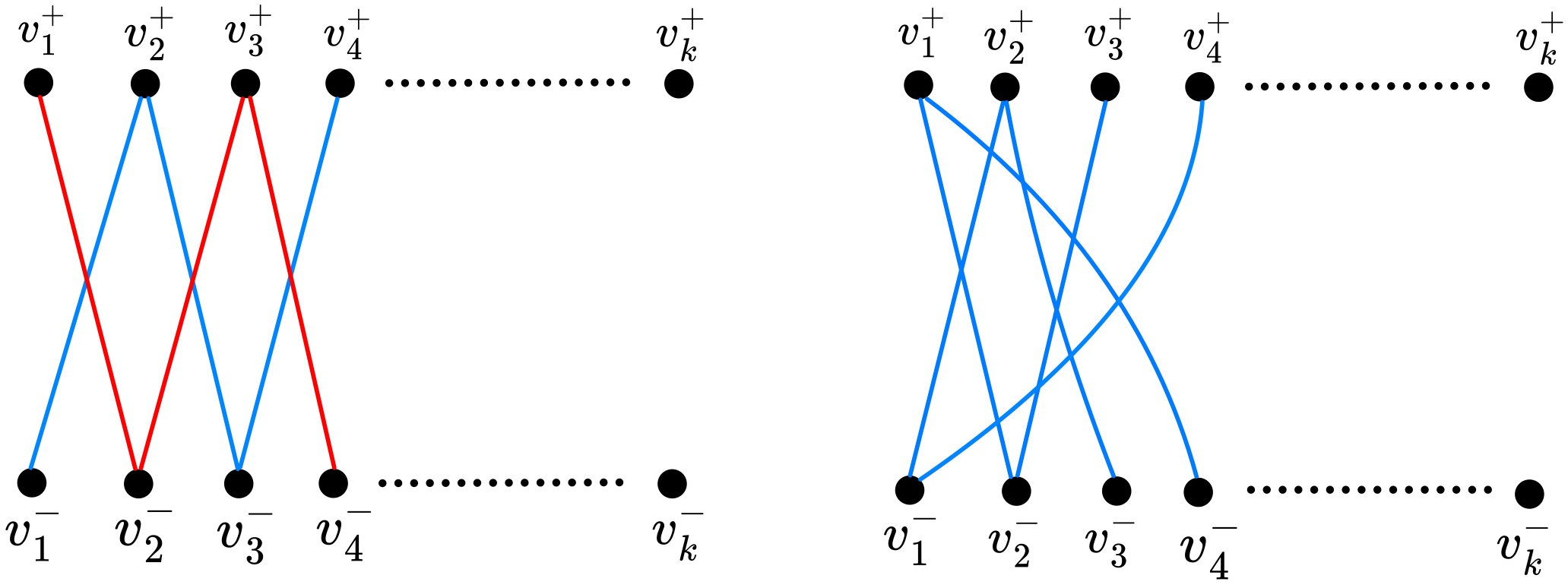}
\caption{Left: Depicted are two disjoint, symmetric components $\mathcal{C}_\cG, \mathcal{C}_\cG'$ of a welding graph $\cG$. The vertices of $\mathcal{C}_\cG$ are mapped to those of $\mathcal{C}_\cG'$ under the involution $\widehat{\eta}$. Right: Depicted is a self-symmetric component $\mathcal{C}_\cG$ of a welding graph $\cG$. The induced map $\widehat{\eta}$ preserves the vertex set of $\mathcal{C}_\cG$.}
\label{weld_graph_fig}
\end{figure}
\begin{proof}
The equivalence of the first two statements is obvious from the definition of $\cG$ and the induced map $\widehat{\eta}$.

Suppose that $\widehat{\eta}(\mathcal{C}_\cG)=\mathcal{C}_\cG$. Since $\widehat{\eta}(v_i^\pm)=v_i^\mp$, it follows that $v_i^\pm\in V(\mathcal{C}_\cG)$ if and only if $v_i^\mp\in V(\mathcal{C}_\cG)$. Thus, $\{i_1,\cdots,i_\alpha\}=\{j_1\cdots,j_\beta\}$.

The third statement trivially implies the fourth.

Finally, let $i\in\{i_1,\cdots,i_\alpha\}\cap\{j_1\cdots,j_\beta\}$. Then $\Omega_i^-, \Omega_i^+\subset\mathcal{C}_\Sigma$. As $\eta$ carries $\Omega_i^-$ to $\Omega_i^+$, we have $\eta(\mathcal{C}_\Sigma)\cap\mathcal{C}_\Sigma\neq\emptyset$. But the homeomorphism $\eta$ must carry a component of $\Sigma$ to a component of $\Sigma$, and hence $\eta(\mathcal{C}_\Sigma)=\mathcal{C}_\Sigma$.
\end{proof}

\subsection{Hyperellipticity of blender surfaces}\label{blender_surf_subsec}
For a weak B-involution $S:\overline{\cD} \rightarrow \widehat{\C}$, let $\Sigma, \mathfrak{D}, \eta$ be as in Theorem~\ref{alg_weak_b_inv_thm}. Note that the involution $\eta$ gives rise to a degree two branched covering $\Sigma\to\Sigma/\langle\eta\rangle$ of which $\eta$ is a deck transformation. We will now give a more direct description of the quotient surface $\Sigma/\langle\eta\rangle$ in terms of the  weak B-involution $S:\overline{\cD} \rightarrow \widehat{\C}$. We intend to use this description to show that the welded surfaces $\Sigma$ associated to the weak B-involutions that arise from our conformal mating procedure (see Theorem~\ref{conf_mating_thm}) are hyperelliptic.

\subsubsection{An auxiliary surface}
Consider a inversive multi-domain $\mathcal{D}$. The map $S$ is a conformal involution in a neighborhood of $\partial^0\cD$; this gives a pairing of the connected components of $\partial^0\mathcal{D}$. We construct a Riemann surface from $\cD\cup\partial^0\cD$ by welding the paired boundary components via $S$. The resulting surface would be a finite union of bordered Riemann surfaces of finite genus. Let us denote this surface by $\widecheck{\Sigma}^0$.

\begin{theorem}\label{zipped_surface_thm}
    Let $(\Sigma, \widecheck{\Sigma}^0, \eta)$ be as above. Then the following hold.
    \begin{enumerate}\upshape
        \item $\widecheck{\Sigma}^0$ is a finite type surface.
        \item $\widecheck{\Sigma} \cong \Sigma/\langle\eta\rangle$, where $\widecheck{\Sigma}$ is the compactification of the surface $\widecheck{\Sigma}^0$ (obtained by filling in the punctures). In particular, the components of $\widecheck{\Sigma}$ correspond bijectively to the $\eta-$orbits of the components of $\Sigma$.
    \end{enumerate}
\end{theorem}
\begin{proof}
Recall that $\Sigma^0$ stands for the finite type surface obtained by gluing two copies of $\cD\cup\partial^0\cD$ via the boundary involution $S$, and $\Sigma$ is its compactification. Recall also the welding maps $\psi_1$ and $\psi_2$; the conformal embeddings of $\mathcal{D}\cup\partial^0\cD$ into $\Sigma$. Further, we have $\mathfrak{D}=\psi_1(\cD)$ and $\eta(\mathfrak{D})=\psi_2(\cD)$.
We will denote the boundary of $\mathfrak{D}$ in the punctured surface $\Sigma^0$ by $\partial^0\mathfrak{D}$, so that $\psi_1(\partial^0\cD)=\psi_2(\partial^0\cD)=\partial^0\mathfrak{D}$. The boundary of $\mathfrak{D}$ in the compact surface $\Sigma$ will be denoted by the usual notation $\partial\mathfrak{D}$. Then, we have $\Sigma = \mathfrak{D} \cup \eta(\mathfrak{D})\cup \partial\mathfrak{D}$, and $\Sigma^0 = \mathfrak{D} \cup \eta(\mathfrak{D})\cup \partial^0\mathfrak{D}$.

1) Let $\phi$ be the quotient map from $\mathcal{D}\cup \partial^0\mathcal{D}$ onto $\widecheck{\Sigma}^0$, coming from the welding construction of $\widecheck{\Sigma}^0$. It is evident that the map $\phi|_{\mathcal{D}}$ is a conformal embedding. 
Define the map $\mathcal{Q}: \Sigma^0 \rightarrow \widecheck{\Sigma}^0$ as
    \begin{equation*}
    \mathcal{Q}\equiv \begin{cases*}
      \phi\circ \psi_1^{-1} & on $\mathfrak{D}\cup \partial^0\mathfrak{D}$, \\
      \phi\circ \psi_2^{-1}        & on $\eta(\mathfrak{D})$.
    \end{cases*}
\end{equation*}
The map $\mathcal{Q}$ is clearly continuous on the (real-analytic) common boundary $\partial^0\mathfrak{D}$ because $\phi\equiv \phi\circ S$ on $\partial^0\cD$. This shows that $\mathcal{Q}$ is a degree $2$ holomorphic map from the finite type surface $\Sigma^0$ onto $\widecheck{\Sigma}^0$, and hence $\widecheck{\Sigma}^0$ is a finite type surface. 

2) By continuity of $\mathcal{Q}$, punctures in $\Sigma^0$ go to punctures in $\widecheck{\Sigma}^0$. This gives us an extended map $\mathcal{Q}:\Sigma \rightarrow \widecheck{\Sigma}$, where $\widecheck{\Sigma}$ is the compactification of $\widecheck{\Sigma}^0$ (by abuse of notation, we shall reuse the name $\mathcal{Q}$). Note that $\mathcal{Q}\circ\eta \equiv \mathcal{Q}$. In other words, the involution $\eta$ generates the group of deck transformations for $\mathcal{Q}$. Therefore, $\widecheck{\Sigma} \cong \Sigma/\langle\eta\rangle$.    
\end{proof}

\begin{remark}
If $\cD$ is an inversive domain (i.e., it is connected), then the associated welding graph $\cG$ and hence the welded surface $\Sigma$ as well as the auxiliary surface $\widecheck{\Sigma}$ are also connected. However, in general, the surfaces $\widecheck{\Sigma}$ and $\Sigma$ can be connected even if $\cD$ is disconnected (see Section~\ref{examples_sec}).
\end{remark}

\subsubsection{Components of blender surfaces are spheres or hyperelliptic}
We will now see that for a weak B-involution $S:\overline{\cD}\to\widehat{\C}$ constructed via the conformal mating procedure of Theorem~\ref{conf_mating_thm}, there are strong restrictions on the topology of the auxiliary surface $\widecheck{\Sigma}$. Indeed, regardless of the connectivities of the components of $\cD$, the components of $\widecheck{\Sigma}$ are necessarily simply connected. This, in turn, forces the blender surface $\Sigma$ to be made up of copies of Riemann spheres and hyperelliptic Riemann surfaces.

Recall from Theorem~\ref{zipped_surface_thm} that there exists a degree two meromorphic map $\mathcal{Q}:\Sigma\to\widecheck{\Sigma}\cong\Sigma/\langle\eta\rangle$ such that $\mathcal{Q}\circ\eta\equiv \mathcal{Q}$.

\begin{theorem}
\noindent\begin{enumerate}[leftmargin=8mm]\upshape
    \item For weak B-involutions arising from the conformal mating construction of Theorem~\ref{conf_mating_thm}, the surface $\widecheck{\Sigma}$ is a disjoint union of  Riemann spheres $\widehat{\C}$.
    \item Each component of the blender surface $\Sigma$ is either
    \begin{enumerate}
        \item a copy of the Riemann sphere that is mapped conformally to $\widehat{\C}$ by $\mathcal{Q}$, or
        \item a hyperelliptic Riemann surface (possibly the Riemann sphere) that is mapped as a degree $2$ branched covering to $\widehat{\C}$ by $\mathcal{Q}$.
    \end{enumerate} 
\end{enumerate}
\label{hyperelliptic_thm}
\end{theorem}

\begin{proof} 
1) Let $S:\overline{\cD}\to\widehat{\C}$ be a weak B-involution on an inversive multi-domain $\mathcal{D}$ realizing a conformal mating in the sense of Theorem~\ref{conf_mating_thm}. Let $V_1, V_2, \cdots, V_l$ be the components of $\widehat{\mathbb{C}}\setminus\overline{\cD}$. In other words, $\displaystyle T = \bigcup_{i = 1}^l\overline{V_i}$. 

We will now show that the process of passing from $\cD\cup\partial^0\cD$ to $\widecheck{\Sigma}^0$ amounts to zipping the holes $V_i$.
By the very nature of the (factor) Bowen-Series map construction, we see that $\partial V_i$ comes equipped with a line of symmetry, and $S\vert_{\partial V_i}$ is a topological reflection with respect to this line of symmetry. This line of symmetry connects (the projections of) two parabolic fixed points or order $2$ orbifold points of the corresponding Fuchsian group. Further, this line of symmetry divides $\partial V_i$ into $\partial V_i^+$ and $\partial V_i^-$. We define as follows an equivalence relation $\sim_{\text{zip}}$ on  $\widehat{\mathbb{C}}\setminus X$, where $X$ is the set of singular points on $\partial\Omega$. For $\displaystyle x, y \in \widehat{\mathbb{C}}\setminus X=\big(\mathcal{D}\cup\partial^0\mathcal{D}\big) \bigcup \big(\bigcup_{i = 1}^lV_i\big)$, we say that $x \sim_{\text{zip}} y$, if and only if for some $i\in \{1, 2, \cdots, l\}$, $x$ and $y$ belong to the closure of a hyperbolic geodesic of $V_i$ connecting some point $p \in \partial V_i^-\setminus X$ to $S(p)\in \partial V_i^+\setminus X$. Clearly, $\widecheck{\Sigma}^0 \cong \big(\widehat{\mathbb{C}}\backslash X\big)/\sim_{\text{zip}}$. It also follows from the above zipping constructing that for each $i\in\{1,\cdots,l\}$, the image of $\overline{V_i}\setminus X$ in $\big(\widehat{\mathbb{C}}\backslash X\big)/\sim_{\text{zip}}$ is the union of finitely many disjoint (possibly a single) simple arcs. This implies that $\big(\widehat{\mathbb{C}}\backslash X\big)/\sim_{\text{zip}}$ is a finite union of punctured spheres. Hence, its compactification $\widecheck{\Sigma}$ is a disjoint union of finitely many copies of Riemann spheres.

2) Recall from the discussion in Section~\ref{weld_graph_subsec} that there are two types of components of $\Sigma$: those that are preserved by the involution $\eta$, and those that lie in a $2-$cycle under $\eta$. If $\mathcal{C}_\Sigma$ is a component that lies in a $2-$cycle under $\eta$, then by definition, $\mathcal{Q}$ maps $\mathcal{C}_\Sigma \sqcup \eta(\mathcal{C}_\Sigma)$ to a component of $\widecheck{\Sigma}$, a Riemann sphere. Note that the degree of the map $\mathcal{Q}\big|_{\mathcal{C}_\Sigma}$ is $1$. This implies that both $\mathcal{C}_\Sigma$ and $\eta(\mathcal{C}_\Sigma)$ are Riemann spheres. Now let $\mathcal{C}_\Sigma$ be an $\eta$-invariant component of $\Sigma$. Then $\mathcal{Q}\big|_{\mathcal{C}_\Sigma}$ maps $\mathcal{C}_\Sigma$ to a component of $\widecheck{\Sigma}$, a Riemann sphere. This gives rise to a degree two meromorphic function on $\mathcal{C}_\Sigma$, forcing $\mathcal{C}_\Sigma$ to be hyperelliptic (possibly the Riemann sphere). 
\end{proof} 

\noindent The description of $\widecheck{\Sigma}$ given in Theorem~\ref{hyperelliptic_thm} motivates the following nomenclature.

\begin{definition}\label{zipped_def}
For a weak B-involution $S$ arising from the conformal mating construction of Theorem~\ref{conf_mating_thm}, the surface $\widecheck{\Sigma}$ will be called the \emph{zipped surface} associated with $S$.
\end{definition}

\begin{prop}\label{genus_prop}
Let $\displaystyle S:\overline{\cD} \to\widehat{\C}$, where $\cD:= \bigsqcup_{i=1}^k \Omega_i$, be a conformal mating produced in Theorem~\ref{conf_mating_thm} such that the associated welding graph $\cG$ is connected. Then $\Sigma$ is a compact Riemann surface of genus $g=\frac{\#\mathrm{Fix}(\eta)-2}{2}$, where $\mathrm{Fix}(\eta)$ denotes the set of fixed points of $\eta$. In particular, under the above assumptions, if $S$ has at least three fixed points on $\partial^0\cD$, then $g\geq 1$.
\end{prop}
\begin{proof}
By Theorems~\ref{zipped_surface_thm} and~\ref{hyperelliptic_thm}, the surface $\widecheck{\Sigma}$ is the Riemann sphere. Further, the critical points of the degree $2$ holomorphic map $\mathcal{Q}:\Sigma\to\Sigma/\langle\eta\rangle\cong\widecheck{\Sigma}\cong\widehat{\C}$ are precisely the fixed points of $\eta$. The result now follows from the Riemann-Hurwitz formula.

Recall that $\mathcal{R}:\partial^0\mathfrak{D}\to\partial^0\mathfrak{\cD}$ conjugates $\eta\vert_{\partial^0\mathfrak{D}}$ to $S\vert_{\partial^0\cD}$. Thus, if $S$ has at least three fixed points on $\partial^0\cD$, then $\eta$ has at least three fixed points on $\partial^0\mathfrak{D}$. The second statement now follows from this observation and the first part.
\end{proof}

\subsubsection{Constructing hyperelliptic blender surfaces}\label{construct_blender_subsec}
We now present an explicit construction of conformal matings $S:\overline{\cD}\to\widehat{\C}$ whose associated blender surfaces are connected, and hence hyperelliptic.

\begin{prop}\label{mating_surf_exists_prop}
    Let $\{\Gamma_i\}_{i=1}^l$ be Fuchsian groups (introduced in Section~\ref{mateable_maps_subsec}), with $A_{\Gamma_i}$ the associated (factor) Bowen-Series map for $\Gamma_i$, $i\in\{1,\cdots, l\}$. Let $\{B_j\}_{j=1}^r$ be hyperbolic Blaschke products. Then there exists a critically fixed polynomial $P$ with the following properties.
    {\upshape
    \begin{enumerate}[leftmargin=8mm]
        \item $P$, $\{A_{\Gamma_i}\}_{i=1}^l$, and $\{B_j\}_{j=1}^r$ can be conformally mated.
        \item Let $S$ be the conformal mating of $P$, $\{A_{\Gamma_i}\}_{i=1}^l$, and $\{B_j\}_{j=1}^r$.
        Then the corresponding blender surface $\Sigma$ is connected.
    \end{enumerate}
    }
\end{prop}

\begin{proof}
    1) Applying Lemma~\ref{combi_lem} to the set of positive integers 
    $$
   \displaystyle \{\text{deg}(A_{\Gamma_i})\vert_{\mathbb{S}^1}-1\}_{i=1}^l \bigcup \{\text{deg}(B_j)-1\}_{j=1}^r,
    $$ 
    we get a critically fixed polynomial $P$, whose $l+r$ finite critical points have the given multiplicities. By Theorem~\ref{conf_mating_thm}, $P$, $\{A_{\Gamma_i}\}_{i=1}^l$, and $\{B_j\}_{j=1}^r$ can be conformally mated. 

    2) Let $S: \overline{\mathcal{D}} \rightarrow \widehat{\mathbb{C}}$ be the associated weak $B$-involution, with $X$ the set of singular points on $\partial\cD$, and let $\widehat{\C}\setminus\overline{\mathcal{D}} = \bigcup_{i=1}^lV_i$.  We shall show that the corresponding blender surface $\Sigma$ is connected by showing that $\mathcal{D}$ is connected.
    Recall the David homeomorphism $H:\widehat{\C}\to\widehat{\C}$ from Theorem~\ref{conf_mating_thm}. By construction of $S$, each $V_i$ is contained in a Jordan domain $\mathfrak{U}_i:=H(U_i)$, where $U_i$ is a bounded Fatou component of $P$. Further, $\mathfrak{U}_i\cap\mathfrak{U}_j=\emptyset$ for $i\neq j$. It also follows from the construction of $S$ that each $V_i$ is a Jordan domain and $\partial V_i\cap\partial \mathfrak{U}_i$ is a subset of $X$. 
    By way of contradiction, suppose that $\mathcal{D}$ is disconnected. Then, there exist $k>1$ and a chain $V_{i_0}, V_{i_1}, \cdots, V_{i_{k-1}}$, such that $\partial V_{i_j}$ meets $\partial V_{i_{j+1}}$ at a point in $X$ for all $j\in\{0,\cdots,k-1\}$ (where $j+1$ is taken mod $k$). This implies that $\partial\mathfrak{U}_{i_j}$ touches $\partial\mathfrak{U}_{i_{j+1}}$ at a point in $X$ for all $j\in\{0,\cdots,k-1\}$ (where $j+1$ is taken mod $k$). But this contradicts the fact that the filled Julia set $K(P)$ of the polynomial $P$ is a full compact set and hence so is $H(K(P))$. Therefore, $\cD$ is connected.    
\end{proof}

\begin{cor}\label{orbifold_genus_cor}
   Given $\{\Gamma_i\}_{i = 1}^l$ and $\{B_j\}_{j=1}^r$, consider the weak B-involution $S: \overline{\mathcal{D}}\rightarrow \widehat{\mathbb{C}}$ and the blender surface $\Sigma$ constructed in Proposition~\ref{mating_surf_exists_prop}. Let $b$ be the number of orbifold points of order $2$ in $\displaystyle\bigsqcup_{i = 1}^l \mathbb{D}/\widehat{\Gamma}_i$, and $g$ the genus of the surface $\Sigma$. If $b\geq 3$, then $g\geq 1$.
\end{cor}

\begin{proof}
    By Proposition~\ref{genus_prop}, we see that if $S$ has at least three fixed points on $\partial^0\mathcal{D}$, then $g\geq 1$. Every orbifold point of order $2$ in $\mathbb{D}/\widehat{\Gamma}_i$, $i\in\{1,\cdots,l\}$, corresponds to a unique fixed point of $S$ on $\partial^0\mathcal{D}$. The corollary now follows from the assumption that~$b\geq 3$.  
\end{proof}

\section{A gallery of examples}\label{examples_sec}

In this section, we present a host of examples that guided our investigation leading to Theorem~\ref{mating_thm_intro}. We describe the conformal matings for several explicit examples of the rational map $R$, Fuchsian groups $\{\widehat{\Gamma}_i\}_{i=1}^l$, and hyperbolic Blaschke products $\{B_j\}_{j=1}^r$. More precisely, we invoke Theorem~\ref{conf_mating_thm} to get weak B-involutions $S:\overline{\cD}\rightarrow\widehat{\mathbb{C}}$, whose dynamics is equivalent to the dynamics of the constituents $R$, $\{A_{\Gamma_i}\}_{i=1}^l$, and $\{B_j\}_{j=1}^r$ on appropriate domains. We shall also describe the topology of the associated blender surface $\Sigma$ (arising from results in Section~\ref{b_involutions_sec}).  

Recall that in Section~\ref{mateable_maps_subsec}, for $p$ even, we used $\Gamma_{n, p}$ to denote two different groups corresponding to Cases~I and~II. Since the specific nature of the groups will play a role in this section, we reserve the notation $\Gamma_{n,p}$ for groups constructed as in Case~I, and the notation $\Gamma_{n,p}^2$ for groups constructed as in Case~II.

    \begin{figure}[h!]
    \captionsetup{width=0.98\linewidth}
    \begin{tikzpicture}
    \node[anchor=south west,inner sep=0] at (0,0) {\includegraphics[width=0.4\textwidth]{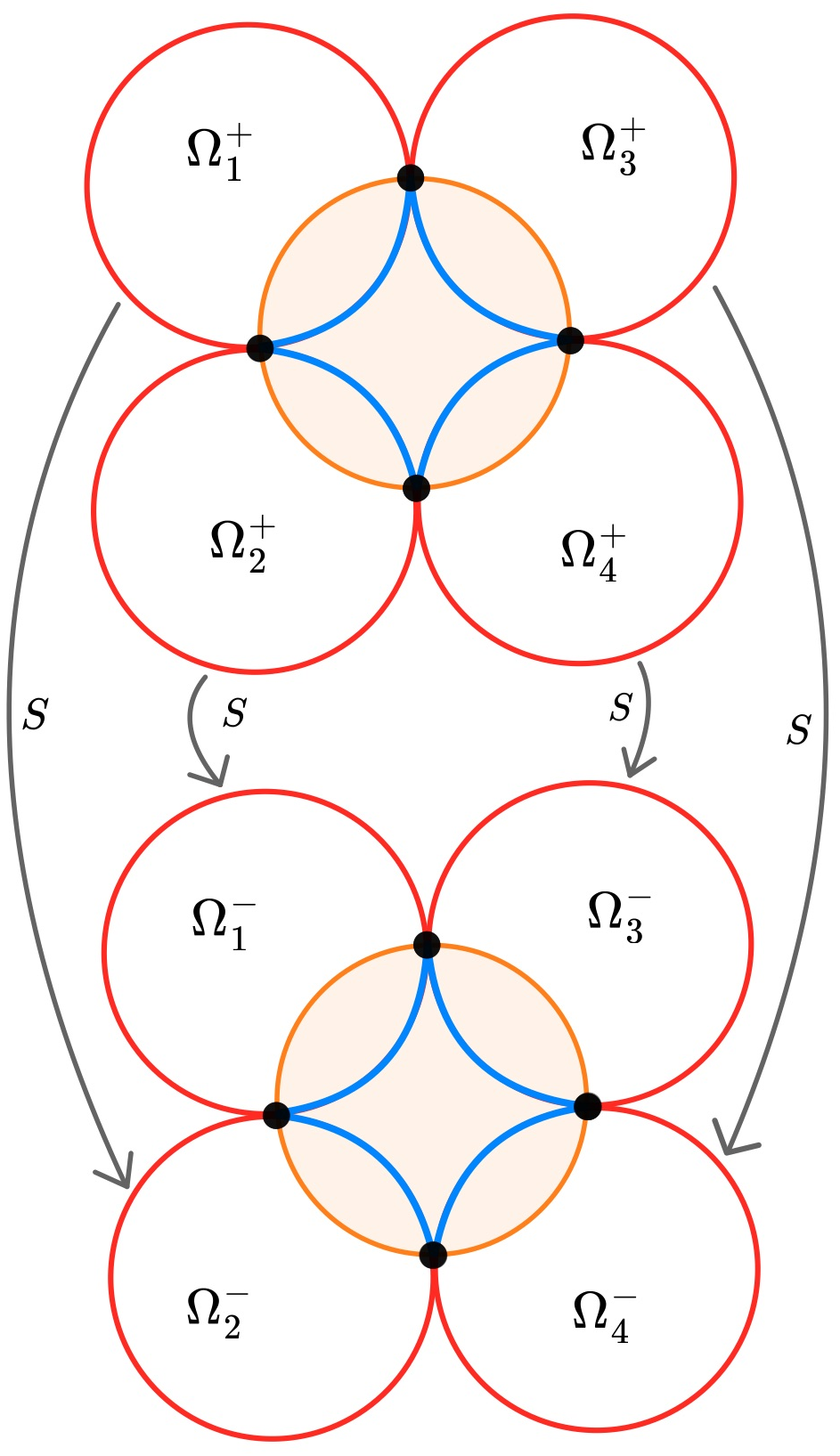}}; 
    \node[anchor=south west,inner sep=0] at (6.6,1.2) {\includegraphics[width=0.32\textwidth]{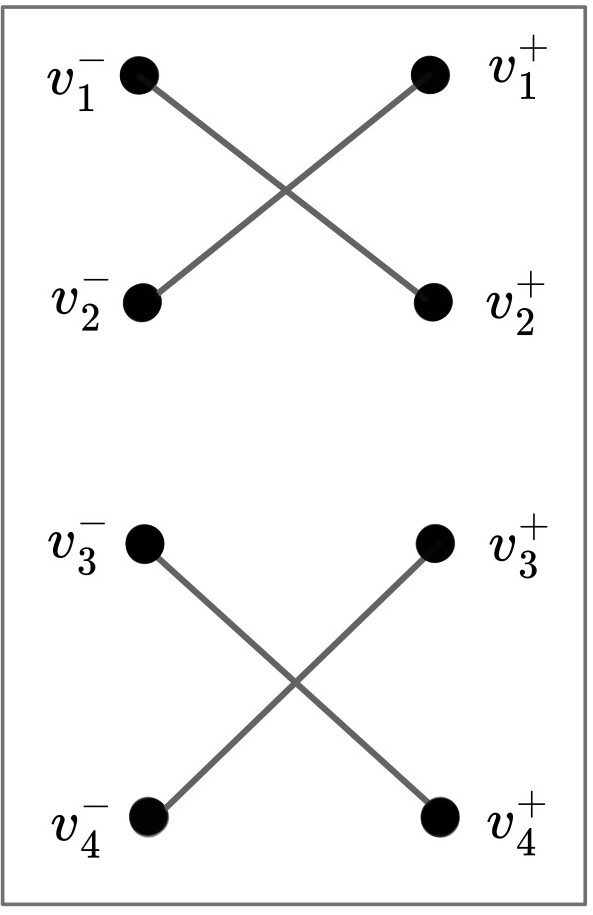}};
    \end{tikzpicture}
    \caption{Left: The boundary identifications that set up the welding construction of the blender surface of Example~\ref{000_ex} is shown here. The arrows indicate how the map $S$ take the boundary components of the top copy of the domain $\cD = \bigcup_{i=1}^4\Omega_i$ to its bottom copy. Right: The welding graph associated with the conformal mating $S$ is shown here. The four connected components of the welding graph would correspond to the four components of the blender surface $\Sigma$. Further, each connected component of $\Sigma$ is obtained by capping a disk off by a disk. So, $\Sigma$ is a disjoint union of four Riemann spheres.}
    \label{000_fig}
    \end{figure}
\begin{example}[\textbf{Disconnected blender and zipped surfaces}]\label{000_ex}
    The map $R(z) = z^3$ has two completely invariant Fatou components. Let $S:\overline{\cD}\rightarrow\widehat{\mathbb{C}}$ be the conformal mating of $R$, and $\{A_{\Gamma_i}^{\text{BS}}\}_{i=1}^{2}$, where $\Gamma_1 = \Gamma_2 = \Gamma_{1, 4}$.
    The domain $\overline{\cD}$ is a union of four closed topological disks, and the corresponding welding graph has four connected components as shown in Figure~\ref{000_fig}. It is, in fact, not hard to see that $S$ is a piecewise M{\"o}bius map. The blender surface $\Sigma$ is a disjoint union of four spheres, and since the action of the group $\langle\eta\rangle$ partitions $\Sigma$ into two orbits; the zipped surface $\widecheck{\Sigma} \cong \Sigma/\langle\eta\rangle$ is a disjoint union of two spheres.
\end{example}

\begin{example}[\textbf{Disconnected blender and connected zipped surfaces}]\label{001_ex}
    Let $R$ be the critically fixed polynomial given by $R(z) =  z^3 + \frac{3z}{2}$. Borrowing notation from Section~\ref{crit_fixed_rat_subsec}, note that, $\text{Fix}(R) \supsetneq \text{Crit}(R) = \{\frac{i}{\sqrt{2}}, -\frac{i}{\sqrt{2}}, \infty\}$. The point $\infty$ is a critical point with multiplicity $2$, while the other two critical points have multiplicities $1$. 
\begin{figure}[h!]
    \captionsetup{width=0.98\linewidth}
    \begin{tikzpicture}
    \node[anchor=south west,inner sep=0] at (0,0) {\includegraphics[width=0.4\textwidth]{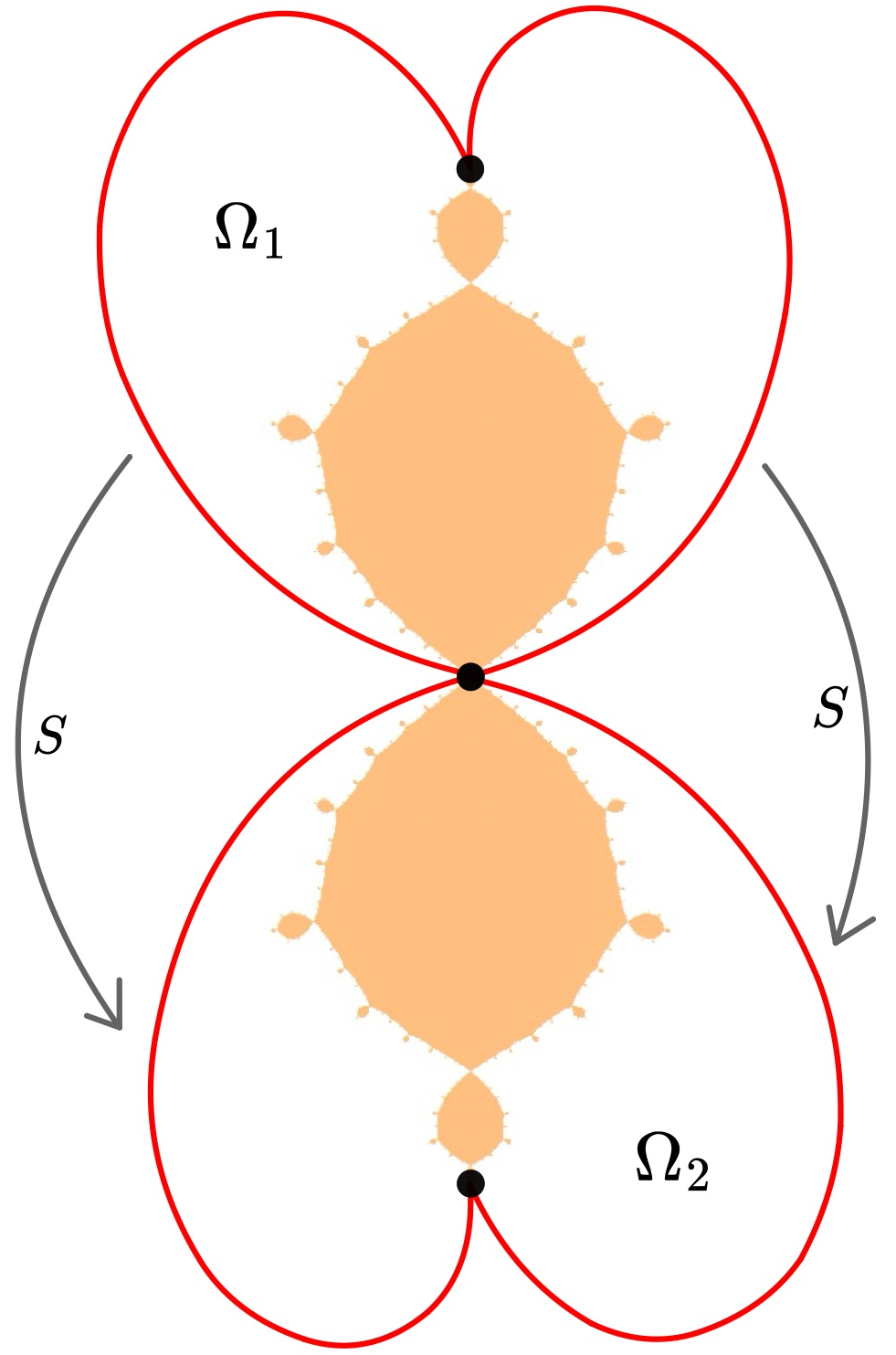}}; 
    \end{tikzpicture}
    \caption{Shown here is a cartoon of the dynamical plane of $S$ of Example~\ref{001_ex}. The map $S$ takes the boundary of the topological disk $\Omega_1$ to that of $\Omega_2$, and vice versa, as indicated.}
    \label{001_fig}
    \end{figure}    
    
     We replace the dynamics of $R$ on the basin of infinity with the Bowen-Series map $A_{\Gamma_{1, 4}}^{\text{BS}}$; and on the other two immediate basins corresponding to the fixed-points $\frac{i}{\sqrt{2}}$ and $-\frac{i}{\sqrt{2}}$; we replace the dynamics of $R$ by $\{B_j\}_{j=1}^{2}$, where $B_j$ is a quadratic Blaschke product with an attracting fixed point in $\D$, for $j\in \{1, 2\}$. Let $S$ denote a conformal mating that realizes this construction.  

    The domain of definition of the conformal mating $S$ is a disjoint union of two closed topological disks touching at a point (see Figure~\ref{001_fig}). The welding graph associated with $S$ would have two components, and the blender surface $\Sigma$ in this case would turn out to be a disjoint union of two Riemann spheres. The group $\langle\eta\rangle$ acts transitively on $\Sigma$; as a result, the zipped surface $\widecheck{\Sigma}$ is a single sphere.    
\end{example}

\begin{example}[\textbf{A connected blender surface from an inversive multi-domain}]\label{011_ex}
    Let $R(z) = z^3 +\frac{3z}{2}$ (as in Example~\ref{001_ex}). We now set $S$ to be a conformal mating of $R$, and $\{A_{\Gamma_i}\}_{i=1}^3$, where $\Gamma_1 = \Gamma_{1, 4}$, and $\Gamma_2 = \Gamma_3 = \Gamma_{3, 1}$. The map $S$ replaces the dynamics of $R$ on the basin of infinity with that of $A_{\Gamma_1}=A_{\Gamma_1}^{\mathrm{BS}}$, and on the other two immediate basins with that of $A_{\Gamma_2}=A_{\Gamma_2}^{\mathrm{fBS}}$ and $A_{\Gamma_3}=A_{\Gamma_3}^{\mathrm{BS}}$.
 \begin{figure}[h!]
    \captionsetup{width=0.98\linewidth}
    \begin{tikzpicture}
    \node[anchor=south west,inner sep=0] at (0,0) {\includegraphics[width=0.38\textwidth]{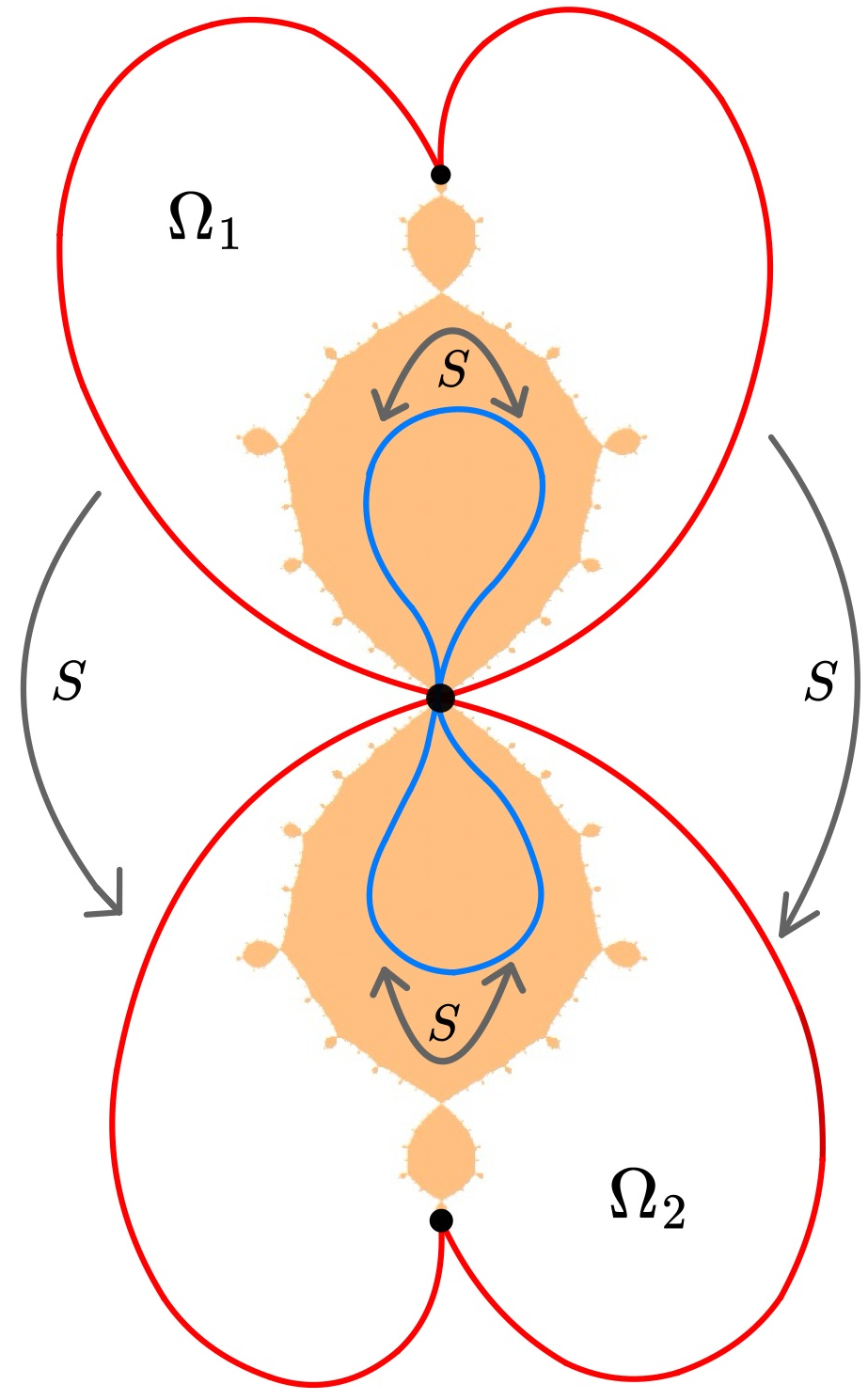}}; 
    \node[anchor=south west,inner sep=0] at (6,0) {\includegraphics[width=0.36\textwidth]{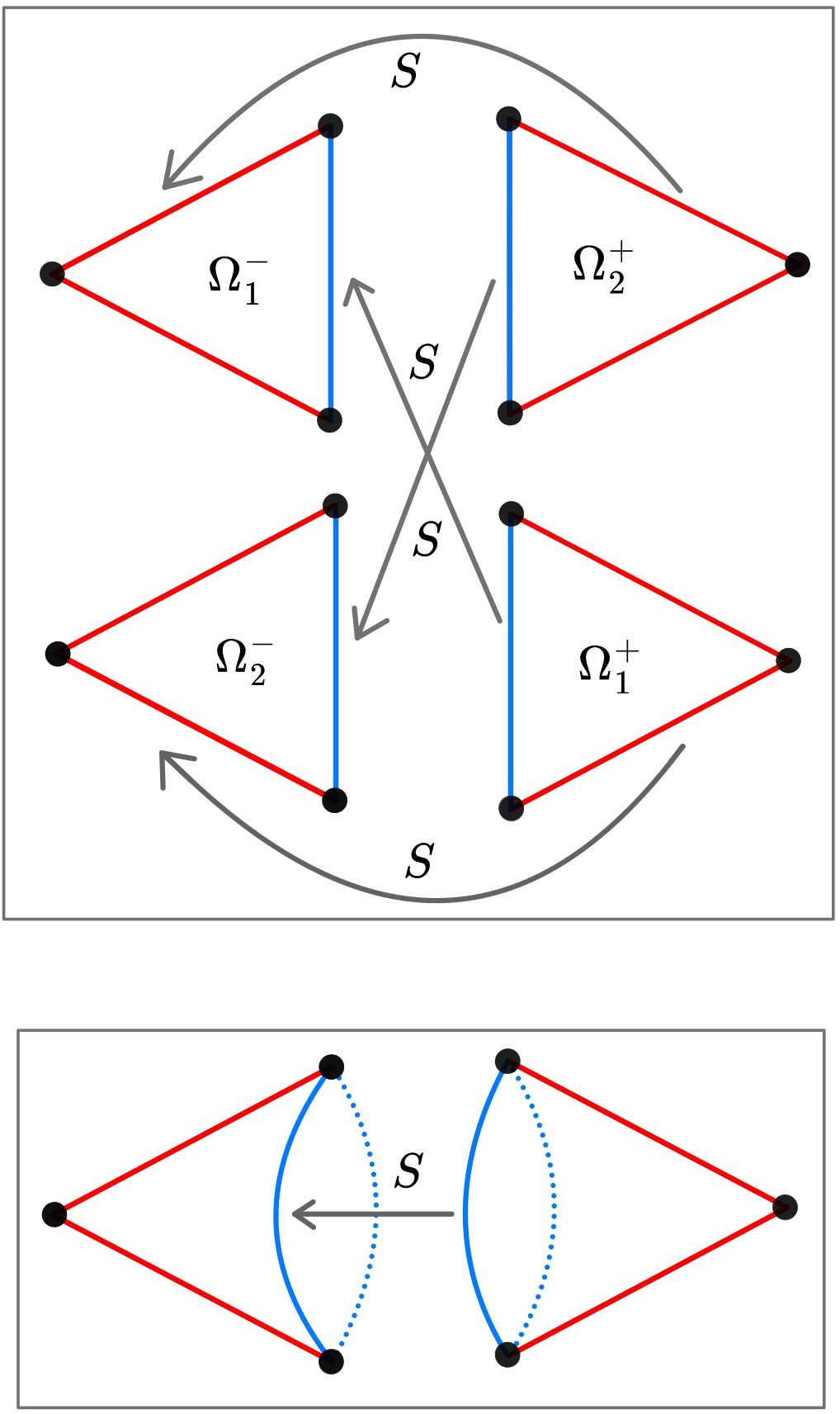}}; 
    \end{tikzpicture}
    \caption{Left: Here we see the dynamical plane of the conformal mating $S$ of Example~\ref{011_ex}, with arrows indicating the action of $S$ on $\partial\cD$. Top Right: The two triangles on the left each map to $\Omega_1$ and $\Omega_2$ under appropriate Riemann maps, and the two triangles on the right are their copies. The induced action of $S$ on the boundaries of the four triangles is indicated. Bottom Right: The blender surface $\Sigma$, topologically speaking, is obtained by gluing the two cones via an orientation-reversing identification of the blue curves, as~shown.}
    \label{011_fig}
    \end{figure}

    The set $\cD$ is a disjoint union of two open topological disks, call them $\Omega_1$ and $\Omega_2$. Each $\partial\Omega_i$ is a wedge of two circles, and $\partial\Omega_1\cap\partial\Omega_2$ is a single point (see Figure~\ref{011_fig}). The welding graph associated with $S$ is connected, and the blender surface $\Sigma$ is a single sphere; forcing $\widecheck{\Sigma}\cong \Sigma/\langle\eta\rangle$ to also be a sphere. By Proposition~\ref{genus_prop}, $\eta$ is a M\"obius map of order $2$ with two fixed points, therefore $\eta$ is equivalent to the map~$z\mapsto 1/z$. 
\end{example}

\begin{example}[\textbf{Elliptic blender surfaces}]\label{111_ex}
    Let $R(z) = z^3 + \frac{3z}{2}$ (as in Example~\ref{001_ex}). Let $A_{\Gamma}=A_{\Gamma}^{\mathrm{BS}}$ be the continuous Bowen-Series map associated with a representation $(\rho: \Gamma_{1, 4}^2\rightarrow\Gamma)\in \text{Teich}(\Gamma_{1,4}^2)$, and let $B$ be a quadratic Blaschke product with an attracting fixed point in $\mathbb{D}$. The conformal mating $S$ is such that it replaces the dynamics of $R$ on the basin of infinity with that of $A_{\Gamma}$, on the immediate basin of the fixed point $\frac{i}{\sqrt{2}}$ with that of $A_{\Gamma_{3,1}}^{\text{fBS}}$, and on the immediate basin of the fixed point $-\frac{i}{\sqrt{2}}$ with that of $B$.

    Note that the domain $\cD$ is an annulus bound by the two connected components of $\partial \cD$ (see Figure~\ref{111_fig}). The welding construction that gives us the blender surface $\Sigma$, topologically speaking, doubles the annulus to a torus. By Theorem~\ref{hyperelliptic_thm} the zipped surface $\widecheck{\Sigma}$ would be a single sphere. 
    \begin{figure}[h!]
    \captionsetup{width=0.98\linewidth}
    \begin{tikzpicture}
    \node[anchor=south west,inner sep=0] at (0,0) {\includegraphics[width=0.66\textwidth]{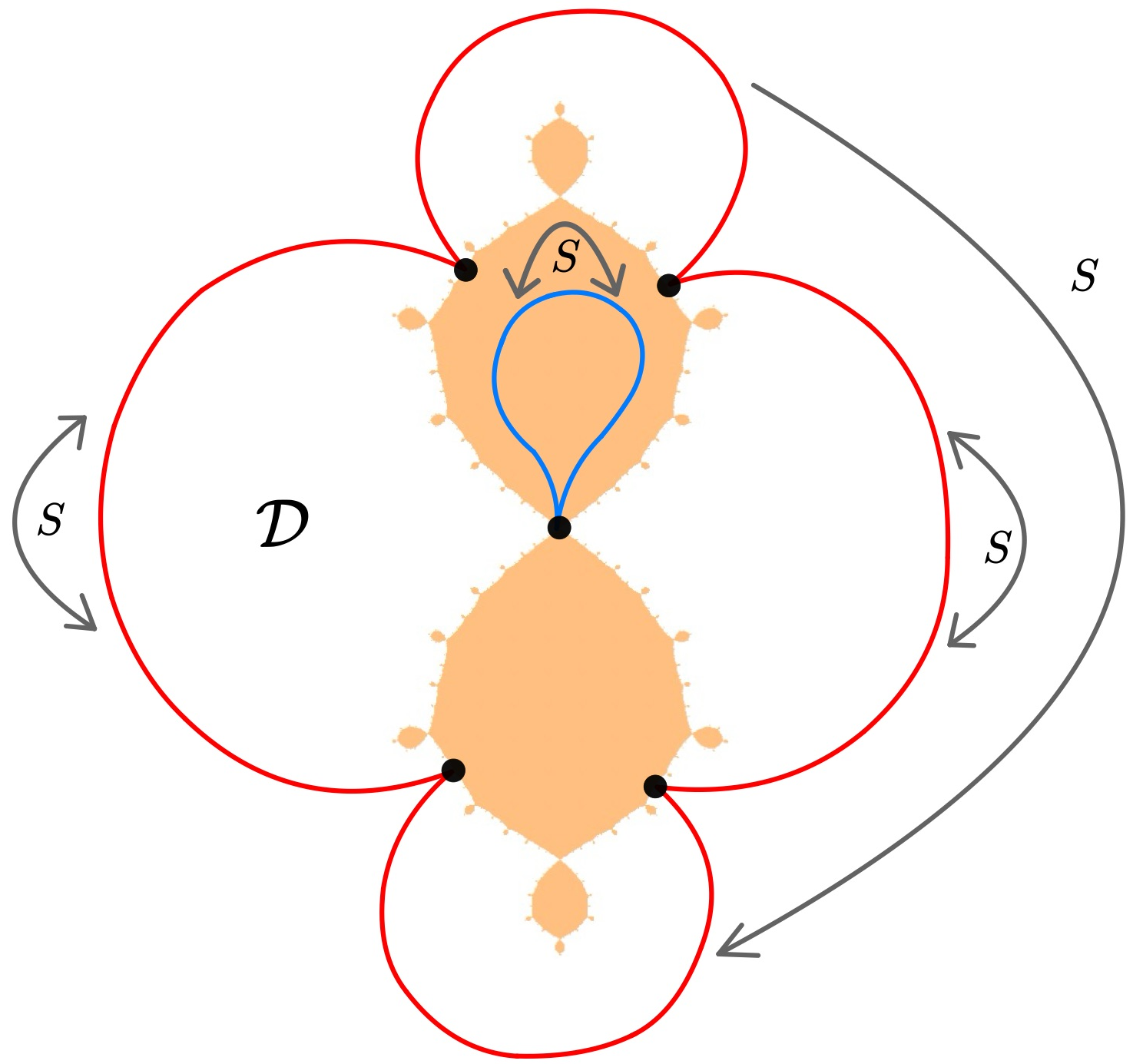}}; 
    \end{tikzpicture}
    \caption{The domain of definition $\overline{\cD}$ of the weak B-involution $S$ of Example~\ref{111_ex} is bound by the red quadrilateral and the blue Jordan curve. The map $S$ permutes the components of $\partial^0\cD$ as depicted in the figure.}
    \label{111_fig}
    \end{figure}
    The group $\Gamma_{1,4}^2$ uniformizes a genus $0$ surface with two punctures and two order $2$ orbifold points. Therefore, unlike all the previous examples, the space $\text{Teich}(\Gamma_{1,4}^2)$ is non-trivial, in fact a $1-$dimensional complex manifold. Let $\mathcal{B}_2$ parametrize the quadratic Blaschke products with an attracting fixed point in $\mathbb{D}$, so $B\in\mathcal{B}_2$. Now note that for every parameter $t\in \text{Teich}(\Gamma_{1,4}^2)\times\mathcal{B}_2$, this specific construction would give us an elliptic curve $\Sigma_t$ (i.e., the blender surface), along with a degree $4$ elliptic function $\cR_t:\Sigma_t\rightarrow\widehat{\mathbb{C}}$ (i.e., the uniformizing meromorphic map)  with $6$ simple critical points and a double critical~point.
\end{example}

\begin{example}[\textbf{Genus $2$ blender surfaces from triply connected inversive domains}]\label{111''_ex}
    Let $R(z) = z^4 + \frac{4z^3}{\sqrt[3]{9}}$, and let $A_{\Gamma}$ be as in Example~\ref{111_ex}. One can check that $R$ is a critically fixed polynomial with $\text{Fix}(R) \supsetneq \text{Crit}(R) = \{-\sqrt[3]{3}, 0, \infty\}$. The critical points at $\infty, 0,$ and $-\sqrt[3]{3}$ have multiplicity $3$, $2$, and $1$, respectively. We let the conformal mating $S$ be the one that replaces the dynamics of $R$ on the basin of infinity with that of $A_{\Gamma_{5, 1}}^{\text{fBS}}$ (such that the fixed point $1$ of $A_{\Gamma_{5, 1}}^{\text{fBS}}$ is identified with the unique non-separating fixed point of $R$ on $\mathcal{J}(R)$); on the basin of the fixed point $0$ with that of $A_{\Gamma}$, and on the basin of the fixed point $-\sqrt[3]{3}$ with that of $A_{\Gamma_{1,3}}^{\text{BS}}$. Further, we conjugate $S$ via an appropriate  M\"obius map to ensure that it has a unique pole.
    \begin{figure}[h!]
    \captionsetup{width=0.98\linewidth}
    \begin{tikzpicture}
    \node[anchor=south west,inner sep=0] at (0,0) {\includegraphics[width=0.72\textwidth]{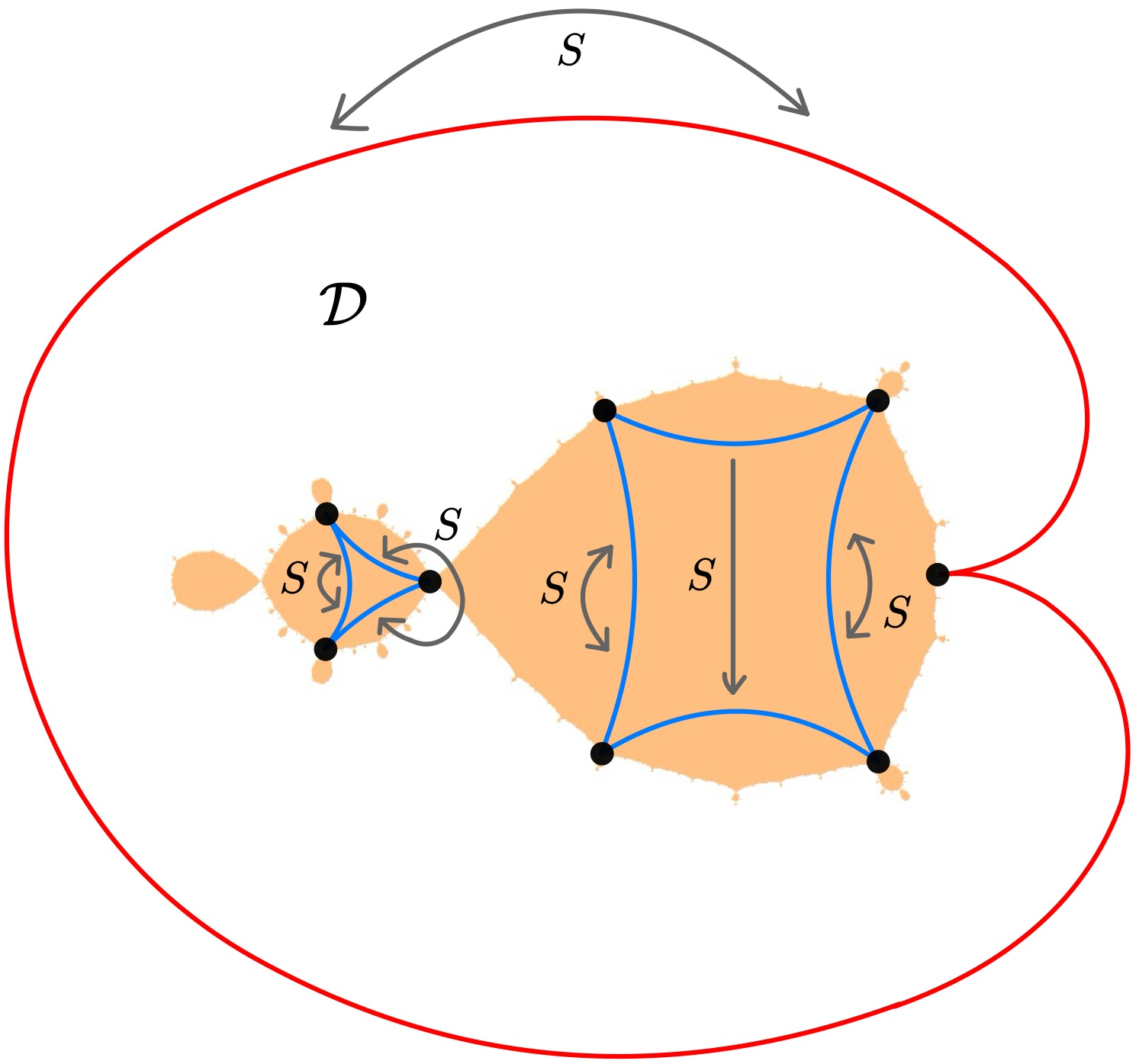}}; 
    \end{tikzpicture}
    \caption{The domain of definition $\overline{\cD}$ of the conformal mating $S$ of Example~\ref{111''_ex} is the region bound by the cardioid, the triangle, and the quadrilateral.}
    \label{111''_fig}
    \end{figure}  
    
    The domain $\overline{\cD}$ is topologically a pair of pants (see Figure~\ref{111''_fig}), and the blender surface is obtained by gluing $\overline{\cD}$ to itself. Therefore, $\Sigma$ is a genus $2$ hyperelliptic Riemann surface, and the uniformizing meromorphic map $\cR:\Sigma\rightarrow\widehat{\mathbb{C}}$ has degree $5$ with a unique pole. 
\end{example}

\begin{example}[\textbf{Rigid Newton examples}]\label{111'''_ex}
    Consider the polynomial $P(z) = z^n-1$, $n\geq 3$, and let $R(z):= z - \frac{P(z)}{P'(z)}$ be the Newton map for the polynomial $P(z)$. Note that, 
    $$
    R(z) = z - \frac{z^n - 1}{nz^{n-1}} = z - \frac{z}{n} + \frac{1}{nz^{n-1}}, \hspace{2mm}\text{and}\hspace{2mm}R'(z)= \Bigg(\frac{n-1}{n}\Bigg)\cdot\Bigg(1-\frac{1}{z^n}\Bigg).
    $$
    As $R$ maps the point $0$ to the point $\infty$ with local degree $n-1$, we conclude that $\text{Crit}(R) = \{0\}\cup\{e^{\frac{2\pi ik}{n}}\}_{k\in\{1,\cdots, n\}}$, where $0$ is a critical point of multiplicity $n-2$, while all the $n$-th roots of unity are simple critical points of $R$. Let $S$ be the conformal mating of $R$, and $n$ copies of $A_{\Gamma_{3, 1}}^{\text{fBS}}$, one for each immediate basin of attraction of~$R$.

    The domain $\cD$ is a topological disk whose boundary is a wedge of $n$ Jordan curves. Since $\cD$ is connected, the associated blender surface $\Sigma$ is also connected (see Figure~\ref{111'''_fig}). The involution $\eta$ has $2g + 2$ fixed points, where $g$ denotes the genus of $\Sigma$. 
       \begin{figure}[h!]
    \captionsetup{width=0.98\linewidth}
    \begin{tikzpicture}
    \node[anchor=south west,inner sep=0] at (0,0) {\includegraphics[width=0.48\textwidth]{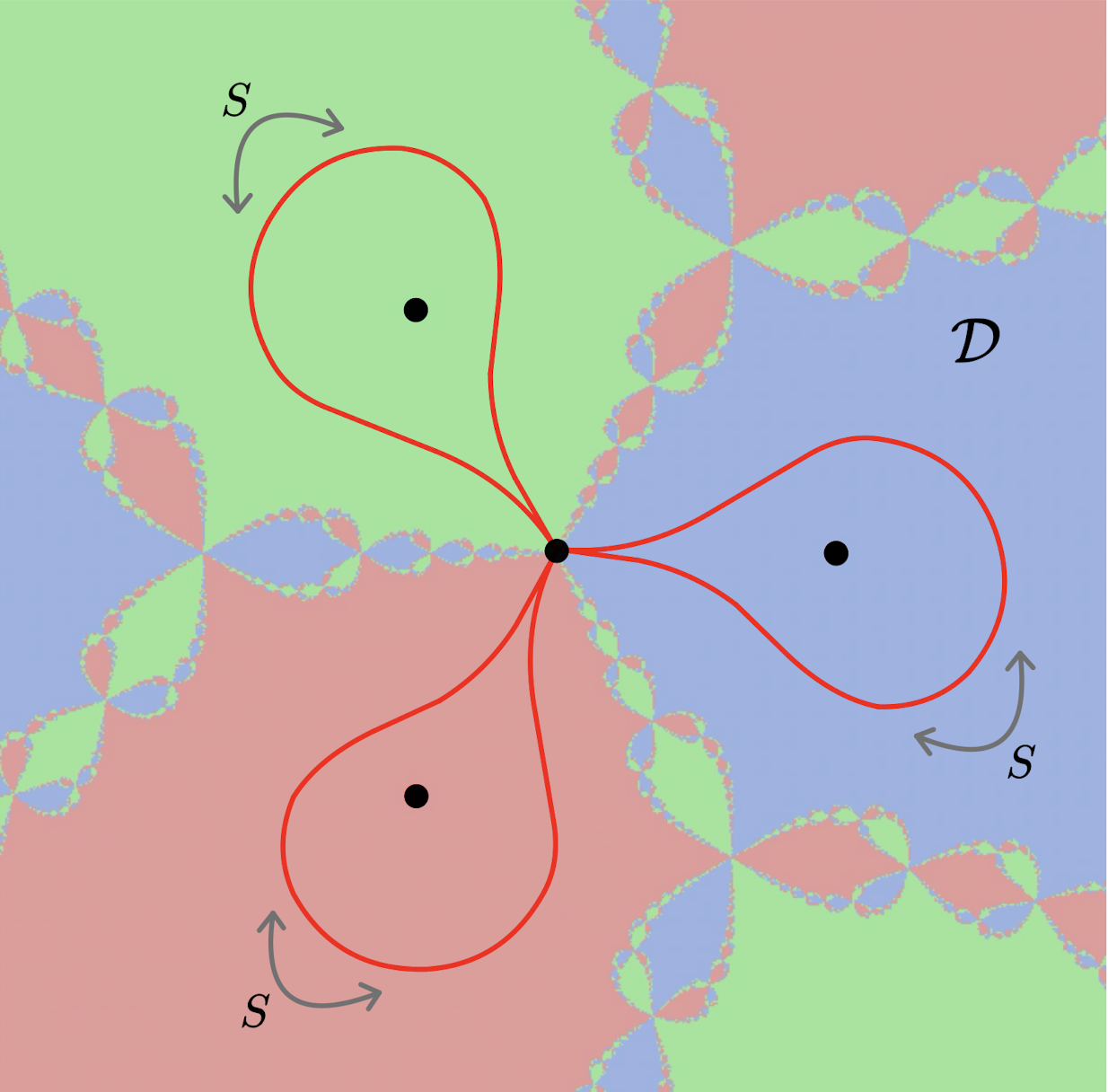}}; 
    \node[anchor=south west,inner sep=0] at (6.2,0) {\includegraphics[width=0.48\textwidth]{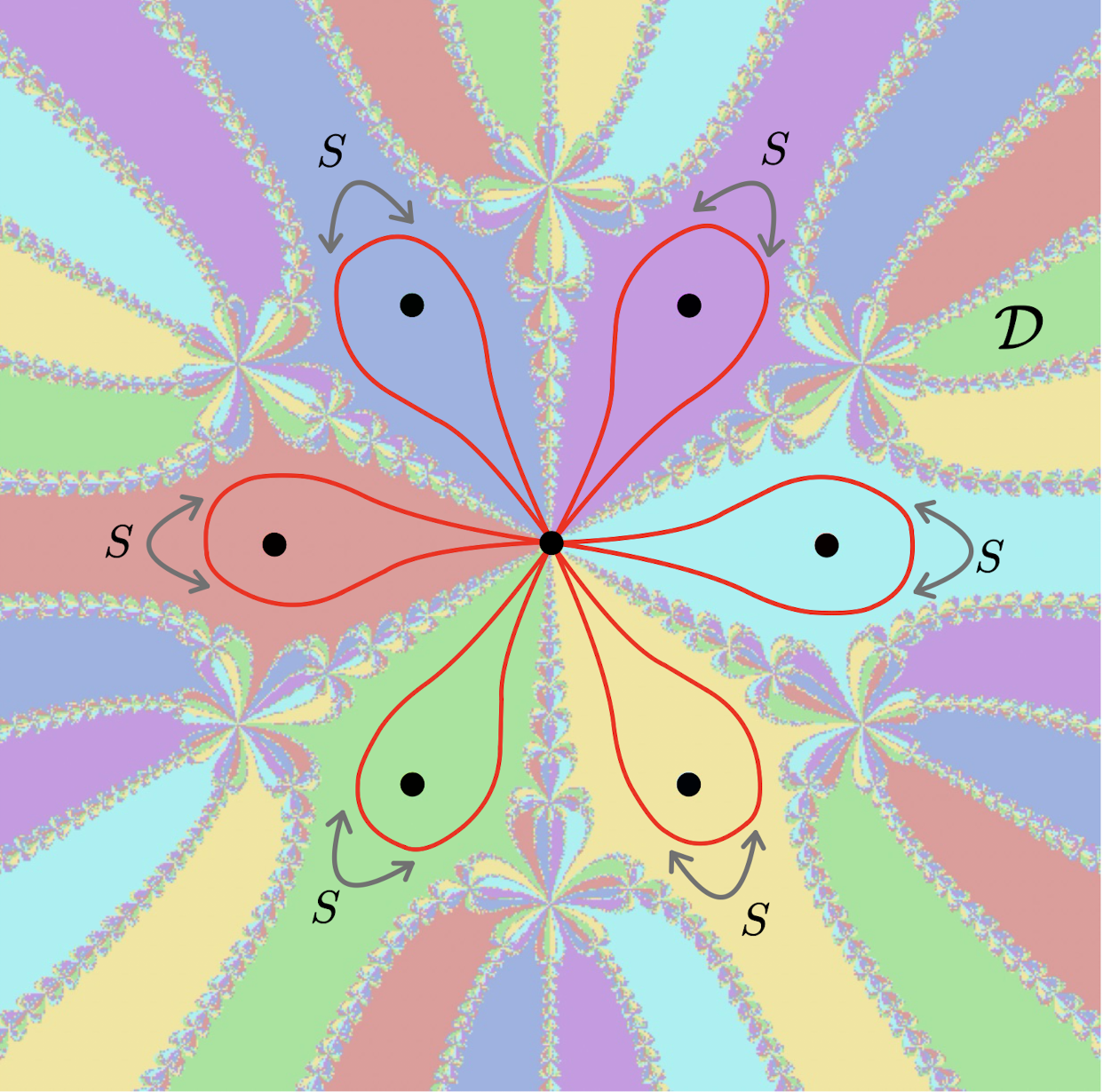}}; 
    \end{tikzpicture}
    \caption{Illustrated here are instances of the dynamical plane of $S$ from Example~\ref{111'''_ex}, for $n =3$ and $n=6$. Left: The three colored regions represent the immediate basins of the Newton map for the polynomial $z^3-1$, along with their pre-images. In each immediate basin, the original dynamics has been replaced by that of $A_{\Gamma_{3, 1}}^{\text{fBS}}$. Right: This image is based on the Newton map for the polynomial $z^6 - 1$, with modifications: in each of the six immediate basins, centered at the sixth roots of unity, the Newton dynamics has been surgically replaced by $A_{\Gamma_{3,1}}^{\text{fBS}}$. The conformal mating $S$ is defined in the exterior of the six red Jordan curves meeting at the central marked point; this corresponds to the point at infinity where all the immediate basins of the Newton map meet.}
    \label{111'''_fig}
    \end{figure}  
    
    We shall now attempt to arrive at a formula for $g$. Note that $A_{\Gamma_{3, 1}}^{\text{fBS}}$ has a fixed point on $l\cap C_{1,1}$ (see Section~\ref{fBS_subsec}). Under a conformal conjugation, this fixed point is carried to a fixed point of the map $S$. In addition, the map $R$ has a fixed point at $\infty$, and via a topological conjugation to the dynamical plane of $S$, the fixed point at $\infty$ maps to the intersection point of the $n$ Jordan curves forming $\partial\mathcal{D}$. We denote this point by $z_{\infty}$ (see the proof of Theorem~\ref{conf_mating_thm} for an explicit description of these conjugations). Thus, the map $S$ has a total of $n+1$ fixed points on $\partial\cD$: one arising from each copy of $A_{\Gamma_{3, 1}}^{\text{fBS}}$, and one corresponding to the point at $\infty$. 
    
    Note that the action of $S$ on $\partial^0\cD$ is topologically conjugate via the map $\cR$ to the involution $\eta$ on $\partial\mathfrak{D}\setminus\cR^{-1}(z_\infty)$ (see Equation~\ref{alg_char_eq}). Therefore, $\eta$ has exactly $n$ fixed points on $\partial\mathfrak{D}\setminus\cR^{-1}(z_\infty)$. Since all fixed points of $\eta$ lie on $\partial\mathfrak{D}$, it remains to study the action of $\eta$ on $\cR^{-1}(z_{\infty})\cap\partial\mathfrak{D}$. Observe that the quadratic map $\cQ:\Sigma\rightarrow\widecheck{\Sigma}$ maps $\cR^{-1}(z_{\infty})\cap\partial\mathfrak{D}$ to the point corresponding to $z_{\infty}$ in $\widecheck{\Sigma}$. Therefore, $\cR^{-1}(z_{\infty})\cap\partial\mathfrak{D}$ must either consist of a fixed point of $\eta$ or form a 2-cycle under $\eta$. Applying Proposition~\ref{genus_prop} and accounting for parity considerations, we arrive at the following conclusion:
    \begin{equation*}
        g = \begin{cases*}
            \frac{n}{2} - 1, & \text{if }$n$ is even,\\
            \frac{n - 1}{2}, & if $n$ is odd.
        \end{cases*}
    \end{equation*}
\end{example}

\begin{example}[\textbf{Genus two blender surfaces from doubly connected inversive domains}]\label{final_ex}
    Consider the critically fixed polynomial 
    $$
    R(z) = z^7 - \frac{7}{5}(\alpha^2 + \overline{\alpha}^2)z^5 + \frac{7}{3}|\alpha|^4z^3,
    $$ 
    where, $\alpha$ is a complex number in the first quadrant satisfying 
$$
15\alpha+6\alpha^7-14\alpha^5\overline{\alpha}^2=0.
$$    
    Since $R$ is symmetric about the real and imaginary axes; one can check that $\text{Fix}(R) \supsetneq \text{Crit}(R) = \{0, \alpha, \overline{\alpha}, -\alpha, -\overline{\alpha}, \infty\}$. The point $0$ is a critical point with multiplicity $2$, and all other finite critical points are simple. Let $B\in\mathcal{B}_2$ be as in Example~\ref{111_ex}. We let the conformal mating $S:\overline{\cD}\rightarrow\widehat{\C}$ be such that it replaces the dynamics of $R$ on the basin of the fixed point $0$ with that of $A_{\Gamma_{4,1}}^{\text{fBS}}$; on the basins of fixed points $\{\alpha, -\alpha, -\overline{\alpha}\}$ with that of $A_{\Gamma_{3,1}}^{\text{fBS}}$; and on the basin of the fixed point $\overline{\alpha}$ with that of $B$. 
    \begin{figure}[h!]
    \captionsetup{width=0.98\linewidth}
    \begin{tikzpicture}
    \node[anchor=south west,inner sep=0] at (0,0) {\includegraphics[width=0.8\textwidth]{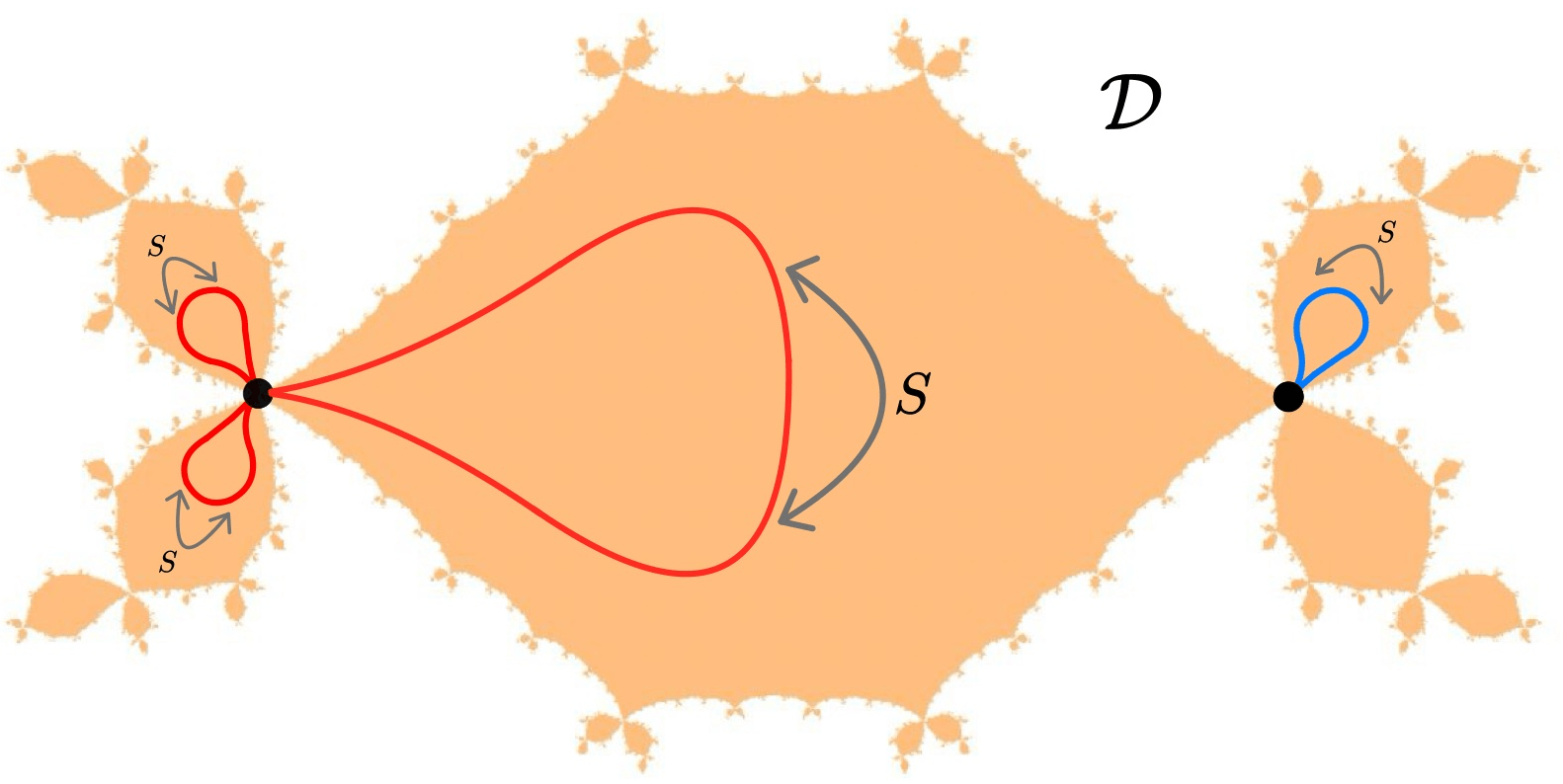}}; 
    \end{tikzpicture}
    \caption{Shown here is the dynamical plane of the conformal mating $S$ of Example~\ref{final_ex}. The two marked points correspond to the two repelling fixed points of the polynomial $R$.}
    \label{final_fig}
    \end{figure}  
    
    The domain $\cD$ is an annulus bound by the red and blue curves as shown in Figure~\ref{final_fig}. Let us denote the red curve by $\gamma_1$, and the blue curve by $\gamma_2$. Observe that welding two copies of $\cD$ along $\gamma_1$ via the map $S|_{\gamma_1}$ yields a genus $1$ surface with two boundary components; we call them $C_1$ and $C_2$ (cf. Figure~\ref{top_weld_surf_1_fig}). By construction, both $C_1$ and $C_2$ can be identified with $\gamma_2$, and the conformal involution $S|_{\gamma_2}$ induces an orientation-reversing analytic map $\widecheck{S}:C_1\rightarrow C_2$. Finally, gluing the two boundary components $C_1$ and $C_2$ via the map $\widecheck{S}$ would conclude the welding procedure by giving us a closed genus $2$ blender surface $\Sigma$. The uniformizing meromorphic map $\cR:\Sigma\rightarrow\widehat{\C}$ in this case has degree $8$ with a pole of order $7$ and a simple pole. Further, $\cR$ has a critical point of multiplicity $6$ (namely, the order $7$ pole),  a critical point of multiplicity $3$ (coming from $A_{\Gamma_{4,1}}^{\text{fBS}}$), three critical points of multiplicity $2$ each (coming from $A_{\Gamma_{3,1}}^{\text{fBS}}$), and three simple critical points (one from the Blaschke product $B$, and the other two from the singular points on $\partial\cD$). 
\end{example}

\begin{remark} Observe that if $\cD$ is connected, then both the surfaces $\Sigma$ and $\widecheck{\Sigma}$ are connected. On the other hand if $\cD$ is disconnected but $\Sigma$ is connected, then $\widecheck{\Sigma}\cong \Sigma/\langle\eta\rangle$ is also connected. Thus, with respect to connectedness of $\cD$, $\Sigma$, and $\widecheck{\Sigma}$, there are four distinct possibilities. We have demonstrated the occurrence of all four possibilities through Examples~\ref{000_ex} to \ref{111_ex}.
\end{remark}

\section{Correspondences on hyperelliptic surfaces}\label{correspondences_sec}

Let $S: \overline{\mathcal{D}} \rightarrow \widehat{\mathbb{C}}$ be the conformal mating of $R$, $\{A_{\Gamma_i}\}_{i = 1}^l$, and $\{B_j\}_{j = 1}^r$, as in Section~\ref{conf_mating_sec}. For the majority of this section, to ensure that our construction of algebraic correspondences is presented in as general a light as our techniques allow (see Example~\ref{111'''_ex}), we will work with general subhyperbolic rational maps~$R$. 
For the proof of Theorem~\ref{mating_thm_intro} given at the very end of this section, we will set $R$ to be a critically fixed polynomial $P$ as constructed in Proposition~\ref{mating_surf_exists_prop} (cf. Lemma~\ref{combi_lem}). The subsequent analysis in Section~\ref{Hurwitz_sec} will only involve these critically fixed polynomials $P$.

The non-invertibility of the factor Bowen-Series maps $\{A_{\Gamma_i}\}_{i=1}^l$, that go into the construction of the conformal mating $S$, is an obstruction to the recovery of the full group structure of $\widehat{\Gamma}_i$ from the action of the map $S$ itself. To mitigate this problem, we will now use the the map $\cR$ to define an algebraic correspondence $\mathfrak{C}$ on the blender surface $\Sigma$ (where $\cR, \Sigma$ are associated with the weak B-involution $S$, see Theorem~\ref{alg_weak_b_inv_thm}). A careful analysis of the dynamics of $\mathfrak{C}$ facilitates the full reconstruction of the underlying Fuchsian groups $\{\widehat{\Gamma}_i\}_{i = 1}^l$, as well as the recovery of the maps $\{B_j\}_{j=1}^r$. This will demonstrate that the correspondence $\mathfrak{C}$ unites the actions of the groups $\widehat{\Gamma}_i$ and the Blaschke products $B_j$.

Let $\mathcal{R}$ and $\mathfrak{D}$ be as in Theorem~\ref{alg_weak_b_inv_thm}. Define a holomorphic correspondence $\mathfrak{C}\subset \Sigma \times \Sigma$ of bi-degree $d-1:d-1$ (cf. \cite{dinh2006distribution}) as
\begin{equation}
(z, w) \in \mathfrak{C} \iff \frac{\cR(w)-\cR(\eta(z))}{w- \eta(z)}=0,
\label{corr_eq}
\end{equation}
where $d$ is the degree of $\cR:\Sigma\to\widehat{\C}$.
Note that the equation $\cR(w)=\cR(\eta(z))$ defines a reducible holomorphic curve in $\Sigma\times\Sigma$, and $\{(z,w)\in\Sigma\times\Sigma: w=\eta(z)\}$ consists of a certain component thereof. The correspondence $\mathfrak{C}$ is defined by removing this component from the former curve.
We regard the correspondence $\mathfrak{C}$ as a multi-valued map: 
$$
z\mapsto w=\mathcal{R}^{-1}(\mathcal{R}(\eta(z))).
$$
Evidently, the local branches of the correspondences are defined by $\eta$ and the (local) deck transformations of $\cR$. The explicit algebraic description of $S$ in Theorem~\ref{alg_weak_b_inv_thm} is captured in the following commutative diagram,

\begin{equation}
\begin{tikzcd}
	{\overline{\mathfrak{D}}} && {\overline{\mathcal{D}}} \\
	\\
	{\Sigma\backslash\mathfrak{D}} && {\widehat{\mathbb{C}}}
	\arrow["{\mathcal{R}}", from=1-1, to=1-3]
	\arrow["\eta",swap, from=1-1, to=3-1]
	\arrow["{\mathcal{R}}"', from=3-1, to=3-3]
	\arrow["S"',swap, from=1-3, to=3-3]
\end{tikzcd}
\label{alg_char_cd}
\end{equation}

\noindent As a consequence, we immediately get that for $z\in \overline{\mathfrak{D}}$,
\begin{equation}\label{forward_corr_eq}
 (z, w) \in \mathfrak{C} \iff \cR(w) = \cR(\eta(z)) = S(\cR(z)),\ \ w\neq \eta(z), 
\end{equation}
and for $z \in \Sigma\backslash\mathfrak{D}$,
\begin{equation}\label{backward_corr_eq}
    (z,w)\in \mathfrak{C} \iff \cR(w) = \cR(\eta(z)) = S^{-1}(\cR(z)),\ \ w\neq \eta(z),
\end{equation}
for a suitable inverse branch of $S^{-1}$.
\begin{remark}\label{b_inv_corr_reln_rem}
    The brief discussion above shows that a forward branch of the correspondence $\mathfrak{C}$ sends $z$ to $w$ if and only if either the conformal mating $S$ or one of its inverse branches sends $R(z)$ to $R(w)$. In other words, the meromorphic map $\cR$ lifts the dynamics of $S$ to the correspondence $\mathfrak{C}$.
\end{remark}

\subsection{Dynamically invariant sets for $\mathfrak{C}$}\label{inv_partition_corr_subsec}
Recall from Remark~\ref{invariant_subset_rem} the dynamical partition
$$
\widehat{\mathbb{C}} = \Bigg(\bigsqcup_{i=1}^{l}\mathcal{W}_i\Bigg)\bigsqcup \cK.
$$
Define
$$
\widetilde{\cK}:= \cR^{-1}(\cK),\quad \textrm{and}\quad \widetilde{\mathcal{W}}_i := \cR^{-1}(\mathcal{W}_i),\  i\in\{1, \cdots,l\},
$$
\noindent as the lifts of the non-escaping and escaping sets of the conformal mating $S$ under $\mathcal{R}$. Further, let $\Lambda(\mathfrak{C}):=\partial\widetilde{\cK}=\cR^{-1}(\partial\cK)$. We shall call $\Lambda(\mathfrak{C})$ the \emph{limit set} of $\mathfrak{C}$.

\begin{prop}\label{invariance_prop}\upshape 
The following statements hold for $i \in \{1, \cdots, l\}$.

\begin{enumerate}
    \item $\eta(\widetilde{\mathcal{W}_i}) = \widetilde{\mathcal{W}_i}$, and $\eta(\widetilde{\cK}) = \widetilde{\cK}$.
    \item Given a point $(z, w) \in \mathfrak{C}$, we have $z\in \widetilde{\mathcal{W}}_i$ (respectively, $z\in \widetilde{\cK}$) if and only if $w\in \widetilde{\mathcal{W}}_i$ (respectively, $w\in \widetilde{\cK}$).
\end{enumerate} 
\end{prop}

\begin{proof} Let $\mathcal{X}$ be a stand-in notation for the sets $\{\widetilde{\mathcal{W}_i}\}_{i=1}^{l}$, and $\widetilde{\cK}$. Clearly, by Relation~\eqref{alg_char_cd}, one of the following holds for $z\in \mathcal{X}$,
    $$
    \eta(z) \in \cR^{-1}(S(\cR(z))), \hspace{2mm}\text{or}\hspace{1.5mm} \eta(z) \in \cR^{-1}(S^{-1}(\cR(z))), 
    $$
\noindent depending on whether $z\in \mathcal{X}\cap\overline{\mathfrak{D}}$ or $z\in \mathcal{X}\backslash\mathfrak{D}$. By definition, $\cR(\mathcal{X})$ is completely invariant under the conformal mating $S$, therefore, $\eta(\mathcal{X}) = \mathcal{X}$. A similar reasoning applies to the second claim of the proposition.
\end{proof}

Fix an index $i\in \{1, 2, \cdots, l\}$. Observe that points in every component of the escaping set $\cW_i$ of $S$ eventually map to $\cT_i$ before leaving the domain $\overline{\cD}$ of the conformal mating (see Remark~\ref{invariant_subset_rem} for the definition of $\cT_i$). 
For a component $C$ of $\cW_i$, let $C'$ be components of $\cW_i$ such that $S(C') = C$. If $C\neq \cT_i$, then by Relation~\ref{alg_char_cd},
\begin{equation}
\cR^{-1}(C) = \Bigg(\big(\cR\big|_{\mathfrak{D}}\big)^{-1}(C)\Bigg)\bigsqcup \Bigg(\bigsqcup_{C'}\eta\Big(\big(\cR\big|_{\mathfrak{D}}\big)^{-1}(C')\Big)\Bigg).
\label{preimage_formula_eq}
\end{equation}
\noindent On the other hand, if $C = \cT_i$, then for $\widecheck{C} = C\cap \overline{\cD}$, 
$$
\cR^{-1}(C) = \Bigg(\big(\cR\big|_{\overline{\mathfrak{D}}}\big)^{-1}(\widecheck{C})\Bigg)\bigcup\Bigg(\eta\Big(\big(\cR\big|_{\overline{\mathfrak{D}}}\big)^{-1}(\widecheck{C})\Big)\Bigg)\bigsqcup\Bigg(\bigsqcup_{C'\neq C} \eta\Big(\big(\cR\big|_{\mathfrak{D}}\big)^{-1}(C')\Big)\Bigg).
$$
In a nutshell, the involution $\eta$, restricted to $\mathfrak{D}$, mimics the action of $S$ on $\mathcal{D}$. Specifically, it maps components of $\cR^{-1}(C') \cap \mathfrak{D}$ to components of $\cR^{-1}(C) \cap \mathfrak{D}^c$, if $S$ maps $C'$ to $C$. Now define,
\begin{equation}
\widetilde{\cT_i}:= \big(\cR\big|_{\overline{\mathfrak{D}}}\big)^{-1}(\widecheck{\cT_i})\cup\eta\Big(\big(\cR\big|_{\overline{\mathfrak{D}}}\big)^{-1}(\widecheck{\cT_i})\Big)\subseteq \cR^{-1}(\cT_i),
\label{tiling_comp_eq}
\end{equation}
where $\widecheck{\cT_i} = \cT_i\cap\overline{\cD}$. By the construction of the blender surface $\Sigma$, the open set $\widetilde{\cT_i}$ is obtained by gluing two copies of $\widecheck{\cT}_i$ along $\partial\widecheck{\cT}_i\cap\partial^0\cD$ via the weak B-involution $S$. Consequently, $\widetilde{\cT}_i$ is a disjoint union of $p_i$ topological disks (cf. \cite[Lemma~5.3, Figure~10]{mj2023matings}). By Corollary~\ref{crit_points_cor} and Relation~\ref{alg_char_eq}, the map $\cR$ carries each component of $\widetilde{\cT_i}$ as an $n_i:1$ branched cover onto $\cT_i$, and when $n_i\geq 3$, each component of $\widetilde{\cT_i}$ contains a unique critical point of $\cR$ of multiplicity $n_i-1$.

\begin{remark}\label{tesselation_rem}
    The tessellation of the unit disk $\mathbb{D}$ by the group $\Gamma_i$ can be transferred to the domain $\cT_i$. Define the rank zero tile as $\cT^0_i = \cT_i \setminus \cD$. The rank $m$ tiles are the connected components of the pre-image $S^{-m}(\cT^0_i)$. Every point in $\cT_i$ lies in a rank $m$ tile for some $m \geq 0$, and the action of $S$ maps each rank $m\geq 1$ tile to a rank $m-1$ tile. Therefore, all points in $\cT_i \cap \overline{\cD}$ eventually escape to the “hole” (i.e., the rank zero tile).

This dynamical tessellation can be lifted via the map $\cR|_{\widetilde{\cT}_i}$ to obtain a similar tiling structure in $\widetilde{\cT}_i$. A rank $m$ tile in $\widetilde{\cT}_i$ is defined as a connected component of the pre-image under $\cR|_{\widetilde{\cT}_i}$ of a rank $m$ tile in $\cT_i$. We denote the rank zero tile(s) in $\widetilde{\cT}_i$ by $\widetilde{\cT}_i^0$; these are the $p_i$ disjoint pre-images of $\cT^0_i$, one in each component of $\widetilde{\cT}_i$.
\end{remark}

\begin{remark}\label{esc_inv_rem}
The set $\widetilde{\cT_i}$ is invariant under $\eta$, which brings to attention a stark difference between the dynamics of the correspondence $\mathfrak{C}$ and the conformal mating $S$. While the points of $\cW_i$ eventually escape $\cD$ under the iterates of $S$; points in $\widetilde{\cW}_i$ ``eventually" map to $\widetilde{\cT_i}$, after which they remain  in $\widetilde{\cT_i}$ under $\eta$ as well as under suitable local deck transformations of $\cR$. In other words, by lifting the dynamics of the conformal mating $S$ to that of the correspondence $\mathfrak{C}$, we managed to suppress the escaping behavior exhibited by the former. Broadly speaking, we shall employ this philosophy to recover from $\mathfrak{C}$ the groups $\{\widehat{\Gamma}_i\}_{i=1}^l$. 
\end{remark}

\noindent Now fix an index $j\in \{1, 2, \cdots, r\}$ and define, 
\begin{equation}
\widetilde{\cV}_j:= \big(\cR\big|_{\mathfrak{D}}\big)^{-1}(\cV_j)\sqcup\eta\Big(\big(\cR\big|_{\mathfrak{D}}\big)^{-1}(\cV_j)\Big)\subseteq \cR^{-1}(\cV_j),
\label{blaschke_comp_eq}
\end{equation}
(see Remark~\ref{invariant_subset_rem} for the definition of $\cV_j$). By definition, $\widetilde{\cV}_j$ is a disjoint union of a pair of topological disks. By Theorem~\ref{conf_mating_thm} and Relation~\ref{alg_char_eq}, the map $\cR$ carries the component of $\widetilde{\cV}_j$ lying in $\mathfrak{D}$ univalently onto $\cV_j$; while the other component maps onto $\cV_j$ as a $(\deg(B_j)-1):1$ branched cover. 

As in Relation~\ref{preimage_formula_eq}, given a component $C$ of $\cW_{l+j}$, the pre-image $\cR^{-1}(C)$ can be written down explicitly in terms of components $C'$ of $\cW_{l+j}$ such that $S(C') = C$. 

Define $\mathfrak{C}_i:= \mathfrak{C}\cap\big(\widetilde{\cT}_i \times \widetilde{\cT}_i\big)$. One can think of $\mathfrak{C}_i$ as a multi-valued map:
\begin{equation}\label{restricted_corr_eq}
   z\mapsto w = \Big(\cR\big|_{\widetilde{\cT}_i}\Big)^{-1}\bigg(\Big(\cR\big|_{\widetilde{\cT}_i}\Big)(\eta(z))\bigg), 
\end{equation}
where the local branches are generated by $\eta$ and the local deck transformations of the branched covering map $\cR\big|_{\widetilde{\cT_i}}: \widetilde{\cT_i} \rightarrow\cT_i$. In the coming two subsections we shall prove the following theorem (cf. \cite[Theorem~5.2]{mj2023matings}). We recall the notation $\widehat{\Gamma}_i = \langle\Gamma_i, M_{\omega_i}\rangle$ from Section~\ref{fBS_subsec}. 

\begin{theorem}\label{corr_recover_thm}
    The correspondence $\mathfrak{C}$ is a combination of the groups $\{\widehat{\Gamma}_i\}_{i=1}^l$, and the Blaschke products $\{B_j\}_{j=1}^r$, in the following sense.
    \begin{enumerate}[leftmargin=8mm]
        \item The action of $\mathfrak{C}_i$ on $\widetilde{\cT_i}$ is equivalent to the action of a group $\langle\eta, \tau_i\rangle$ of conformal automorphisms of $\widetilde{\cT_i}$, $i\in\{1,\cdots, l\}$. Moreover, the above group acts properly discontinuously, and the quotient orbifold $\widetilde{\cT_i}/\mathfrak{C}_i$ is biholomorphic to~$\mathbb{D}/\widehat{\Gamma}_i$.
       
        \item The correspondence $\mathfrak{C}$ has a forward branch (respectively, a backward branch) carrying $\widetilde{\cV}_j\cap \mathfrak{D}$ (respectively, $\widetilde{\cV}_j\setminus \overline{\mathfrak{D}}$) onto itself, and this branch is conformally conjugate to the Blaschke product $B_j$, $j\in\{1,\cdots,r\}$.
    \end{enumerate}
\end{theorem}

The proof would be split as follows. In Section~\ref{group-like_subsubsec} we shall study the action of $\mathfrak{C}$ on $\widetilde{\cT_i}$, $i\in\{1,\cdots, l\}$, to recover the groups $\{\widehat{\Gamma}_i\}_{i=1}^l$, and establish the quotient orbifold equivalence $\widetilde{\cT_i}/\mathfrak{C}_i\cong \mathbb{D}/\widehat{\Gamma}_i$. In Section~\ref{map-like_subsubsec}, we will study the action of $\mathfrak{C}$ on $\widetilde{\cV}_j$, to recover the Blaschke products $\{B_j\}_{j=1}^r$.

\begin{remark}
Some features of the algebraic correspondence $\mathfrak{C}$ deserve clarification.
The action of the correspondence $\mathfrak{C}$ splits the surface $\Sigma$ into the invariant regions
$$
\widetilde{\cW}_i=\cR^{-1}(\cW_i),\ i\in\{1,,\cdots,l\},\quad \widetilde{\mathcal{W}}_{l+j}=\cR^{-1}(\cW_{l+j}),\ j\in\{1,\cdots,r\},
$$
and 
$$
\mathcal{Z}:=\Sigma\setminus\bigcup_{i=1}^{l+r} \widetilde{\cW}_i
$$
(cf. Remark~\ref{invariant_subset_rem}).
However, the group action (of $\widehat{\Gamma}_i$) is displayed by the correspondence only on the component $\widetilde{\cT}_i$ of $\widetilde{\cW}_i$, $i\in\{1,\cdots,l\}$. Similarly, the Blaschke product action (of $B_j)$ is exhibited by $\mathfrak{C}$ only on the open set $\widetilde{\cV}_j$ of $\widetilde{\cW}_{l+j}$, $j\in\{1,\cdots,r\}$. On the remaining invariant region $\mathcal{Z}$, a suitable branch of the correspondence is conjugate to the action of the subhyperbolic rational map $R$. The last fact is a consequence of the construction of the underlying B-involution $S$ (cf. Theorem~\ref{conf_mating_thm}) and the relation between $S$ and the correspondence $\mathfrak{C}$ (cf. Remark~\ref{b_inv_corr_reln_rem}). The map $S$ is manufactured by replacing the dynamics of the subhyperbolic rational map $R$ on some of its invariant Fatou components with the actions of (factor) Bowen-Series maps $\{A_{\Gamma_i}\}_{i=1}^l$ and hyperbolic Blaschke products $\{B_j\}_{j=1}^r$, while leaving the rest of the dynamics of $R$ untouched. Consequently, $S$ is conjugate to $R$ on the complement of the grand orbits of the surgered Fatou components. The above fact and the intimate connection between the branches of the correspondence $\mathfrak{C}$ and the B-involution $S$ imply that a suitable branch of $\mathfrak{C}$ on $\mathcal{Z}$ is conjugate to $R$ (see Remark~\ref{nodal_rem} for details). 
\end{remark}

\subsection{Group-like dynamics in $\mathfrak{C}$}\label{group-like_subsubsec}Fix an index $i\in \{1, 2, \cdots, l\}$. Recall (from the proof of Theorem~\ref{conf_mating_thm}) the conformal map $\mathfrak{X}_i:\mathbb{D}\rightarrow\cT_i$ that satisfies
\begin{equation}\label{fBS_conj_eq}
     \mathfrak{X}_i \circ A^{\text{fBS}}_{\Gamma_i} =  S \circ \mathfrak{X}_i
\end{equation}
on appropriate regions. This map sends the origin in the dynamical plane of $A_{\Gamma_i}$ to a point in $\cT_i$; we will refer to this image point as a \emph{distinguished point}. 

Note that each component of $\widetilde{\cT}_i$ contains a unique pre-image of the distinguished point under $\cR$. If $n_i\geq 3$, the point $0$ is the unique critical value of $A_{\Gamma_i}$, and hence the distinguished point would be the unique critical value of $\cR\vert_{\widetilde{\cT}_i}$ in $\cT_i$. Further, each component of $\widetilde{\cT}_i$ contains a pre-image of the distinguished point (under the map $\cR$) that maps to it with local degree $n_i$.   

\begin{lem}\label{aut_corr_lem}
    The open set $\widetilde{\cT_i}$ admits a biholomorphism $\tau_i$ such that
    $$\tau^{n_ip_i}_i\equiv \text{id}, \hspace{2mm} and \hspace{2mm}\Big(\cR\big|_{\widetilde{\cT}_i}\Big)^{-1}\bigg(\Big(\cR\big|_{\widetilde{\cT}_i}\Big)(z)\bigg) = \{z, \tau_i(z), \cdots, \tau_i^{n_ip_i-1}(z)\} \hspace{2mm}\forall z\in \widetilde{\cT_i}.$$
\end{lem}

\begin{proof}
    Enumerate the $p_i$ components of $\widetilde{\cT_i}$ by the index set $\cI:=\{1,\cdots,p_i\}$. Let $\Phi: \mathbb{D}\times\cI\rightarrow \widetilde{\cT_i}$ be a biholomorphism such that for each $j\in \cI$, the map $\Phi$ sends $(0, j)$ to the unique pre-image of the distinguished point (of $\cT_i$) in the $j$-th component of $\widetilde{\cT_i}$. Therefore, after possibly pre-composing $\Phi$ with a rotation on each $\mathbb{D}\times\{j\}$, we can write $\widetilde{\cR}:= \mathfrak{X}_i^{-1}\circ \cR\circ \Phi:\mathbb{D}\times\cI\rightarrow \D$ as
\begin{equation}\label{auxillary_cover_eq}
    (w, j) \mapsto w^{n_i}, \hspace{2mm}w\in \mathbb{D}, \hspace{1mm}j\in \cI.
    \end{equation}
    Now define the map $\widetilde{\tau}_i:\mathbb{D}\times\cI\rightarrow\mathbb{D}\times\cI$ as follows,
\begin{equation}\label{auxillary_aut_eq}
    \widetilde{\tau}_i(w, j):= \begin{cases*}
        (w, j+1), & $\text{for } j \in \{1, \cdots, p_i-1\}$\\
        (e^{\frac{2i\pi}{n_i}}w, 1), & $\text{for } j = p_i$.
    \end{cases*}
    \end{equation}
    One can check using Formulas~\eqref{auxillary_cover_eq} and~\eqref{auxillary_aut_eq} that, 
    $$
    \widetilde{\tau}_i^{n_ip_i} = \text{id}, \hspace{2mm} \text{and  }\widetilde{\cR}^{-1}(\widetilde{\cR}(w, j)) = \{(w, j), \widetilde{\tau}(w, j), \cdots, \widetilde{\tau}^{n_ip_i-1}(w, j)\},
    $$
    for $(w, j) \in \mathbb{D}\times\cI$. The required automorphism $\tau_i$ is given by $\Phi\circ\widetilde{\tau}_i\circ\Phi^{-1}$.
\end{proof}

\begin{cor}\label{group_action_corr}
    The dynamics of $\mathfrak{C}_i$ on $\widetilde{\cT}_i$ is equivalent to the action of the group $\langle\eta, \tau_i\rangle$ of conformal automorphisms of $\widetilde{\cT}_i$.
\end{cor}

\begin{proof}
    By Lemma~\ref{aut_corr_lem} it is clear that the forward branches of $\mathfrak{C}_i$ on $\widetilde{\cT}_i$ are given by $\tau_i\circ\eta, \cdots, \tau^{n_ip_i-1}_i\circ\eta$. Thanks to the relation $\tau_i = (\tau_i^2\circ\eta)\circ(\tau_i\circ\eta)^{-1}$, we conclude that
    $$
    \langle\tau_i\circ\eta, \cdots, \tau^{n_ip_i-1}_i\circ\eta\rangle = \langle\eta, \tau_i\rangle.
    $$
    Hence, the corollary follows.
\end{proof}

\begin{prop}\label{orbifold_isom_prop}
    The group $\langle\eta, \tau_i\rangle$ acts properly discontinuously on $\widetilde{\cT_i}$, and the quotient orbifold $\widetilde{\cT_i}/\langle\eta, \tau_i\rangle$ is biholomorphic to $\mathbb{D}/\widehat{\Gamma}_i$. 
\end{prop}

\begin{proof}
Let $\Phi$ be as in the proof of Lemma~\ref{aut_corr_lem}. Recall that $\widetilde{\cT}_i$ has $p_i$ components $\dT_{i,1}, \dT_{i,2}, \cdots, \dT_{i, p_i}$, where $\dT_{i,j} := \Phi(\mathbb{D}\times\{j\})$, for $1\leq j\leq p_i$. The interiors of these $p_i$ components divide $\partial\mathfrak{D}\cap\widetilde{\cT}_i$ into $p_i$ open arcs. We shall refer to the collection of these $p_i$ open arcs as the collection of \emph{welding lines}. 
    
Note that $\Phi\vert_{\D\times\{j\}}$ maps $\text{int}(\psi_{\rho_i}(\Pi))\times\{j\})$ conformally onto $\dT_{i,j}\cap\text{int}\widetilde{\cT}^0_i$, and extends homeomorphically to the boundary (in $\mathbb{D}\times\{j\}$).  After possibly pre-composing the maps $\Phi|_{\D\times\{j\}}$ with suitable powers of $M_{\omega_i}$, we can assume that $\Phi\vert_{\D\times\{j\}}$ takes the bi-infinite geodesic $\psi_{\rho_i}(C_{1,j})\times\{j\}\subset \partial\psi_{\rho_i}(\Pi)\times\{j\}$ to the welding line in $\dT_{i,j}$, for $1\leq j\leq p_i$. Let 
    $$
    G_j:= \{f\in \langle\eta, \tau_i\rangle: f(\dT_{i,j}) = \dT_{i,j}\}
    $$
    be the stabilizer of $\dT_{i,j}$ under the action of $\langle\eta, \tau_i\rangle$. 
Note that by definition, the cyclic group $\langle\tau_i\rangle$ generated by the conformal automorphism $\tau_i$ acts transitively on the components of $\widetilde{\cT}_i$. Therefore, it is enough to show that $G_1$ acts properly discontinuously on $\dT_{i,1}$, and that the quotient orbifold $\dT_{i,1}/G_1$ is biholomorphic to $\mathbb{D}/\widehat{\Gamma}_i$. Observe that from the proof of Lemma~\ref{aut_corr_lem}, we have the following commutative diagram.
    \begin{equation}\label{prop6.8_comm_eq}
        \begin{tikzcd}
	   {\mathbb{D}\times\{j\}} && \dT_{i,j} \\
	\\
	{\mathbb{D}} && {\mathcal{T}_i}
	\arrow["{\Phi|_{\mathbb{D}\times\{j\}}}", from=1-1, to=1-3]
	\arrow["{\widetilde{\mathcal{R}}|_{\mathbb{D}\times\{j\}}}"', from=1-1, to=3-1]
	\arrow["{\mathcal{R}}", from=1-3, to=3-3]
	\arrow["{\mathfrak{X}_i}"', from=3-1, to=3-3]
\end{tikzcd}
    \end{equation}

\noindent For the sake of brevity, for $1\leq j\leq p_i$, we denote $\widetilde{\cR}|_{\mathbb{D} \times \{j\}}$ simply by $\widetilde{\cR}_j$, the restriction $\cR|_{\dT_{i,j}}$ by $\cR_j$, and $\Phi|_{\mathbb{D} \times \{j\}}$ by $\Phi_j$.  By definition of (factor) Bowen-Series maps (see Section~\ref{fBS_subsec}), we have, 
\begin{equation}\label{fbS_def_eq}
    \widetilde{\cR}_j\circ A^{\text{BS}}_{\Gamma_i} = A^{\text{fBS}}_{\Gamma_i}\circ \widetilde{\cR}_j,\ \mathrm{on}\ \left( \D\setminus\Int{\psi_{\rho_i}(\Pi)} \right)\times\{j\},
\end{equation}
where by $A^{\text{BS}}_{\Gamma_i}$ we mean the possibly discontinuous Bowen-Series map that descends down to $A^{\text{fBS}}_{\Gamma_i}$ under the $n_i$-fold branched cover $w\mapsto w^{n_i}$. In addition, by the definition of $\widetilde{\tau}_i$ and $\tau_i$, we see that the following holds,
\begin{equation}\label{uniformising_maps_eq}
    \tau_i^{j-1}\circ\Phi_1\circ\widetilde{\tau}_i^{-(j-1)} = \Phi_j,\ \mathrm{on}\ \D\times\{j\}.
\end{equation}
Now by Relation~\ref{prop6.8_comm_eq}, there exists a local inverse branch of $\cR_j$ for which the following identity holds between the germs of functions near $\partial \psi_{\rho_i}(C_{1,j})\times\{j\}$, for $1\leq j\leq p_i$.

\begin{equation}\label{long_derivation_eq}
    \begin{split}
        \Phi_j\circ A_{\Gamma_i}^{\text{BS}} &= (\cR^{-1}_j\circ\mathfrak{X}_i\circ\widetilde{\cR}_j)\circ A_{\Gamma_i}^{\text{BS}}\\
        &= \cR^{-1}_j\circ\mathfrak{X}_i\circ A_{\Gamma_i}^{\text{fBS}}\circ \widetilde{\cR}_j \hspace{2mm}\text{        (by \ref{fbS_def_eq})}\\
        &= \cR^{-1}_j\circ S\circ\mathfrak{X}_i\circ\widetilde{\cR}_j\hspace{2mm}\text{        (by \ref{fBS_conj_eq})}\\
        &= \cR^{-1}_j\circ S\circ\mathfrak{X}_i\circ(\mathfrak{X}_i^{-1}\circ \cR_j\circ \Phi_j)\hspace{2mm}\text{        (by definition of }\widetilde{\cR})\\
        &= \cR^{-1}_j\circ S\circ\cR_j\circ \Phi_j\\
        &= \cR^{-1}_j\circ\cR_j\circ\eta\circ \Phi_j\hspace{2mm}\text{        (by \ref{alg_char_cd})}\\ 
        &= f_j\circ\Phi_j,\hspace{2mm}\text{  
        (for some }f_j\in G_j).
    \end{split}
\end{equation}
Note that the local germ of $\Phi_j\circ A_{\Gamma_i}^{\text{BS}}\circ \Phi^{-1}_j$, near $\Phi_j\left(\partial \psi_{\rho_i}(C_{1,j})\times\{j\}\right)$, is precisely the local germ of the generator $\rho_i(g_{1,j})$ of $\Gamma_i$, conjugated by $\Phi_j$. 
We denote by $\gamma_j\in G_1$ the element $\tau_i^{-(j-1)}\circ f_j\circ \tau_i^{j-1}$, and note that
by Equations~\eqref{uniformising_maps_eq} and~\eqref{long_derivation_eq}, we have the relation
\begin{equation}\label{generators_eq}
    \widetilde{\tau}_i^{-(j-1)}\circ A_{\Gamma_i}^{\text{BS}}\circ\widetilde{\tau}_i^{j-1} = \Phi_1^{-1}\circ\tau_i^{-(j-1)}\circ f_j\circ \tau_i^{j-1}\circ \Phi_1=\Phi_1^{-1}\circ\gamma_j\circ \Phi_1,  
\end{equation}
near $\partial \psi_{\rho_i}(C_{1,j})\times\{1\}$.

Further, by construction, $\tau_i^{p_i}$ restricts to an order $n_i$ automorphism of $\dT_{i,1}$, and $\Phi_1$ conjugates $M_{\omega_i}$ to $\tau_i^{p_i}$. Therefore, it follows by the Identity Theorem and Equations~\eqref{long_derivation_eq} and~\eqref{generators_eq} that the generators of $\widehat{\Gamma}_i$, under conjugation by $\Phi_1$, generate $G_1$. Set 
$$
\mathfrak{S}:= \Phi_1(\psi_{\rho_i}(\widehat{\Pi})),
$$
\noindent where $\widehat{\Pi}$ is the fundamental domain of $\widehat{\Gamma}_{n_i, p_i}$ (see Section~\ref{fBS_subsec}). The maps $\{\gamma_j\}_{j\in \cI}$ and $\tau_i^{p_i}$ furnish the closed fundamental domain $\mathfrak{S}$ with side-pairings that are conjugate to the preferred side-pairing transformations of $\widehat{\Gamma}_i$.

 It now follows that the group $\langle\eta, \tau_i\rangle$ acts properly discontinuously on $\widetilde{\cT}_i$, admitting $\mathfrak{S}$ as its closed fundamental domain. And finally, the map $\Phi$ would descend down the quotient to give an explicit biholomorphism between the orbifolds $\widetilde{\cT_i}/\langle\eta, \tau_i\rangle$ and $\mathbb{D}/\widehat{\Gamma}_i$.  
\end{proof}

\begin{remark}\label{recover_rep_rem}
       The involution $\eta$ permutes the $p_i$ components of $\widetilde{\cT}_i$, as the set $\widetilde{\cT}_i$ is $\eta$-invariant. Let us denote this permutation by $\sigma_i$. Notably, $\sigma_i$ coincides with the permutation of the bi-infinite geodesics $\{\psi_{\rho_i}(C_{1,j})\}_{j=1}^{p_i}$, induced by the side-pairing transformations of $\Gamma_i$ on $\partial\psi_{\rho_i}(\Pi)$. By Lemma~\ref{aut_corr_lem}, the map $f_j$ in the proof of Proposition~\ref{orbifold_isom_prop} is of the form $\tau_i^{k_j}\circ\eta$, for some $1\leq k_j\leq p_i$, where $k_j$ is dictated by the permutation $\sigma_i$. Now note that by Equation~\ref{generators_eq}, the representation $(\widehat{\rho}_i:\widehat{\Gamma}_{n_i,p_i}\rightarrow\widehat{\Gamma}_i)\in \text{Teich}(\widehat{\Gamma}_{n_i,p_i})$ can be equivalently defined as follows:
       \begin{equation*}
           \widehat{\rho}_i(g_{1,j}):= \Phi_1^{-1}\circ\tau_i^{-(j-1)}\circ (\tau_i^{k_j}\circ\eta)\circ \tau_i^{j-1}\circ\Phi_1,\hspace{2mm}\text{and}\hspace{2mm}\widehat{\rho}_i(M_{\omega_i}):= \Phi_1^{-1}\circ\tau_i^{p_i}\circ\Phi_1 = M_{\omega_i}.
       \end{equation*}
       Therefore, the action of the multi-valued map $\mathfrak{C}_i$ on $\widetilde{\cT}_i$, helps us recover not only the group $\widehat{\Gamma}_i$, but also the representation $\widehat{\rho}_i$ in $\text{Teich}^{\omega_i}(\Gamma_{n_i,p_i})$.   
\end{remark}

\subsection{Map-like dynamics in $\mathfrak{C}$}\label{map-like_subsubsec} Fix an index $j\in \{1, 2, \cdots, r\}$. In the following proposition, we shall prove the second part of Theorem~\ref{corr_recover_thm}.

\begin{prop}\label{recover_Blaschke_prop}
    There exists a forward branch (respectively, a backward branch) of $\mathfrak{C}$ carrying $\widetilde{\cV}_j\cap \mathfrak{D}$ (respectively, $\widetilde{\cV}_j\setminus \overline{\mathfrak{D}}$) onto itself, and this branch is conformally conjugate to the Blaschke product $B_j$.
\end{prop}
\begin{proof}
    The map
    $$
    \Big(\cR\big|_{\mathfrak{D}}\Big)^{-1}\circ \cR\circ \eta: \widetilde{\cV}_j\cap \mathfrak{D}\rightarrow\widetilde{\cV}_j\cap \mathfrak{D}
    $$
    is a forward branch of the correspondence $\mathfrak{C}$, and clearly it is conformally conjugate to $S\big|_{\cV_j}$ via the univalent map $\cR: \widetilde{\cV}_j\cap \mathfrak{D}\rightarrow \cV_j$. Further, it is easy to see that the map 
    $$
    \eta\circ \Big(\cR\big|_{\mathfrak{D}}\Big)^{-1}\circ \cR: \widetilde{\cV}_j\setminus \overline{\mathfrak{D}} \rightarrow \widetilde{\cV}_j\setminus \overline{\mathfrak{D}}
    $$
    is a backward branch of $\mathfrak{C}$. The involution $\eta$ restricts to a conformal conjugacy between the backward and forward branches of the correspondence $\mathfrak{C}$. Now, since $S\big|_{\cV_j}$ is conformally conjugate to $B_j$, the result follows.
\end{proof}

\begin{remark}\label{nodal_rem}
    Let $\mathcal{Z}:=\Sigma\setminus\bigcup_{i=1}^{l+r} \widetilde{\cW}_i$. Suppose that $\cR$ is injective on $\overline{\mathfrak{D}}$. Then, by replacing $\widetilde{\cV}_j\cap \mathfrak{D}$ (respectively, $\widetilde{\cV}_j\setminus \overline{\mathfrak{D}}$) with $\mathcal{Z}\cap \overline{\mathfrak{D}}$ (respectively, $\mathcal{Z}\setminus \mathfrak{D}$) in Proposition~\ref{recover_Blaschke_prop}, we obtain a topological conjugacy between $R$ and a forward branch (respectively, a backward branch) of the correspondence $\mathfrak{C}$, carrying $\mathcal{Z}\cap \overline{\mathfrak{D}}$ (respectively, $\mathcal{Z}\setminus \mathfrak{D}$) onto itself. Otherwise, one needs to pass to a nodal surface where the map $\cR$ descends to an injective map. This is precisely why a nodal surface was constructed in \cite[\S 5.2]{mj2023matings} to describe matings of Fuchsian groups with polynomials whose Julia sets are not Jordan curves.
\end{remark}

\subsection{Unification of Fuchsian groups and Blaschke products in correspondences}\label{unify_subsec}

We now combine the results proved so far to establish the main theorem of this section, i.e., Theorem~\ref{corr_recover_thm}. This is followed by two concrete instantiations, illustrating examples from Section~\ref{examples_sec}. Finally, we complete the proof of Theorem~\ref{mating_thm_intro}, thereby concluding a key component of the narrative introduced in Section~\ref{intro_sec}.

\begin{proof}[Proof of Theorem~\ref{corr_recover_thm}]
The first item is the content of Corollary~\ref{group_action_corr} and Proposition~\ref{orbifold_isom_prop}, while the second item follows from Proposition~\ref{recover_Blaschke_prop}.  
\end{proof}

\subsubsection{Illustrations}
1) Let $S:\overline{\cD}\rightarrow\widehat{\mathbb{C}}$ be a conformal mating as in Example~\ref{111_ex}. The domain $\cD$ is an annulus bound by a blue teardrop-shaped Jordan curve and a red topological quadrilateral (see Figure~\ref{111_fig}). The groups $\Gamma_1$ and $\Gamma_2$ that go into the construction of $S$ are $\Gamma_1:=\Gamma_{3,1}$ and $\Gamma_2:=\rho(\Gamma_{1, 4}^2)$, for some $(\rho:\Gamma_{1, 4}^2\rightarrow\Gamma_2)\in\text{Teich}(\Gamma_{1,4}^2)$. Here, the set $\widetilde{\cT_1}$ is an $\eta$-invariant topological disk, and the set $\widetilde{\cT}_2$ is a union of four topological disks; two of which are $\eta$-invariant, while the other two form a $2$-cycle under $\eta$ (see Figure~\ref{new_fig}). The associated correspondence $\mathfrak{C}$, which is defined on a torus, combines the actions of a quadratic Blaschke product $B$, and the groups $\Gamma_1, \Gamma_2$. 

\begin{figure}[h!]
    \captionsetup{width=0.98\linewidth}
    \begin{tikzpicture}
    \node[anchor=south west,inner sep=0] at (0,0) {\includegraphics[width=0.5\textwidth]{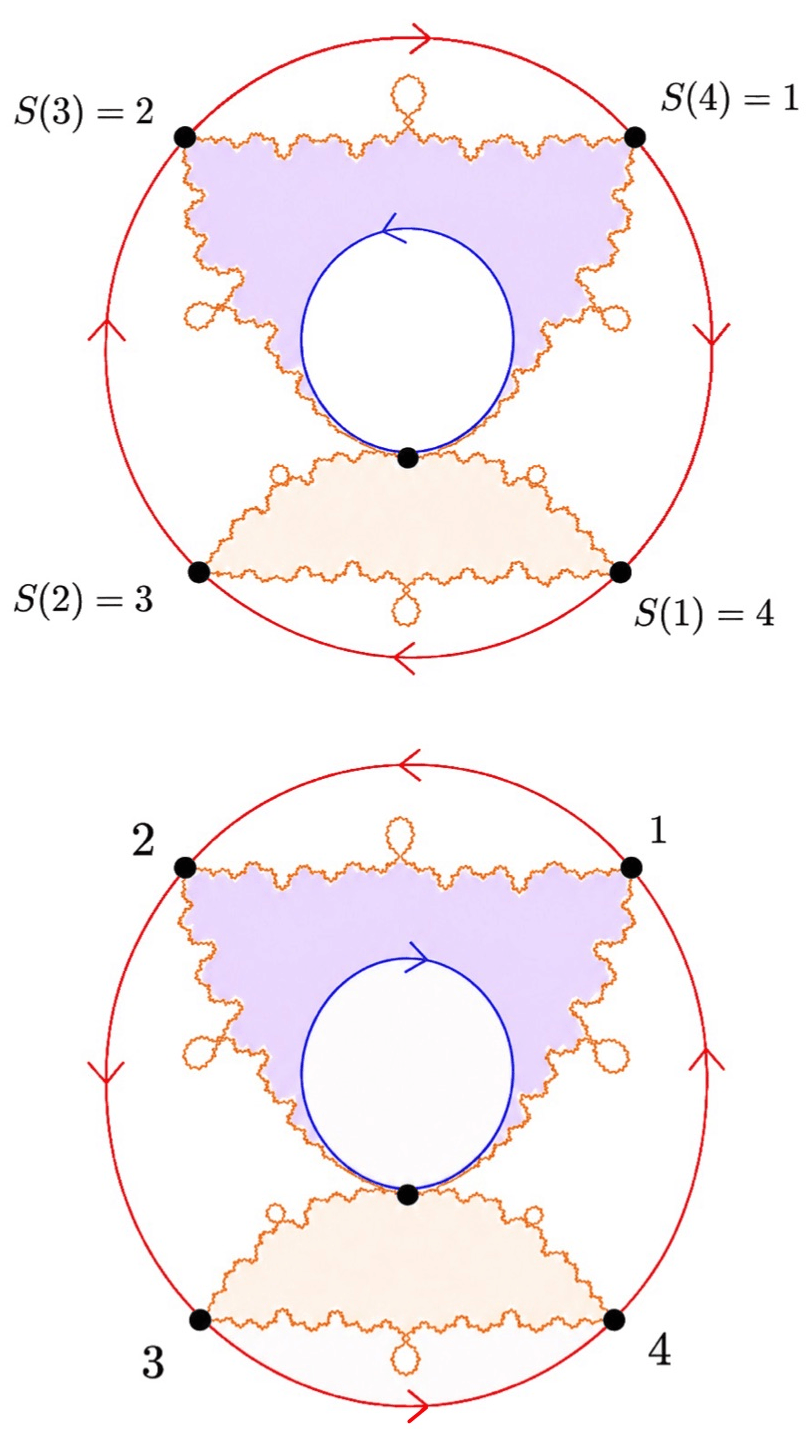}}; 
    \node[anchor=south west,inner sep=0] at (6.6,1.2) {\includegraphics[width=0.5\textwidth]{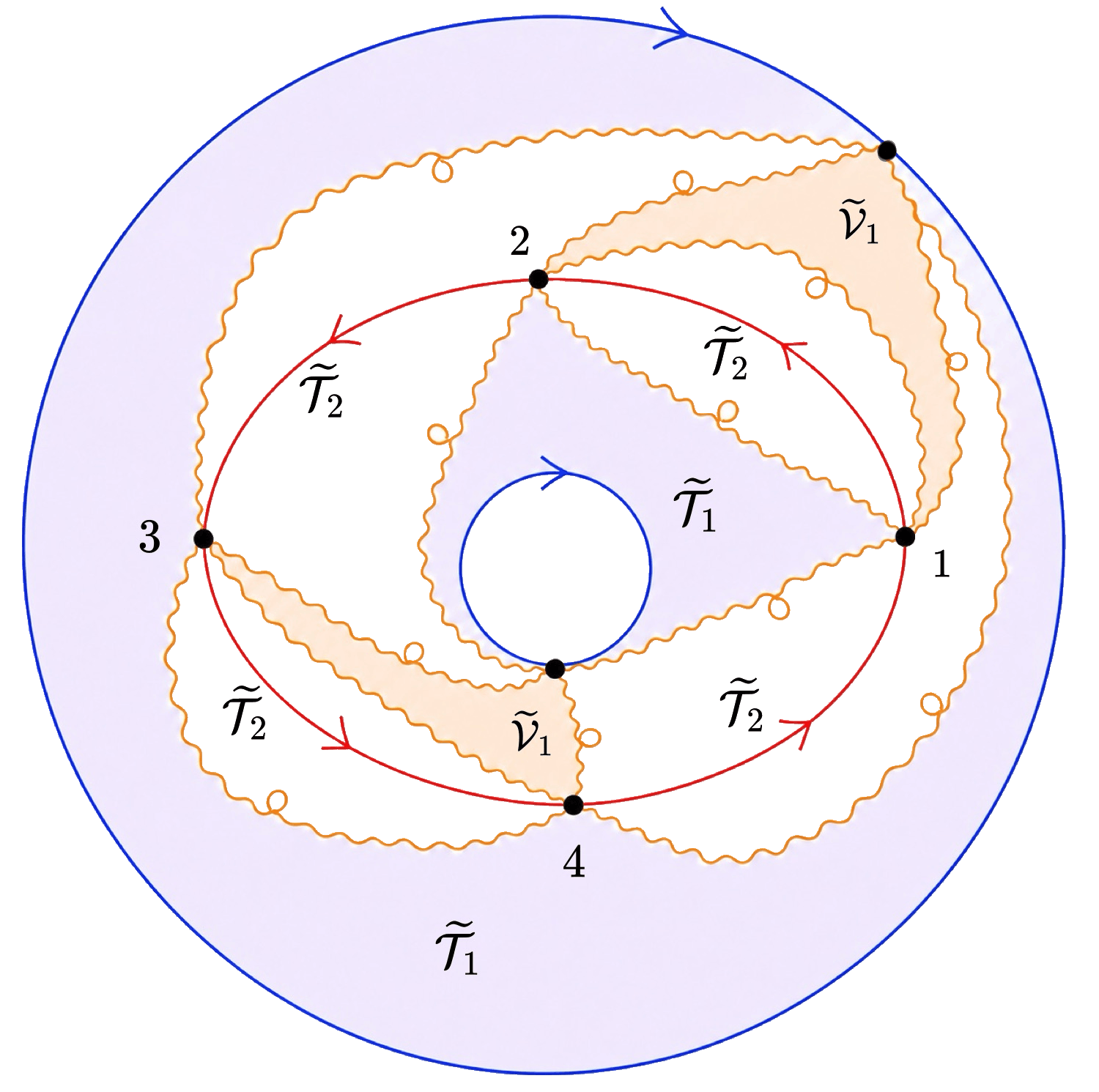}};
    \end{tikzpicture}
    \caption{Left: Shown here are topological models of two copies of the inversive domain $\mathcal{D}$ of Example~\ref{111_ex} (cf. Figure~\ref{111_fig}). The dynamics of $S$ in the region colored purple is equivalent to that of $A_{\Gamma_{3, 1}}^{\text{fBS}}$. The dynamics of $S$ in the invariant component colored beige is equivalent to that of $B$. Finally, the dynamics of $S$ in the invariant component containing the outer red circle is equivalent to that of $A_{\Gamma}$. Right: Illustrated here is the annulus obtained by gluing the two copies of $\mathcal{D}$ along the outer red circles in the left figures via the map $S$. The blender surface $\Sigma$ can be obtained by gluing the inner and outer blue circles bounding the annulus. The two black marked points on the blue circles map to each other under the B-involution $S|_{\partial\mathcal{D}}$. The gluing yields an $\eta$-invariant open topological disk $\widetilde{\cT}_1$. The open set $\widetilde{\cT}_2$, a union of four open topological disks, is also labeled. The two regions colored in beige form a $2$-cycle under $\eta$, and their union is the set $\widetilde{\cV}_1$. The blue and red Jordan curves divide the surface $\Sigma$ into two components, one of which is $\mathfrak{D}$. By Theorem~\ref{corr_recover_thm}, the action of the correspondence $\mathfrak{C}$ on $\widetilde{\cT}_1,$ and $\widetilde{\cT}_2$ are equivalent to those of the groups $\widehat{\Gamma}_{3,1},$ and $  \Gamma_{1,4}^2$, respectively. Further, the forward branch of the correspondence preserving $\widetilde{\cV}_1\cap\mathfrak{D}$ is equivalent to the Blaschke product $B$. The involution $\eta$ has two fixed points on the red Jordan curve (one on the counter-clockwise arc connecting points labeled $2$ and $3$, and the other on the counter-clockwise arc connecting points labeled $4$ and $1$), and two fixed points on the blue Jordan curve one of which is the point marked black.}
    \label{new_fig}
\end{figure}

2) Let $S:\overline{\cD}\rightarrow\widehat{\mathbb{C}}$ be a conformal mating as in Example~\ref{111''_ex}. The triply connected domain $\cD$ is bound by a red cardioid-shaped Jordan curve, a blue triangle, and a blue quadrilateral (see Figure~\ref{111''_fig}). The groups $\Gamma_1, \Gamma_2$, and $\Gamma_3$ used to manufacture $S$ are $\Gamma_1:=\Gamma_{5,1}$, $\Gamma_2:=\Gamma_{1,3}$, and $\Gamma_3:=\rho(\Gamma_{1, 4}^2)$, for some $(\rho:\Gamma_{1, 4}^2\rightarrow\Gamma_3)\in\text{Teich}(\Gamma_{1,4}^2)$. In this case, the set $\widetilde{\cT_1}$ is an $\eta$-invariant topological disk, $\widetilde{\cT}_2$ is a union of three topological disks (one of which is $\eta$-invariant, while the other two form a $2$-cycle under $\eta$), and $\widetilde{\cT}_3$ is a union of four topological disks (two of which are $\eta$-invariant, while the other two form a $2$-cycle under $\eta$). The associated correspondence $\mathfrak{C}$, which is defined on a hyperelliptic Riemann surface of genus $2$, combines the actions of the groups $\Gamma_1, \Gamma_2$, and $\Gamma_3$.

\subsubsection{Proof of Theorem~\ref{mating_thm_intro}}\label{mating_thm_intro_proof_subsec}

By Proposition~\ref{mating_surf_exists_prop}, we know that the groups $\{\widehat{\Gamma}_i\}_{i=1}^l$ and the Blaschke products $\{B_j\}_{j=1}^r$ can be conformally mated, and that the associated blender surface is hyperelliptic. The conclusion now follows directly from Theorem~\ref{corr_recover_thm}.

\section{Moduli spaces of correspondences in Hurwitz spaces}\label{Hurwitz_sec} Here we turn to the second main theme of our paper and complete the proof of Theorem~\ref{hurwitz_thm_intro}. To that end, we begin by identifying the equivalence classes of our correspondences with a suitable subset of the Hurwitz space of degree~$d$ meromorphic maps on genus~$g$ Riemann surfaces with~$q$ ordered marked points. We then construct a map from the product of the parameter spaces of Fuchsian groups (as described in Section~\ref{mateable_maps_subsec}) and hyperbolic Blaschke products to this Hurwitz space and show that it is injective.

\subsection{Standard conformal matings}\label{standard_matings_subsec}

Consider the groups $\Gamma_i \in \text{Teich}^{\omega_i}(\Gamma_{n_i,p_i})$, for $1\leq i\leq l$, and the maps $B_j \in \mathcal{B}_{d_j}$, for $1\leq j\leq r$, where $\mathcal{B}_{d_j}$ denotes the parameter space of degree $d_j$ hyperbolic Blaschke products upto M\"obius conjugation. Moving forward, we can and will assume that the attracting fixed point of $B_j$ is at $0$; and that it fixes $1$ on $\partial\D$.

By Lemma~\ref{combi_lem}, there exists a critically fixed polynomial $P$ with bounded immediate basins $\{U_i\}_{i = 1}^{l+r}$, such that $\deg(P|_{U_i}) = n_ip_i-1$, for $1\leq i\leq l$, and $\text{deg}(P|_{U_{l+j}}) = d_{j}$, for $1\leq j\leq r$. Throughout this section, we shall fix such a polynomial $P$. Let the finite critical points of $P$ be $\{c_1, c_2, \cdots, c_{l+r}\}$, with $c_i\in U_i$, for $i\in\{1,\cdots, l+r\}$. In addition, we also fix points $\xi_i\in \partial U_i$ such that $P(\xi_i) = \xi_i$, for $i\in\{1,\cdots, l+r\}$ (we note that the $\xi_i$'s are not necessarily distinct). Let $r_i:\D\rightarrow U_i$ be the uniformizer of $U_i$, such that $r_i(0) = c_i$, and $r_i(1) = \xi_i$. Since $U_i$ is a Jordan domain, we have a homeomorphic extension $r_i:\overline{\D}\to\overline{U_i}$, for $i\in\{1,\cdots, l+r\}$. Then the Blaschke products 
$$
\mathscr{B}_i\Big|_{\mathbb{D}}:= r_i^{-1}\circ P\circ r_i,\ i\in\{1,\cdots,l+r\},
$$
fix the points $1\in\partial\D$, and have a superattracting fixed point of maximal local degree at $0$. As a result, each $\mathscr{B}_i$ is of the form $w\mapsto w^m$, for appropriate $m\geq 2$. 

For $1\leq i\leq l$, recall from Section~\ref{mateable_maps_subsec} the normalized topological conjugacies $h_i:\mathbb{S}^1\to\mathbb{S}^1$ from $\mathscr{B}_i$ to $A_{\Gamma_i}$.
Further, for $1\leq j\leq r$, we normalize the quasisymmetric conjugacy $h_{l+j}:\mathbb{S}^1\rightarrow\mathbb{S}^1$ between $\mathscr{B}_{l+j}$ and $B_j$, by requiring $h_{l+j}(1) = 1$.  Note that these normalizations uniquely determine the topological semi-conjugacies $\zeta_i^{-1}:\mathbb{S}^1\to\partial U_i$, $i\in \{1,2,\cdots,l\}$, between $A_{\Gamma_i}$ and $P$, and $\zeta_{l+j}^{-1}:\mathbb{S}^1\to\partial U_{l+j}$, $j\in \{1,\cdots,r\}$, between $B_{j}$ and $P$; where 
$$
\zeta_i:= h_i\circ r_i^{-1},\ \text{for}\ i\in \{1,2,\cdots,l+r\}.
$$
With the collection of maps $\{\zeta_i\}_{i=1}^{l+r}$, one can run the proof of Theorem~\ref{conf_mating_thm}, to get a conformal mating $S:\overline{\cD}\rightarrow\widehat{\C}$ that replaces the dynamics of $P$ on $U_i$ with that of $A_{\Gamma_i}$, for $1\leq i\leq l$, and on $U_{l+j}$ with that of $B_j$, for $1\leq j\leq r$. We call such a map $S$, a \emph{standard conformal mating} associated with the collections $\{\Gamma_i\}_{i=1}^l$ and~$\{B_j\}_{j=1}^r$. 

\begin{prop}\label{uniqueness_standard_mating_prop}
    Let $S:\overline{\cD}\rightarrow\widehat{\C}$ and $S':\overline{\cD'}\rightarrow\widehat{\C}$ be two standard conformal matings associated with the collections $\{\Gamma_i\}_{i=1}^l$ and $\{B_j\}_{j=1}^r$. Then there exists a M{\"o}bius map $M$ such that $S_1 = M\circ S_2\circ M^{-1}$.
\end{prop}
\begin{proof}
    Let us denote by $\{\mathfrak{X}_i\}_{i=1}^{l+r},\ \mathfrak{X}_P$, and by $\{\mathfrak{X}'_i\}_{i=1}^{l+r},\ \mathfrak{X}'_P$, the conformal conjugacies associated with the maps $S$ and $S'$, respectively (see the proof of Theorem~\ref{conf_mating_thm}). For $1 \leq i \leq l$, the map $S$ is univalent on every component of $\mathcal{W}_i$ except $\mathcal{T}_i$. Similarly, for $1 \leq j \leq r$, $S$ is univalent on every component of $\mathcal{W}_{l+j}$ except $\mathcal{V}_j$. An analogous statement holds for $S'$ in its domain of definition $\overline{\mathcal{D}'}$. Therefore, for each $1\leq i\leq l+r$, the maps, $\mathfrak{X}_P'\circ\mathfrak{X}_P^{-1}$, $\mathfrak{X}_i'\circ\mathfrak{X}_i^{-1}$, and,
    $$
    (S')^{-n}\circ\mathfrak{X}_i'\circ\mathfrak{X}_i^{-1}\circ S^{\circ n},
    $$ 
    defined on the pre-images $\left\{S^{-n}\left(\overline{\cT}_i\right)\right\}_{i=1}^l$ and $\left\{S^{-n}\left(\overline{\cV_j}\right)\right\}_{j=1}^r$, for $n\geq 1$, patch together to give a homeomorphism $M:\widehat{\C}\rightarrow\widehat{\C}$, that is conformal away from $\mathfrak{X}_P(\mathcal{J}(P))$. The basin of infinity of $P$ is a John domain \cite[Chapter~7, Theorem~3.1]{carleson1996complex}; hence $\mathcal{J}(P)$ is $W^{1,1}-$removable \cite{jones2000removability}. By \cite[Theorem~2.7]{lyubich2020david}, the set $\mathfrak{X}_P(\mathcal{J}(P))$ is conformally removable. Thus, the map $M$ extends conformally to all of $\widehat{\C}$ giving the desired M\"obius map. 
\end{proof}

For a standard conformal mating $S:\overline{\cD}\rightarrow\widehat{\C}$ (as above), let $\Sigma$ and $\cR$ be the associated blender surface and the uniformizing meromorphic map, respectively. The following proposition records the extent of uniqueness of such an association. 

\begin{prop}\label{uniqeness_of_mero_prop}
    Suppose there exist $4$-tuples $(\Sigma_1, \cR_1, \mathfrak{D}_1, \eta_1)$, $(\Sigma_2, \cR_2, \mathfrak{D}_2, \eta_2)$ satisfying:
    \begin{enumerate}
        \item $\eta_i$ is an involution of $\Sigma_i$, with $\eta_i(\partial\mathfrak{D}_i) = \partial\mathfrak{D}_i$,
        \item $\cR_i|_{\mathfrak{D}_i}$ is injective,
        \item $\cR_i(\overline{\mathfrak{D}_i}) = \overline{\cD}$, and,
        \item $S|_{\overline{\cD}} = \cR_i\circ\eta_i\circ(\cR_i|_{\overline{\mathfrak{D}}_i})^{-1}$,
    \end{enumerate}
    for $i\in\{1, 2\}$. Then, there exists a biholomorphism $M: \Sigma_1 \rightarrow\Sigma_2$, such that $\eta_2\circ M = M\circ \eta_1$, $M(\mathfrak{D}_1) = \mathfrak{D}_2$, and $\cR_1 = \cR_2\circ M$.
\end{prop}
\begin{proof}
    Set $\Sigma_i^0:= \Sigma\setminus(\cR_i|_{\partial\mathfrak{D}_i})^{-1}(X)$, $i\in\{1,2\}$, where $X$ is the set of singular points on $\partial\cD$. Define $M:\Sigma_1^0\rightarrow\Sigma_2$ as follows,
    \begin{equation*}
        M(z):= \begin{cases*}
        \big(\cR_2\big|_{\mathfrak{\overline{\mathfrak{D}}_2}}\big)^{-1}\circ \cR_1(z), & if $z\in \Sigma_1^0\cap\overline{\mathfrak{D}}_1$,\\
        \eta_2\circ\big(\cR_2\big|_{\mathfrak{\overline{\mathfrak{D}}_2}}\big)^{-1}\circ \cR_1\circ\eta_1, &if $z\in \Sigma_1^0\setminus\overline{\mathfrak{D}}_1$.
        \end{cases*}
    \end{equation*}
    Note that the function $\big(\cR_2|_{\mathfrak{\partial\mathfrak{D}_2}}\big)^{-1}\circ \cR_1(z)$ conjugates $\eta_1|_{\partial\mathfrak{D}_1\cap\Sigma_1^0}$ to $\eta_2|_{\partial\mathfrak{D}_2\cap\Sigma_2^0}$. Therefore, $M$ is continuous on $\Sigma_1^0$ and conformal away from $\Sigma_1^0\cap\partial\mathfrak{D}_1$. Further, $\Sigma_1^0\cap\partial\mathfrak{D}_1$ is a disjoint union of non-singular analytic curves, and is hence conformally removable. Therefore, $M$ is an injective holomorphic map on $\Sigma_1^0$. Finally, by the Riemann Removability Theorem, $M$ extends to $\Sigma_1$ univalently. From the definition of $M$, it now follows that $\eta_2\circ M = M\circ \eta_1$, $M(\mathfrak{D}_1) = \mathfrak{D}_2$, and $\cR_1 = \cR_2\circ M$.  
\end{proof}

\subsection{Topological equivalence of limit sets}\label{limit_sets_equiv_subsec} Let $\{\Gamma_i\}_{i=1}^l$, $\{B_j\}_{j=1}^r$, $S$, $\Sigma$, and $\cR$ be as described in Section~\ref{standard_matings_subsec}. Further let $\{\mathfrak{X}_i\}_{i=1}^{l+r}$ and $\mathfrak{X}_P$ be the conformal conjugacies associated with the map $S$ (see the proof of Theorem~\ref{conf_mating_thm}). 

Consider for $1\leq i\leq l$, and $1\leq j\leq r$, a collection of groups $\Gamma_i'\in \text{Teich}^{\omega_i}(\Gamma_{n_i, p_i})$, and maps $B_j'\in \mathcal{B}_{d_j}$. Let $S':\overline{\cD'}\rightarrow\widehat{\C}$ be a standard conformal mating associated with $\{\Gamma_i'\}_{i=1}^l$ and $\{B_j'\}_{j=1}^r$. By results in Section~\ref{b_involutions_sec}, the map $S'$ yields a triple $(\Sigma', \eta', \cR')$, where $\Sigma'$ is the blender surface with hyperelliptic involution $\eta'$, and $\cR'$, the uniformizing meromorphic map. Let $\mathfrak{C}'$ be the correspondence defined by $\cR'$ and $\eta'$ on $\Sigma'$ (as in Equation~\ref{corr_eq}). In this subsection, we will establish the existence of a quasiconformal map that conjugates the action of $\mathfrak{C}$ on $\Lambda(\mathfrak{C})$ to the action of $\mathfrak{C}'$ on~$\Lambda(\mathfrak{C'})$.

Let us denote the conformal conjugacies associated with $S'$ by $\{\mathfrak{X}_i'\}_{i=1}^{l+r}$, and $\mathfrak{X}_P'$. The map $S'$ can be re-constructed from $S$ by the following standard quasiconformal surgery. For $1\leq i\leq l$, we make use of the quasiconformal conjugacy between $A_{\Gamma_i}$ and $A_{\Gamma_i'}$ to deform the complex structure on $\cW_i$. Similarly, for $1\leq j\leq r$, a quasisymmetric conjugacy between $B_j\vert_{\mathbb{S}^1}$ and $B_j'\vert_{\mathbb{S}^1}$ allows us to re-define the map $S$ on $\cV_j$ (cf. \cite{mcmullen1988automorphisms}, \cite[\S 7.5]{branner2014quasiconformal}). After straightening, this gives us a meromorphic map $\widetilde{S}$ that replaces the dynamics of $S$, on appropriate domains, with that of $\{A_{\Gamma_i'}\}_{i=1}^l$ and $\{B_j'\}_{j=1}^r$. In addition, there exists a quasiconformal map that conjugates the dynamics of $S$ to that of $\widetilde{S}$, away from the domains $\{\cV_j\}_{j=1}^r$. Since by Proposition~\ref{uniqueness_standard_mating_prop}, the map $\widetilde{S}$ is M\"obius conjugate to $S'$; there exists a quasiconformal map $\varphi$ that conjugates the action of $S$ to that of $S'$, away from $\bigcup_{j=1}^r\cV_j$.

The following proposition would now help us lift the map $\varphi$ to a quasiconformal map between the blender surfaces $\Sigma$ and $\Sigma'$. Recall that by Proposition~\ref{mating_surf_exists_prop}, $\Sigma$ and $\Sigma'$ are hyperelliptic Riemann surfaces with involutions, say, $\eta$ and $\eta'$, respectively.

\begin{prop}\label{qc_lift_prop}
    There exists a quasiconformal map $\widetilde{\varphi}:\Sigma\rightarrow\Sigma'$, such that $\cR'\equiv \varphi\circ\cR\circ\widetilde{\varphi}^{-1}$ outside $\widetilde{\phi}\left(\eta\left(\left(\cR\vert_{\mathfrak{D}}\right)^{-1}\left(\bigcup_{j=1}^r\cV_j\right)\right)\right)$, and $\widetilde{\varphi}\circ\eta\circ\widetilde{\varphi}^{-1} = \eta'$.
\end{prop}

\begin{proof}(cf. \cite[\S 5.3]{luo2025}.)
    Define a map, 
    \begin{equation*}
        \widecheck{\cR}:\Sigma\rightarrow\widehat{\C}, \hspace{6mm}\widecheck{\cR}:= \begin{cases*}
        \varphi\circ\cR, & on $\overline{\mathfrak{D}}$\\
        S'\circ\varphi\circ\cR\circ\eta, & on $\Sigma\setminus\mathfrak{D}$.
    \end{cases*}
    \end{equation*}
    The piecewise definitions patch up continuously because we have, $S\circ \cR\circ\eta = \cR$ on $\partial\mathfrak{D}$, and $\varphi\circ S = S'\circ \varphi$ on $\partial\cD$. As a result, $\widecheck{\cR}$ is quasi-regular. Let $\mu$ be the Beltrami coefficient on $\Sigma$ defined as the pull-back of the standard complex structure on $\widehat{\C}$ under the map $\widecheck{\cR}$. By construction, 
    $$
    S' = \widecheck{\cR}\circ\eta\circ\left(\widecheck{\cR}\big|_{\mathfrak{D}}\right)^{-1},
    $$
    which implies that $\mu$ is $\eta-$invariant. Let $\widecheck{\varphi}$ be the quasiconformal straightening map for $\mu$; i.e., $\widecheck{\varphi}:\Sigma\to\widetilde{\Sigma}$ is a quasiconformal homeomorphism that pulls back the standard complex structure on some Riemann surface $\widetilde{\Sigma}$ (not necessarily biholomorphic to $\Sigma)$ to the complex structure given by $\mu$ on $\Sigma$. Define the map, 
    $$
    \widetilde{\cR}:\widetilde{\Sigma}\rightarrow\widehat{\C}, \hspace{6mm} \widetilde{\cR}:=\widecheck{\cR}\circ\widecheck{\varphi}^{-1}.  
    $$
    Since $\varphi$ conjugates the action of $S$ to that of $S'$ away from $\bigcup_{j=1}^r\cV_j$, we conclude that the holomorphic branched cover $\widetilde{\cR}$ satisfies 
    $$
    \widetilde{\cR} \equiv \varphi\circ\cR\circ\widecheck{\varphi}^{-1}
    $$ 
    outside $\widetilde{\phi}\left(\eta\left(\left(\cR\vert_{\mathfrak{D}}\right)^{-1}\left(\bigcup_{j=1}^r\cV_j\right)\right)\right)$
    Further, the map $\widecheck{\varphi}\circ\eta\circ\widecheck{\varphi}^{-1}$ is an involution on $\widetilde{\Sigma}$ with $2g+2$ fixed points, where $g$ is the genus of the surface $\Sigma$. Therefore, $\widecheck{\varphi}\circ\eta\circ\widecheck{\varphi}^{-1} = \widetilde{\eta}$, where $\widetilde{\eta}$ is a hyperelliptic involution on $\widetilde{\Sigma}$. Now set, $\widetilde{\mathfrak{D}}:= \widecheck{\varphi}(\mathfrak{D})$, and observe that,
    $$
    S'\big|_{\cD'} = \widetilde{\cR}\circ\widetilde{\eta}\circ\left(\widetilde{\cR}\big|_{\widetilde{\mathfrak{D}}}\right)^{-1}.
    $$
    Therefore, the $4-$tuples $(\widetilde{\Sigma}, \widetilde{\cR}, \widetilde{\mathfrak{D}}, \widetilde{\eta})$ and $(\Sigma', \cR', \mathfrak{D}', \eta')$ both satisfy the hypotheses of Proposition~\ref{uniqeness_of_mero_prop} for the standard conformal mating $S':\overline{\mathcal{D}'}\rightarrow\widehat{\C}$. This gives us a biholomorphism $M:\widetilde{\Sigma}\rightarrow\Sigma'$ satisfying $\eta'\circ M = M\circ \widetilde{\eta}$, $M(\widetilde{\mathfrak{D}}) = \mathfrak{D}'$, and $\widetilde{\cR} = \cR'\circ M$. It is now immediately checked that the map $\widetilde{\varphi}:= M\circ\widecheck{\varphi}$ is the desired quasiconformal map with $\cR'\equiv \widecheck{\cR}\circ\widetilde{\phi}^{-1}$ and $\widetilde{\phi}\circ\eta\circ\widetilde{\phi}^{-1}\equiv \eta'$.
\end{proof}

\begin{cor}\label{equiv_limit_sets_cor}
    The map $\widetilde{\varphi}$ conjugates the action of $\mathfrak{C}$ on $\Lambda(\mathfrak{C})$ to the action of $\mathfrak{C}'$ on $\Lambda(\mathfrak{C}')$.
\end{cor}
\begin{proof}
   We continue to use the notation introduced in Proposition~\ref{qc_lift_prop}. Also recall from Section~\ref{inv_partition_corr_subsec} that $\Lambda(\mathfrak{C})=\cR^{-1}(\partial\cK)$ and $\Lambda(\mathfrak{C}')=(\cR')^{-1}(\partial\cK')$, where $\cK, \cK'$ are the non-escaping sets of $S, S'$, respectively. 
   
   Note that since $\phi$ is a conjugacy between $S$ and $S'$ away from $\bigcup_{j=1}^r\cV_j$, it carries $\partial\cK$ onto $\partial\cK'$. This fact, and the definition of the quasi-regular map $\widecheck{\cR}$, together show that $\widecheck{\cR}^{-1}(\partial\cK')=\Lambda(\mathfrak{C})$. 
   Taking the relation $\cR'\equiv \widecheck{\cR}\circ\widetilde{\phi}^{-1}$ into consideration, we now have that
   $$
   \Lambda(\mathfrak{C}')= (\cR')^{-1}(\partial\cK')=\widetilde{\phi}\left(\widecheck{\cR}^{-1}(\partial\cK')\right)=\widetilde{\phi}(\Lambda(\mathfrak{C})).
   $$
   Finally, as $\cR'=\phi\circ\cR\circ\widetilde{\phi}^{-1}$ on $\widetilde{\phi}(\Lambda(\mathfrak{C}))=\Lambda(\mathfrak{C}')$, and $\widetilde{\phi}\circ\eta=\eta'\circ\widetilde{\phi}$, the following conjugacy relation holds between the local branches of $\mathfrak{C}\vert_{\Lambda(\mathfrak{C})}$ and $\mathfrak{C}'\vert_{\Lambda(\mathfrak{C}')}$:
   \begin{align*}
\widetilde{\phi}\circ\left(\cR^{-1}\circ\cR\circ\eta\right)\circ\widetilde{\phi}^{-1}&=(\widetilde{\phi}\circ\cR^{-1}\circ\phi^{-1})\circ(\phi\circ\cR\circ\widetilde{\phi}^{-1})\circ(\widetilde{\phi}\circ\eta\circ\widetilde{\phi}^{-1})\\
&=(\cR')^{-1}\circ\cR'\circ\eta'.
   \end{align*}
\end{proof}

\subsection{Marked blender surfaces}\label{marked_blender surface_subsec} Recall that in Section~\ref{group-like_subsubsec}, for a fixed index $i\in\{1,\cdots,l\}$, the sides of the preferred fundamental polygon $\Pi$ of the group $\Gamma_{n_i,p_i}$ were denoted by $C_{1,s}$, $s\in\{1,\cdots,p_i\}$, and the uniformizing map of the corresponding tiling set component $\dT_{i,s}$ was denoted by $\Phi_s$. Henceforth, we use the same notation for all $i\in \{1, 2, \cdots, l\}$ to avoid an overload of indices (i.e., we suppress the index $i$ in the notation $C_{1,s}$ and $\Phi_s$). 

From the proof of Proposition~\ref{orbifold_isom_prop}, the $p_i$-components $\dT_{i,1}, \dT_{i,2}, \cdots, \dT_{i,p_i}$ of $\widetilde{\mathcal{T}}_i$ can be indexed so that the welding line in $\dT_{i,s}$ is given by $\Phi_s(\psi_{\rho_i}(C_{1,s})\times\{s\})$. Set $p := \sum_{i=1}^lp_i$.

We now describe a procedure for marking an ordered set of points on the blender surface $\Sigma$. The first $p+r+1$ points are marked as follows:
\begin{itemize}[leftmargin=4mm]\label{markings_item}
  \item For $1\leq i\leq l$, $1\leq s\leq p_i$, let $y_{i,s}$ be the point $\Phi_s(\psi_{\rho_i}(z_{1,s})\times\{s\})\in\dT_{i,s}$, where $z_{i,s}$ is the Euclidean mid-point of the arc $C_{1,s}$. Define $x_k$, for $1\leq k\leq p$, to be the $k$-th term in the lexicographically ordered sequence $\{y_{i,s}\}_{i,s}$. 
  
  \item For $1 \leq j \leq r$, define $x_{p+j}$ to be the point in $\widetilde{\mathcal{V}}_j \cap \mathfrak{D}$ such that $\mathcal{R}(x_{p+j}) = \mathfrak{X}_{l+j}(0)$.
  
  \item Finally, let $x_{p+r+1}$ be a point on $\Lambda(\mathfrak{C})$ satisfying $\mathcal{R}(x_{p+r+1}) = \mathfrak{X}_P(z_0)$, where $z_0$ is the landing point of the dynamical $0-$ray of $P$.
\end{itemize}
The remaining ordered points have to be marked with care to avoid repetition. We do it as follows. For $1\leq i\leq l$, and $1\leq s\leq p_i$, we let 
    $$
    w_{i,s}:= \Phi_s\Bigg(\psi_{\rho_i}\left(\exp{\left(\frac{2\pi i(s-1)}{n_ip_i}\right)}\right)\times\{s\}\Bigg),
    $$
which is one of the end-points of the unique welding line in $\dT_{i,s}$. Again, as before, for $1\leq k\leq p$, let $t_k$ be the $k$-th term in the lexicographically ordered sequence $\{w_{i,s}\}_{i,s}$. Further, for $1\leq j\leq r$, we let $t_{p+j}$ to be the point in $\overline{\mathfrak{D}}$ such that $\cR(t_{p+j}) = \mathfrak{X}_{l+j}(1)=\mathfrak{X}_P(\xi_{l+j})$ (see Section~\ref{standard_matings_subsec}). 
\begin{itemize}[leftmargin=4mm]
    \item Consider the ordered sequence $\{x_1,\cdots,x_{p+r+1},t_1,\cdots,t_{p+r}\}$. Note that the first $p+r+1$ points in this sequence are pairwise distinct. Now, if a point appears more than once in the sequence, we keep its first occurrence and truncate the ordered sequence by deleting the later occurrences. Applying this procedure to all repeating points, we obtain an ordered list $\{x_1,\cdots,x_q\}$ of pairwise distinct points on $\Sigma$, where $q\geq p+r+1$.
\end{itemize}

\begin{remark}\label{markings_preserved_rem}
    If the tuples $(\Sigma_1, \cR_1)$, $(\Sigma_2, \cR_2)$ were both associated with the conformal mating $S$, that is, they satisfied Conditions $(1)$-$(4)$ of Proposition~\ref{uniqeness_of_mero_prop}; then $M$ would take the marked points on $\Sigma_1$ to the marked points on $\Sigma_2$ bijectively, as an ordered set. 
\end{remark}

\subsubsection{An invariant graph for the markings}\label{marking_graphs_subsec}

Recall that there is a total number of $q$ marked points on $\Sigma$. Henceforth, all our blender surfaces would be marked according to the convention described above. So for the groups $\{\Gamma_i\}_{i=1}^l$, and the Blaschke products $\{B_j\}_{j=1}^r$, we now have an associated marked blender surface $(\Sigma, x_1, x_2, \cdots, x_q)$, and a uniformizing meromorphic map $\cR$. 

Given such a marked blender surface $\Sigma$, consider the graph $\mathscr{G}$ with vertex set $\{v_1, v_2, \cdots, v_q\}$, where each vertex $v_i$ corresponds to the marked point $x_i$. We place an edge between two vertices $v_{i_1}$ and $v_{i_2}$ if and only if:
\begin{itemize}[leftmargin=4mm]
    \item there is a welding line connecting $x_{i_1}$ and $x_{i_2}$, or
    \item there exists $j\in\{1,\cdots,r\}$ such that, possibly after swapping the indices $i_1$ and $i_2$, we have $\cR(x_{i_1})\in\cV_j$ and $\cR(x_{i_2})=\mathfrak{X}_{l+j}(1)=\mathfrak{X}_P(\xi_{l+j})\in\partial\cV_j$.
\end{itemize}

Note that for $i\in\{1,\cdots,p+r\}$, the vertex $v_i$ has valence $1$ in the graph $\mathscr{G}$. Further, by Corollary~\ref{equiv_limit_sets_cor}, $\mathscr{G}$ would remain the same for marked blender surfaces associated with any collection of groups $\Gamma_i\in\text{Teich}^{\omega_i}(\Gamma_{n_i,p_i})$, $i\in\{1,\cdots,l\}$, and Blaschke products $B_j\in\mathcal{B}_{d_j}$, $j\in\{1,\cdots,r\}$. 

\subsection{Equivalence classes of correspondences as a Hurwitz space}\label{equivalence_of_corr_subsec}
Throughout this section, we denote by $\Sigma$ (or, $\Sigma_i$, $\Sigma'$), a hyperelliptic Riemann surface of genus $g$ with $q$ marked points $\{x_i\}_{i=1}^q$. Typically, the term `hyperelliptic' refers to compact Riemann surfaces of genus $g \geq 2$ that admit a degree 2 branched cover to the Riemann sphere; or equivalently, compact Riemann surfaces of genus $g \geq 2$ admitting a conformal involution with $2g+2$ fixed points. However, for our purposes, we shall extend this terminology to include all compact Riemann surfaces that admit such a cover (equivalently, an involution), regardless of genus. 

It is a classical fact that for Riemann surfaces of genus $g = 0$ and $g = 1$, a hyperelliptic involution $\eta$ can, up to conjugation by a conformal automorphism, be taken to be $\eta(z) = 1/z$ and $\eta(z) = -z$, respectively. Moreover, for surfaces of genus $g \geq 2$, the hyperelliptic involution, when it exists, is unique and central in the automorphism group $\text{Aut}(\Sigma)$ \cite[III.7.9]{farkas1980riemann}. Accordingly, in what follows, by a hyperelliptic involution $\eta$ of a surface $\Sigma$, we would mean $\eta(z) = 1/z$ for $g = 0$, $\eta(z) = -z$ for $g = 1$, and the unique hyperelliptic involution in $\text{Aut}(\Sigma)$ for $g \geq 2$.

With these conventions set, we now consider holomorphic correspondences $\mathfrak{C}\subset \Sigma\times\Sigma$ of the form,
$$
(z,w)\in \mathfrak{C}\iff \frac{\cR(w)-\cR(\eta(z))}{w- \eta(z)}=0,
$$
where $\cR$ is a meromorphic map on $\Sigma$ of degree $d$, and $\eta$ the hyperelliptic involution of $\Sigma$. Note that the tuples $(\Sigma, \cR)$ and $(\Sigma, M\circ \cR)$, for $M\in PSL_2(\C)$, generate the same correspondence $\mathfrak{C}$. Thus, it is natural to identify $\mathfrak{C}$ with an equivalence class of marked tuples $[(\Sigma, \cR)]$, where $(\Sigma, \cR)\sim(\Sigma, \cR')$, if there exists a M\"obius map $M\in PSL_2(\C)$, such that $\cR = M\circ \cR'$.

In light of Proposition~\ref{uniqeness_of_mero_prop}, we say that the correspondences $\mathfrak{C}_1$ and $\mathfrak{C}_2$, defined by the classes of marked tuples $[(\Sigma_1, \cR_1)]$ and $[(\Sigma_2, \cR_2)]$, respectively, are equivalent if there exists a biholomorphism $M: \Sigma_1 \to \Sigma_2$, such that $\cR_1= \cR_2\circ M$ and $\eta_2\circ M = M \circ \eta_1$. Additionally, we require $M$ to preserve the markings as an ordered set.

We are thus led to define the marked space of equivalence classes of correspondences by the relation $\sim_c$, where $(\Sigma, \cR, x_1, x_2, \cdots, x_q) \sim_c (\Sigma', \cR', x_1', x_2'\cdots, x_q')$, if and only if, 
\begin{equation}\label{marked_corr_eq}
    \cR = M_2 \circ \cR'\circ M_1,\ \text{and }\ \eta'\circ M_1 = M_1\circ\eta,
\end{equation}
where, as before, $M_1:\Sigma\rightarrow\Sigma'$ is a biholomorphism that preserves the ordered set of marked points, $M_2\in PSL_2(\C)$, and $\eta$, $\eta'$ are the hyperelliptic involutions of $\Sigma$, $\Sigma'$, respectively. Each equivalence class corresponds to a point in the restricted Hurwitz space $\mathcal{H}_{g, d, q}^{\text{hyp}}$, defined as, 
\begin{align*}
\mathcal{H}_{g,d, q}^{\text{hyp}} := \{f:X \to \C \hspace{1mm}|\ 
X, \text{ a hyperelliptic genus }g \text{ surface } \text{with } q \\  \text{ ordered marked points}, \text{ deg}(f)= d\}/\sim,
\end{align*}
where $\sim$ denotes equivalence up to isomorphism of the domain and post-composition by Möbius transformations on the target; preserving the markings (as in Equation~\ref{marked_corr_eq}). Let us denote by $\mathcal{M}$, the following parameter space,

$$
\mathcal{M}:= \Bigg(\prod_{i=1}^l\text{Teich}^{\omega_i}(\Gamma_{n_i,p_i})\Bigg)\times \Bigg(\prod_{j=1}^r\mathcal{B}_{d_j}\Bigg).
$$
For each $t\in \mathcal{M}$, we get a conformal mating $S_t:\overline{\mathcal{D}}_t\rightarrow\widehat{\C}$, that is unique upto a M\"obius conjugation (see Proposition~\ref{uniqueness_standard_mating_prop}). Associated with the map $S_t$, is the tuple
$$
(\Sigma_t, \cR_t, x_1(t), x_2(t),\cdots, x_q(t)),
$$ 
of a marked blender surface of genus $g$, and a uniformizing meromorphic map of degree $d$. By Proposition~\ref{mating_surf_exists_prop}, $\Sigma_t$ is connected, and hence hyperelliptic. Observe that for the map $M\circ S_t\circ M^{-1}$, with $M\in PSL_2({\C})$, the associated marked pair would be 
$$
(\Sigma_t, M\circ\cR_t, x_1(t), x_2(t),\cdots, x_q(t)).
$$ 
Thus, by Proposition~\ref{uniqeness_of_mero_prop} and Remark~\ref{markings_preserved_rem}, we have a well-defined map
$$
\mathfrak{B}: \mathcal{M} \longrightarrow \mathcal{H}_{g, d, q}^{\text{hyp}}
$$
$$
t \mapsto [\Sigma_t, \cR_t, x_1(t), x_2(t),\cdots, x_q(t)].
$$
Let us denote by $\mathscr{C}$ the image of $\mathfrak{B}$ in the restricted Hurwitz space $\mathcal{H}_{g, d, q}^{\text{hyp}}$, and call it the \emph{moduli space of correspondences} associated with $\mathcal{M}$.

\subsection{The inverse map}\label{injectivity_subsec}
In this subsection, we construct an explicit inverse to the map $\mathfrak{B}$ on the moduli space $\mathscr{C}$ of correspondences associated with $\mathcal{M}$. Let $[\mathfrak{C}] = [\Sigma, \cR, x_1, \cdots, x_q] \in \mathscr{C}$. Recall from Theorem~\ref{corr_recover_thm} that the dynamical structure of $\mathfrak{C}$ emerged primarily via the partition of $\Sigma$ induced by the conformal mating $S:\overline{\cD}\rightarrow\widehat{\C}$. However, the conformal mating $S$ is no longer assumed to be available. Our goal is to recover the full dynamical structure of the correspondence $\mathfrak{C}$, defined by the pair $(\Sigma, \cR)$, directly from the data of $\cR$ and the markings $\{x_i\}_{i=1}^{q}$. Getting to the inverse of $\mathfrak{B}$ amounts to achieving this recovery. 

By construction of $\mathscr{C}$, there exist groups $\{\Gamma_i\}_{i=1}^l$, and maps $\{B_j\}_{j=1}^r$, such that
$$
\mathfrak{B}(\Gamma_1,\cdots,\Gamma_l, B_1,\cdots, B_r) = [\mathfrak{C}].
$$
Let $S$ be a standard conformal mating associated with the collections of groups $\{\Gamma_i\}_{i=1}^l$ and maps $\{B_j\}_{j=1}^r$, equipped with conformal conjugacies $\{\mathfrak{X}_i\}_{i=1}^{l+r}$, and $\mathfrak{X}_P$ (see the proof of Theorem~\ref{conf_mating_thm}). Recall that the map $S$ combines the dynamics of the said collections with a critically fixed polynomial $P$ (as outlined in the beginning of Section~\ref{Hurwitz_sec}). From Equation~\ref{marked_corr_eq} we see that, $\cR(x_{p+r+1})\in \mathfrak{X}_P(z_0)$, where $z_0 \in \mathcal{J}(P)$.

Since the iterated pre-images of $z_0$ is dense in $\mathcal{J}(P)$, the iterated pre-images of $\mathfrak{X}_P(z_0)$ under the map $S$, would be dense in $\partial\cK$ (see Remark~\ref{invariant_subset_rem}). As a result, the closure of the grand orbit of $x_{p+r+1}$ under the correspondence $\mathfrak{C}$, denoted by $\widetilde{\Lambda}$ (the limit set of the correspondence $\mathfrak{C}$), would partition $\Sigma$ into countably many topological disks. 

For $1\leq i\leq l$, the components in $\Sigma\setminus\widetilde{\Lambda}$ containing the first $p$ points in the ordered set $\{x_1,\cdots,x_q\}$, give us the open sets $\widetilde{\cT}_i$. At least one of the endpoints of the welding line in $\dT_{i,s}$ can then be recovered from the last $q-p-r$ points (using the graph $\mathscr{G}$); and the welding line itself is the unique bi-infinite geodesic in the hyperbolic metric of $\dT_{i,s}$ connecting the points $x_{i_1}$ and $x_{i_2}$ if the edge $(v_{i_1}, v_{i_2})\in \mathscr{G}$. The union of these welding lines allows us to reconstruct the boundary $\partial\mathfrak{D}$. Further, the univalence of the map $\mathcal{R}$ on one of the components of $\Sigma \setminus \partial\mathfrak{D}$ leads to the recovery of the domain $\mathfrak{D}$ itself.
 
Set $\mathfrak{C}_i:= \mathfrak{C}\cap (\widetilde{\cT}_i\times \widetilde{\cT}_i)$ (cf. Equation~\ref{restricted_corr_eq}). Note that the action of $\mathfrak{C}_i$ on the welding lines in $\widetilde{\cT}_i$ tessellates $\dT_{i,1}$. Further, the welding line in $\dT_{i,1}$ belongs to the boundary of exactly two tiles, one is contained in $\overline{\mathfrak{D}}$, and the other, a rank $0$ tile, is in the complement $\Sigma\setminus\mathfrak{D}$. Let us denote by $\widetilde{\mathfrak{S}}$ the rank $0$ tile in $\dT_{i,1}$. The proofs of Lemma~\ref{aut_corr_lem} and Proposition~\ref{orbifold_isom_prop} yield the following observations.

The map $\cR$ sends $\widetilde{\mathfrak{S}}$ onto the rank $0$ tile of $\cT_i$ (in the dynamical plane of $S$) as an $n_i:1$ branched cover. The conformal automorphism $\tau_i^{p_i}:\dT_{i,1}\to\dT_{i,1}$ can be recovered as a generator of the deck transformation group for $\cR:\dT_{i,1}\to\cT_i$. The map $\tau_i^{p_i}$ preserves the rank $0$-tile $\widetilde{\mathfrak{S}}$, and the set $\mathfrak{S}$ (of Proposition~\ref{orbifold_isom_prop}) can be taken to be a fundamental domain for the $\tau_i^{p_i}$-action on $\widetilde{\mathfrak{S}}$. Further, the multi-valued map $\mathfrak{C}_i$ acts as a group of conformal automorphisms on $\dT_{i,1}$, and this group action is properly discontinuous, admitting $\mathfrak{S}$ as its closed fundamental domain. The quotient $\dT_{i,1}/\mathfrak{C}_i$ is a genus $0$ orbifold, biholomorphic to an element in $\mathcal{F}$ (see Section~\ref{mateable_maps_subsec}). Further, by Remark~\ref{recover_rep_rem}, the side-pairings on $\partial\mathfrak{S}$ induced by the action of $\mathfrak{C}_i$, help us recover the representation $\widehat{\rho}_i \in \text{Teich}^{\omega_i}(\Gamma_{n_i,p_i})$. Therefore, the correspondence $\mathfrak{C}$ determines a unique element,
$$
(\Gamma_1, \Gamma_2, \cdots, \Gamma_l)\in \prod_{i=1}^l\text{Teich}^{\omega_i}(\Gamma_{n_i,p_i}).
$$

Now for a fixed $j\in \{1, 2, \cdots, r\}$, let $\widecheck{\cV}_j$ be the component in $\Sigma\setminus\widetilde{\Lambda}$ that contains the point $x_{p+j}$. Suppose that $v_{k_j}$ is adjacent to $v_{p+j}$ in the graph $\mathscr{G}$. By Proposition~\ref{recover_Blaschke_prop}, the branch
$\Big(\cR\big|_{\widecheck{\cV}_j}\Big)^{-1}\circ\cR\circ\eta:\widecheck{\cV}_j\to\widecheck{\cV}_j$
of the correspondence $\mathfrak{C}$ can be conjugated to a unique Blaschke product $B_j\in\mathcal{B}_{d_j}$, where the normalized conformal conjugacy from $\widecheck{\cV}_j$ to $\D$ sends $x_{p+j}$ to the origin and the homeomorphic boundary extension of the conjugacy sends $x_{k_j}$ to $1$.
As a result, the correspondence $\mathfrak{C}$ also determines a unique element,
$$
(B_1, B_2, \cdots, B_r) \in \prod_{j=1}^r\mathcal{B}_{d_j}.
$$
Therefore, the above procedure gives us a well-defined map
$$
\mathfrak{B}': \mathscr{C}\longrightarrow \mathcal{M},
$$
that is, by construction, the inverse of the map $\mathfrak{B}$. 

\subsection{Unification of Teichm\"uller spaces and Blaschke spaces in Hurwitz spaces}\label{para_unif_subsec}

From the discussions in Sections~\ref{limit_sets_equiv_subsec},~\ref{equivalence_of_corr_subsec} and~\ref{injectivity_subsec}, it follows that the parameter space for correspondences $\mathfrak{C}$ combining the dynamics of groups $\{\widehat{\Gamma}_i\}_{i=1}^l$ and maps $\{B_j\}_{j=1}^r$ is $\mathscr{C} = \mathfrak{B}(\mathcal{M})$, and we have,
$$
\mathscr{C} \cong \left(\text{Teich}(\D/\widehat{\Gamma}_1)\times\cdots\times\text{Teich}(\D/\widehat{\Gamma}_l)\right)\times\left(\mathcal{B}_{d_1}\times\cdots\times\mathcal{B}_{d_r}\right).
$$
We are now in a position to complete the proof of Theorem~\ref{hurwitz_thm_intro}, thus establishing the second main result of this paper. 
\begin{proof}[Proof of Theorem~\ref{hurwitz_thm_intro}]
    The maps $\mathfrak{B}$, $\mathfrak{B}'$, constructed in Sections~\ref{equivalence_of_corr_subsec} and~\ref{injectivity_subsec}, show the existence of an injective map from $\mathcal{M}$ into the Hurwitz space $\mathcal{H}_{g, d, q}^{\text{hyp}}$. 
\end{proof}

\begin{remark}
    Recall that in Example~\ref{111_ex}, each parameter $t \in \text{Teich}(\Gamma^2_{1,4})\times \mathcal{B}_2$ gives rise to a marked tuple $(\Sigma_t, \cR_t)$, where $\Sigma_t$ is an elliptic curve, and $\cR_t:\Sigma_t\rightarrow\widehat{\C}$ is a degree $4$ elliptic function with $6$ simple critical points and a double critical point. Therefore, by Theorem~\ref{hurwitz_thm_intro}, we get an injective map from $\text{Teich}(\Gamma^2_{1,4})\times \mathcal{B}_2$ into the Hurwitz space $\mathcal{H}_{1, 4, 11}^{\text{hyp}}$.
\end{remark}

\subsection*{Conflict of Interest} The authors declare that they have no conflict of interest.


\newcommand{\etalchar}[1]{$^{#1}$}

\end{document}